\documentclass[10pt]{amsart}

\usepackage{packages}

\title[Primal-Dual Goemans--Williamson for Weighted Fractional Cut Covers]{%
  A Primal-Dual Extension of the Goemans--Williamson Algorithm for the
  Weighted Fractional Cut-Covering Problem%
}
\author[N. Benedetto Proença]{Nathan Benedetto Proença\textsuperscript{1\P\textdagger}}
\address{%
  \textsuperscript{1}Department of Combinatorics and Optimization, University of Waterloo
}
\thanks{%
  \textsuperscript{\P}%
  Research of this author was supported in part by a
  Discovery Grant from the Natural Sciences and Engineering
  Research Council (NSERC) of Canada%
}
\author[M.K. de Carli Silva]{Marcel K. de Carli Silva\textsuperscript{2\textasteriskcentered}}
\address{%
  \textsuperscript{2}Institute of Mathematics and Statistics, University of São Paulo%
}
\thanks{%
  \textsuperscript{\textasteriskcentered}%
  This work was partially supported by Conselho Nacional de
  Desenvolvimento Científico e Tecnológico (CNPq).%
}
\author[C.M. Sato]{Cristiane M. Sato\textsuperscript{3\textasteriskcentered}}
\address{%
  \textsuperscript{3}Center for Mathematics, Computing and Cognition, Federal University of the ABC Region%
}
\author[L. Tunçel]{Levent Tunçel\textsuperscript{1\P}}
\thanks{%
  \textsuperscript{\textdagger}%
  Corresponding Author: Nathan Benedetto Proença.
  Affiliation: University of Waterloo.
  E-mail: \texttt{n2benede@uwaterloo.ca}%
}

\date{November 1, 2024}

\begin{document}

\begin{abstract}
  We study a weighted generalization of the fractional cut-covering
  problem, which we relate to the maximum cut problem via antiblocker
  and gauge duality.
  This relationship allows us to introduce a semidefinite programming
  (SDP) relaxation whose solutions may be rounded into fractional cut
  covers by sampling via the random hyperplane technique.
  We then provide a \(1/\GWalpha\)-approximation algorithm for the
  weighted fractional cut-covering problem, where
  \(\GWalpha \approx 0.878\) is the approximation factor of the
  celebrated Goemans\nobreakdash--Williamson algorithm for the maximum
  cut problem.
  Nearly optimal solutions of the SDPs in our duality framework allow
  one to consider instances of the maximum cut and the fractional
  cut-covering problems as primal-dual pairs, where cuts and
  fractional cut covers simultaneously certify each other's
  approximation quality.
  We exploit this relationship to introduce new combinatorial
  certificates for both problems, as well as a randomized
  polynomial-time algorithm for producing such certificates.
  In~particular, we~show how the Goemans--Williamson algorithm
  implicitly approximates a weighted instance of the fractional
  cut-covering problem, and how our new algorithm explicitly
  approximates a weighted instance of the maximum cut problem.
  We conclude by discussing the role played by geometric
  representations of graphs in our results, and by proving our
  algorithms and analyses to be optimal in several aspects.
\end{abstract}

\maketitle

\newlength{\eqboxwidth}
\settowidth{\eqboxwidth}{\(\coloneqq\)}
\newcommand{\eqaligned}{\mathrel{\makebox[\eqboxwidth][r]{=}}}

\section{Introduction}

Let \(G = (V, E)\) be a simple graph.
For every \(S \subseteq V\), the \emph{cut} with \emph{shore}~\(S\) is
the set~\(\delta(S) \subseteq E\) of edges which have precisely one
vertex in \(S\).
For every nonnegative vector \(z \in \Reals_+^E\) indexed by the
edges, the \emph{fractional cut-covering number} of \((G, z)\) is
\begin{equation}
  \label{eq:cut-covering-problem}
  \fcc(G, z)
  \coloneqq
  \min\setst[\Big]{
    \iprodt {\ones} y
  }{
    y \in \Reals_+^{\Powerset{V}},\,
    \sum_{S \subseteq V} y_S \incidvector{\delta(S)} \ge z
  },
\end{equation}
where the power set of~\(V\) is denoted by~\(\Powerset{V}\), the
incidence vector of \(T \subseteq U\) is
\(\incidvector{T} \in \set{0, 1}^U\), and the vector of all-ones
is~\(\ones\).
When \(z\) is integer-valued, the integer solutions
of~\cref{eq:cut-covering-problem} correspond to multisets of cuts
which cover each edge \(e \in E\) at least \(z_e\) times, thus
explaining the name ``fractional cut-covering''.
The unweighted version of this graph parameter --- i.e., \(\fcc(G)
\coloneqq \fcc(G, \ones)\) --- is used by \nameandcite{Samal2006} to
prove non-existence of cut-continuous functions between certain graphs.
Such functions are maps between the edge sets of graphs that arise in
the study of certain graph flow conjectures
\cite{DeVosNesetrilEtAl2007}.
A \emph{fractional cut cover of \((G,z)\)} is a feasible solution
of~\cref{eq:cut-covering-problem}.

For every \(w \in \Reals_+^E\), the maximum weight of a cut
of~\((G,w)\) is
\begin{equation}
  \label{eq:maxcut-problem}
  \mc(G, w) \coloneqq
  \max\setst{\iprodt{w}{\incidvector{\delta(S)}}}{S \subseteq V}.
\end{equation}
As larger cuts intuitively give rise to smaller covers, this suggests
a combinatorial relationship between
\cref{eq:maxcut-problem,eq:cut-covering-problem}.
The problem of computing \(\mc(G, w)\) is known as the \emph{maximum
cut problem}, and it is one of Karp's original NP-hard problems
\cite{Karp1972}.
Goemans and Williamson's approximation
algorithm~\cite{GoemansWilliamson1995} for this problem is one of the
most celebrated applications of semidefinite programming.
We denote by \(\Sym{V}\) the Euclidean space of real symmetric
\(V \times V\) matrices, and by \(\Psd{V} \subseteq \Sym{V}\) the cone
of \emph{positive semidefinite matrices}, i.e., the set of
symmetric matrices with nonnegative eigenvalues.
The \emph{Laplacian of \(G\)} is the linear function \(\Laplacian_G
\colon \Reals^{E} \to \Reals^{V \times V}\) defined by
\begin{equation}
  \Laplacian_G(w)
  \coloneqq
  \sum_{ij \in E} w_{ij}(e_i - e_j)(e_i - e_j)^\transp \in \Reals^{V \times V}
  \mathrlap{
    \qquad
    \text{for each \(w \in \Reals^{E}\)},
  }
\end{equation}
where \(\setst{e_i}{i \in V} \subseteq \set{0,1}^V\) are the canonical
basis vectors.
The \emph{trace inner product} of \(A, B \in \Reals^{V \times V}\)is
\(\iprod{A}{B} \coloneqq \trace(A^\transp B)\).
The linear function
\(\diag \colon \Reals^{V \times V} \to \Reals^{V}\) extracts the
diagonal of a square matrix, and its adjoint
\(\Diag \colon \Reals^V \to \Reals^{V \times V}\) builds a diagonal
matrix from its argument such that \(\Diag(x)_{ii} = x_i\) for every
\(i \in V\).
Write \(X \succeq Y\) or \(Y \preceq X\) for symmetric matrices \(X\)
and~\(Y\) if \(X-Y\) is positive semidefinite.
Goemans and Williamson's approximation algorithm implies that the
optimal value of the semidefinite program (SDP)
\begin{subequations}
  \label{eq:GW-intro-pd}
  \begin{align}
    \label{eq:GW-intro-def}
    \GW(G, w)
    &\coloneqq \max\setst{
      \iprod{\tfrac{1}{4}\Laplacian_G(w)}{Y}
    }{
      Y \in \Psd{V},\,
      \diag(Y) = \ones
    }
    \\
    \label{eq:GW-gaugef-rho}
    &\eqaligned
    \min\setst{
      \rho
    }{
      \rho \in \Reals_{+},\,
      x \in \Reals^V,\,
      \rho \geq \iprodt{\ones}{x},\,
      \Diag(x) \succeq \tfrac{1}{4}\Laplacian_G(w)
    }
  \end{align}
\end{subequations}
satisfies
\begin{equation}
  \label{eq:GW-approximation}
  \GWalpha \GW(G, w)
  \le \mc(G, w)
  \le \GW(G, w)
  \mathrlap{
    \qquad
    \text{for each \(w \in \Reals_+^{E}\)},
  }
\end{equation}
where
\begin{equation}
  \label{eq:GWalpha-def}
  \GWalpha
  \coloneqq \min_{0 < \vartheta \le \pi}
  \frac{2}{\pi}\frac{\vartheta}{1 - \cos \vartheta}
  \approx 0.878
\end{equation}
is the approximation factor.
\Cref{eq:GW-gaugef-rho} follows from SDP Strong Duality, since both
primal and dual SDPs have Slater points.

The \emph{norm} of a vector \(u \in \Reals^d\) is
\(\norm{u} \coloneqq \sqrt{\iprodt{u}{u}}\).
For a fixed real number \(t \ge 1\), a \emph{vector \(t\)-coloring of G} is
a function \(f \colon V \to \Reals^d\) assigning a unit-norm vector
\(f(i) \in \Reals^d\) to each \(i \in V\) such that
\((t - 1)\iprodt{f(i)}{f(j)} \le -1\) for every \(ij \in E\).
Vector colorings were first introduced in
\cite{KargerMotwaniEtAl1998}.
The smallest value \(t\) for which a graph has a vector \(t\)-coloring
is called the \emph{vector chromatic number} of~\(G\), denoted
by~\(\chivec(G)\).
\Citeauthor{Samal2015}~\cite[Theorem~5.2]{Samal2015} defined a map
from fractional cut covers to vector colorings, thus proving that
\(G\) has a vector \(t\)-coloring such that
\begin{equation}
  \label{eq:fcc-vector-coloring-lowerbound}
  2\paren[\bigg]{1 - \frac{1}{t}}
  \le \fcc(G).
\end{equation}
Neto and Ben-Ameur \cite[Proposition~17]{NetoBen-Ameur2019} tightened
the relationship between fractional cut covers and vector colorings
by showing that
\begin{equation}
  \label{eq:fcc-vector-coloring-upperbound}
  \fcc(G) \le \frac{\pi}{\arccos(1/(1 - t))}
\end{equation}
for every vector \(t\)-coloring of \(G\) such that \(t > 1\).
Assume \(E \neq \emptyset\), so \(t \ge 2\), and set \(\zeta \coloneqq
1/(1 - t)\), so \(\zeta \in [-1, 0)\).
Then~\cref{eq:GWalpha-def} implies that
\begin{equation}
  \label{eq:GWalpha-reciprocal-intro}
  \frac{\pi}{\arccos(1/(1 - t))}
  = \frac{\pi}{\arccos(\zeta)}
  = \frac{\pi}{2}
  \frac{1 - \zeta}{\arccos(\zeta)}
  \frac{2}{1 - \zeta}
  \le \frac{1}{\GWalpha} \frac{2}{1 - \zeta}
  = \frac{1}{\GWalpha}2\paren[\bigg]{1 - \frac{1}{t}}.
\end{equation}
Putting it all together, \cite{NetoBen-Ameur2019} combines
\cref{eq:fcc-vector-coloring-lowerbound},
\cref{eq:fcc-vector-coloring-upperbound},
\cref{eq:GWalpha-reciprocal-intro}, and monotonicity of \(x \mapsto (1
- 1/x)\) to conclude that
\begin{equation}
  \label{eq:vector-coloring-approximates-fcc}
  2\paren[\bigg]{
    1 - \frac{1}{\chivec(G)}
  }
  \le \fcc(G)
  \le \frac{1}{\GWalpha}
  2\paren[\bigg]{
    1 - \frac{1}{\chivec(G)}
  }.
\end{equation}
As \cite[Corollary~4]{NetoBen-Ameur2019} points out, the inequalities
in~\cref{eq:vector-coloring-approximates-fcc} provide a
polynomial-time computable approximation for the unweighted \emph{number}
\(\fcc(G)\), since \(\chivec(G)\) is the optimal value of an SDP which
can be approximated to any given precision in polynomial time.

We invite the reader to compare \cref{eq:GW-approximation}
and~\cref{eq:vector-coloring-approximates-fcc}.
Both describe constant-factor approximations that are computable from
the optimal values of SDPs, and furthermore, both approximation factors
are \(\GWalpha\).
This work exploits and extends the ideas
underlying~\cref{eq:vector-coloring-approximates-fcc}.
For every \(z \in \Reals_+^E\), define
\begin{equation}
  \label{eq:GW-polar-def-intro}
  \GW^{\polar}(G, z)
  \coloneqq \min \setst{
    \mu
  }{
    \mu \in \Reals_+,\,
    Y \in \Psd{V},\,
    \diag(Y) = \mu \ones,\,
    \tfrac{1}{4}\Laplacian_G^*(Y) \ge z
  }
\end{equation}
where \(\Laplacian_G^* \colon \Reals^{V \times V} \to \Reals^E\) is
the adjoint of the Laplacian, i.e.,
\begin{equation}
  \label{eq:Laplacian-adjoint-def}
  \paren[\big]{\Laplacian_G^*(Y)}_{ij}
  =
  Y_{ii} + Y_{jj} - Y_{ij} - Y_{ji}
  \qquad
  \text{
    for each
    \(Y \in \Reals^{V \times V}\) and \(ij \in E\).
  }
\end{equation}
If \(y\) is a fractional cut cover for~\((G,z)\), then
\begin{inlinemath}
  (\mu,Y)
  \coloneqq
  \paren[\big]{
    \iprodt{\ones}{y},
    \sum_{S \subseteq V}
    y_S \oprodsym{(\ones - 2\incidvector{S})}
  }
\end{inlinemath}
is feasible for the SDP~\cref{eq:GW-polar-def-intro}.
If \(f \colon V \to \Reals^d\) is a vector \(t\)-coloring for some
\(t > 1\), we may define \(\mu \coloneqq 2(1-1/t)\) and \(Y
\in \Psd{V}\) by \(Y_{ij} \coloneqq \mu\iprodt{f(i)}{f(j)}\) for every
\(i, j \in V\).
Then \(\diag(Y) = \mu\ones\) and \(\tfrac{1}{4}\Laplacian_G^*(Y) \ge
\ones\), as
\begin{inlinemath}
  \tfrac{1}{4}\paren[\big]{\Laplacian_G^*(Y)}_{ij}
  = \tfrac{1}{2} \mu (1 - \iprodt{f(i)}{f(j)})
  \ge \tfrac{1}{2}\mu\paren[\big]{1 + \tfrac{1}{t-1}}
  = 1
\end{inlinemath}
for every \(ij \in E\).
In this manner, the feasible solutions for~\cref{eq:GW-polar-def-intro}
capture the geometry of vector colorings which
enables~\cref{eq:vector-coloring-approximates-fcc}.
Using \(\GW^{\polar}\),
we~strengthen~\cref{eq:vector-coloring-approximates-fcc} to all
nonnegative weights:
\begin{equation}
  \label{eq:GW-polar-approx-pledge}
  \GW^{\polar}(G, z)
  \le \fcc(G, z)
  \le \frac{1}{\GWalpha} \GW^{\polar}(G, z)
  \mathrlap{
    \qquad
    \text{for each \(z \in \Reals_+^{E}\)}.
  }
\end{equation}
This weighted generalization
of~\cref{eq:vector-coloring-approximates-fcc} stands as the proper
fractional cut-covering analogue to~\cref{eq:GW-approximation} for the
maximum cut problem.

The similarity between~\cref{eq:GW-approximation}
and~\cref{eq:GW-polar-approx-pledge} is the starting point of this
work, whose main contributions include:
\begin{enumerate}[(i)]
\item pinpointing the relationship between the maximum cut
  problem~\cref{eq:maxcut-problem} and the fractional cut-covering
  number~\cref{eq:cut-covering-problem} to gauge and antiblocker
  duality \cite{Fulkerson1971,Fulkerson1972};
\item introducing~\cref{eq:GW-polar-def-intro} as the dual parameter
  to~\cref{eq:GW-intro-def}, immediately obtaining
  that~\cref{eq:GW-polar-approx-pledge} is equivalent
  to~\cref{eq:GW-approximation} via a precisely defined bound
  conversion procedure \cite[Sections~6
  and~7]{BenedettoProencadeCarliSilvaEtAl2021};
\item describing a randomized approximation algorithm, dual to the
  Goemans--Williamson algorithm, which rounds any nearly optimal
  solution of the SDP~\cref{eq:GW-polar-def-intro} to a
  \((1/\GWalpha)\)-approximately optimal fractional cut cover of
  \((G,z)\) with very sparse support;
\item pairing instances of the maximum cut and fractional cut-covering
  problems so that one can obtain approximately optimal solutions for
  \emph{both} instances by solving a \emph{single}~SDP, and so that
  their approximate optimality can be certified by a simultaneous,
  (mostly) combinatorial certificate;
\item showing our algorithms to be best possible in several aspects;
\item clarifying the role played by geometric representation of graphs
  in the aforementioned results.
\end{enumerate}
Our algorithms run in polynomial time in the real-number machine model
(see \cite{BlumCuckerEtAl1998}) with access to two additional oracles:
one computing Cholesky factorizations and one sampling from a standard
normal distribution.
These assumptions streamline our arguments while still building
towards a strongly polynomial-time implementation on a probabilistic
Turing machine.
The access to a Cholesky factorization oracle amounts to assuming
exact square root computation.
For our purposes, efficient algorithms that lead to rational
approximations of the square-root are sufficient, since the
probabilistic nature of our algorithms and the slacks in our analyses
make our algorithms and analyses robust to small enough precision
errors.
In particular, our situation is different than assuming sum of
square-roots problem can be solved in polynomial time.
(For related complexity issues, see, for instance,
\cite{AllenderBurgisserKjeldgaardMiltersen2008} and references
therein.)
The access to an oracle sampling numbers from a standard normal
distribution encapsulates a yet subtler issue.
As even the representation of continuously supported random variables
on Turing machines poses a nontrivial question, our oracle assumption
decouples our analyses from implementation details that are beyond the
scope of this paper.

\subsection{Organization of the Text}

In order to facilitate reading, we unveil these results in increasing
order of abstraction.
We start at \Cref{sec:rounding} by exhibiting a novel randomized
approximation algorithm for the weighted fractional cut covering
problem.
The connection between \(\GW\) and \(\GW^{\polar}\) is the main theme
of \Cref{sec:certificates}.
In \Cref{sec:duality} we express the relationship between both
optimization problems via antiblocker
\cite{Fulkerson1971,Fulkerson1972} and gauge duality
\cite{BenedettoProencadeCarliSilvaEtAl2021}.
We then show how computing either one of the parameters \(\GW\) or
\(\GW^{\polar}\) implicitly computes the other parameter, and how this
can be leveraged to provide simultaneous combinatorial certificates
for the approximate optimality of certain cuts and fractional cut
covers.
The existential results for certificates we~prove in
\Cref{sec:existence-of-certificates} are refined into efficient
algorithms in \Cref{subsec:algorithimic_certicates}.
We recover the role played by vector colorings in this introduction by
relating our approach to geometric representations of graphs in
\Cref{sec:geometric-graphs}.
\Cref{sec:possible-improvements} discusses possible improvements to
our approximation algorithms by collecting noteworthy instances of our
optimization problems: either instances where simpler approaches lead
to degenerate behavior, or instances which show our bounds to be tight.

\subsection{Notation}

For each \(n \in \Naturals\), denote as usual
\([n] \coloneqq \set{1,\dotsc,n}\).
The set of nonnegative real numbers is denoted by \(\Reals_+\), and
the set of positive real numbers is~\(\Reals_{++}\).
Let \(U\) be a finite set.
We~denote by \(\Reals^U\) the real vector space indexed by entries in
\(U\).
For each \(i \in U\), we denote by \(e_i \in \set{0,1}^U\) the
\(i\)-th canonical basis vector.
The \(1\)-norm of a vector \(z \in \Reals^{U}\) is
\(\norm[1]{z} \coloneqq \sum_{i \in U} \abs{z_i}\), the
\(\infty\)-norm of~\(z\) is
\(\norm[\infty]{z} \coloneqq \max\setst{\abs{z_i}}{i \in U}\), and the
support of~\(z\) is
\(\supp(z) \coloneqq \setst{i \in U}{z_i \neq 0}\).

\section{A Randomized Rounding Algorithm for Weighted Fractional Cut Covering}
\label{sec:rounding}

Let \(V\) be a finite set.
The set
\begin{equation*}
  \elliptope{V} \coloneqq \setst{Y \in \Psd{V}}{\diag(Y) = \ones}
\end{equation*}
is commonly referred to as the \emph{elliptope}.
We adopt (extended) Minkowski set operations and write
\(\mu \elliptope{V} \coloneqq \setst{Y \in \Psd{V}}{\diag(Y) =
  \mu\ones}\) for every \(\mu \in \Reals\), and
\(R \elliptope{V} \coloneqq \setst{\mu Y}{\mu\in R,\ Y \in
  \elliptope{V}}\) for each \(R \subseteq \Reals\).
For every \(Y \in \Reals_+\elliptope{V}\) and nonzero
\(h \in \Reals^V\), define
\begin{equation}
\label{eq:GWrv-def}
  \GWrv(Y, h) \coloneqq \setst{i \in V}{\iprodt{e_i}{Y^{\half}h} \ge 0},
\end{equation}
where \(Y^{\half}\) is the unique positive semidefinite square root
of~\(Y\).
\Cref{eq:GWrv-def} describes a possible implementation of the random
hyperplane technique used by \nameandcite{GoemansWilliamson1995} to
sample a shore of a~cut.
Let \((\Omega, \Sigma, \prob)\) be a probability space and let the
random variable \(h \colon \Omega \to \Reals^V\) be a uniformly
distributed unit vector.
For every \(Y \in \Reals_{++}\elliptope{V}\), we~denote by
\(\GWrv(Y) \colon \Omega \to \Powerset{V}\) the random variable given
by
\begin{equation*}
  \GWrv(Y) \colon \omega \in \Omega \mapsto \GWrv(Y, h(\omega))
  \subseteq V,
\end{equation*}
which samples shores.
It is proved in~\cite{GoemansWilliamson1995} that
\begin{equation}
  \label{eq:GW-Y-edge-marginal}
  \prob\paren[\big]{
    ij \in \delta(\GWrv(Y))
  } = \frac{\arccos(Y_{ij})}{\pi}
  \qquad
  \text{
    for every \(ij \in E\) and \(Y \in \elliptope{V}\).
  }
\end{equation}
\Cref{alg:1} leverages the connection between the elliptope and
probability distributions on \(\Powerset{V}\) given by~\(\GWrv\).
We will exploit semidefinite programming to produce a matrix
\(Y \in \Reals_{++}\elliptope{V}\) from which we sample a fractional
cut cover in (randomized) polynomial time by repeated sampling from
\(\GWrv(Y)\).
This section is devoted to proving correctness of \Cref{alg:1}:
namely, that it produces an approximately optimal fractional cut cover
with high probability in polynomial time.

The algorithm works roughly as follows.
First it creates a new weight vector~\(\hat{z}\) from the input weight
vector \(z \in \Reals_+^{E}\) by rounding up the edge weights that
are too small relative to \(\norm[\infty]{z}\).
Then it obtains a nearly optimal solution \((\mu, Y)\) for the
SDP relaxation \cref{eq:GW-polar-def-intro}, which is used to sample
cuts \(\delta(\GWrv(Y))\).
The rounding~up of the weights ensures that every edge has a
significant probability of being in the random cut
\(\delta(\GWrv(Y))\).
The~algorithm builds a fractional cut cover
by using this sampling procedure \(T\) times independently, obtaining a
vector of support size \(T\), which is then scaled.
The parameter \(T\) is defined in the range \(\Theta(\ln(\card{V}))\)
so that it is large enough to guarantee that we obtain a fractional
cut cover with high probability.
Our pseudocode abstracts away the important work of carefully choosing
data structures: in particular, one needs to exploit the sparse nature
of the fractional cut cover produced in its representation.

\begin{algorithm}
  \caption{SDP-based randomized approximation algorithm for fcc}
  \label{alg:1}
  \begin{algorithmic}[1]
    \algrenewcommand\algorithmicrequire{\textbf{Parameters:}}
    \Require a constant approximation factor \(\beta \in (0, \GWalpha)\)
    parameterizes the algorithm \textsc{ApproxFcc${}_{\beta}$}.
    As~in~\cref{eq:5,eq:3}, define the following constants in terms of
    \(\beta\):
    \begin{equation*}
      \tau \coloneqq 1 - \frac{\beta}{\GWalpha} \in (0,1),
      \qquad
      \sigma \coloneqq
      \eps \coloneqq
      \Chernoff \coloneqq
      \tfrac{\tau}{3}
      \in (0,\tfrac{1}{3}),
      \qquad
      \text{and}
      \qquad
      C
      \coloneqq
      {81\sqrt{2}\pi}/{\tau^{5/2}}.
    \end{equation*}
    \algrenewcommand\algorithmicrequire{\textbf{Input:}}
    \Require a graph $G = (V,E)$ and edge weights $z \in \Reals_+^E$
    \algrenewcommand\algorithmicensure{\textbf{Output:}} \Ensure
    \Call{ApproxFcc${}_{\beta}$}{$G, z$} returns a fractional cut cover of
    \((G,z)\) with high probability with objective value bounded above by
    \(\tfrac{1}{\beta}\fcc(G, z)\) and support size bounded above by
    \(T \coloneqq \ceil{C\ln(\card{V})}\), as in~\cref{eq:3}
    \Procedure{ApproxFcc${}_{\beta}$}{$G, z$}
    \State $(V, E) \gets G$
    \State\label{step:preconditioning}%
    $\hat{z}_{e} \gets $
    \Call{max}{$z_{e}, \tfrac{1}{2}\eps\norm[\infty]{z}$}
    \textbf{for each} $e \in E$
    \Comment Round up \(z\)
    \State\label{step:sdp-approximate-solving}%
    Find a feasible $(\mu, Y)$ for $\GW^{\polar}(G, \hat{z})$ in
    \cref{eq:GW-polar-def-intro} with objective value
    $\mu \leq \GW^{\polar}(G, \hat{z}) +\sigma
    \norm[\infty]{z}$ \State $y \gets 0 \in \Reals_+^{\Powerset{V}}$
    \CountUp{$T$}
    \State\label{step:GW-sampling}%
    $S \gets \GWrv(Y)$
    \Comment Sample a shore \(S \subseteq V\) via the random hyperplane technique
    \State $y_{S} \gets y_S + 1$
    \EndCountUp
    \State \textbf{return} $\dfrac{\mu}{(1-\Chernoff)\GWalpha} \dfrac{1}{T} y$
    \EndProcedure
  \end{algorithmic}
\end{algorithm}

Neto and Ben-Ameur \cite[Section~4]{NetoBen-Ameur2019} already scale
the probability distribution on \(\Powerset{V}\) given by \(\GWrv(Z)\)
to define a fractional cut cover from some \(Z \in \Psd{V}\) arising
from a vector coloring.
Our novel formulation~\cref{eq:GW-polar-def-intro} allow us to place
this construction in the weighted setting with  \Cref{prop:sdp-rounding}.

\begin{proposition}
\label{prop:sdp-rounding}
Let \(G = (V, E)\) be a graph and let \(z \in \Reals_+^E\).
Let \((\mu,Y)\) be feasible for~\cref{eq:GW-polar-def-intro}.
Set \(y \in \Reals^{\Powerset{V}}\) by
\begin{equation}
  \label{eq:6}
  y_S \coloneqq
  \dfrac{\mu}{\GWalpha} \prob\paren{\GWrv(Y) = S}
  \mathrlap{
    \qquad
    \text{for every }S \subseteq V.
  }
\end{equation}
Then \(y\) is a fractional cut cover for \((G, z)\) with objective
value \(\iprodt{\ones}{y} = \tfrac{1}{\GWalpha}\mu\).
In particular,
\begin{equation}
  \label{eq:sdp-rounding}
  z_{ij} \le \dfrac{\mu}{\GWalpha} \prob(ij \in \delta(\GWrv(Y)))
  \mathrlap{
    \qquad
    \text{for every \(ij \in E\).}
  }
\end{equation}
\end{proposition}
\begin{proof}
We may assume that \(\mu > 0\).
Set \(\bar{Y} \coloneqq \mu^{-1}Y \in \elliptope{V}\).
By the definition in~\cref{eq:GWrv-def}, we have that \(\GWrv(Y,h) =
\GWrv(\bar{Y},h)\) for every $h\in\Reals^V$, which implies that
$\GWrv(Y) = \GWrv(\bar{Y})$.
Let \(ij \in E\).
Since  \(\bar{Y} \in \elliptope{V}\), we have that
\(\bar{Y}_{ij} \in [-1, 1]\).
If \(z_{ij} = 0\), then~\cref{eq:sdp-rounding} holds trivially.
Assume that \(z_{ij} > 0\).
Since \(\diag(\bar{Y}) = \ones\), we~get
from~\cref{eq:Laplacian-adjoint-def} that
\begin{equation*}
  0
  < z_{ij}
  \le \tfrac{1}{4} \paren[\big]{\Laplacian_G^*(Y)}_{ij}
  = \mu \tfrac{1}{4} \paren[\big]{\Laplacian_G^*(\bar{Y})}_{ij}
  = \mu \tfrac{1}{2}\paren{1 - \bar{Y}_{ij}}.
\end{equation*}
Thus \(\bar{Y}_{ij} < 1\), so \(\arccos(\bar{Y}_{ij}) > 0\), and
\(\prob\paren[\big]{ij \in \delta(\GWrv(\bar{Y}))} > 0\)
by~\cref{eq:GW-Y-edge-marginal}.
Hence
\begin{detailedproof}
  z_{ij}
  &\le \mu \tfrac{1}{2}\paren{1 - \bar{Y}_{ij}}
  = \mu
  \frac{\pi}{2}\frac{1 - \bar{Y}_{ij}}{\arccos(\bar{Y}_{ij})}
  \prob\paren[\big]{ij \in \delta(\GWrv(\bar{Y}))}
  &&\text{by~\cref{eq:GW-Y-edge-marginal}}\\*[5pt]
  &\le \mu\tfrac{1}{\GWalpha} \prob\paren[\big]{ij \in
    \delta(\GWrv(\bar{Y}))}
  &&\text{by~\cref{eq:GWalpha-def}}\\*[5pt]
  &= \sum_{{\substack{S \subseteq V\colon\\[1pt]ij \in \delta(S)}}} \dfrac{\mu}{\GWalpha}
  \prob({\GWrv(\bar{Y}) = S})
  = \sum_{\mathclap{\substack{S \subseteq V\colon\\[1pt]ij \in \delta(S)}}} y_S
  = \paren[\Big]{\sum_{S \subseteq V}{y_S \incidvector{\delta(S)}}}_{ij}.
\end{detailedproof}
As this holds for every \(ij \in E\) for which \(z_{ij} > 0\), we
conclude that \(y\) is a fractional cut cover for \((G, z)\) with
objective value \(\iprodt{\ones}{y} = \tfrac{1}{\GWalpha}\mu\),
and~\cref{eq:sdp-rounding} holds.
\end{proof}

\begin{corollary}
  \label{cor:1}
  Let \(G = (V, E)\) be a graph.
  Then
  \begin{equation*}
    \tag{\ref{eq:GW-polar-approx-pledge}}
    \GW^{\polar}(G, z)
    \le \fcc(G, z)
    \le \frac{1}{\GWalpha} \GW^{\polar}(G, z)
    \mathrlap{
      \qquad
      \text{for each \(z \in \Reals_+^{E}\)}.
    }
  \end{equation*}
\end{corollary}
\begin{proof}
  As mentioned after~\cref{eq:Laplacian-adjoint-def},
  the first
  inequality holds since
  \begin{inlinemath}
    (\mu,Y)
    \coloneqq
    \paren[\big]{
      \iprodt{\ones}{y},
      \sum_{S \subseteq V}
      y_S \oprod{(\ones - 2\incidvector{S})}{(\ones - 2\incidvector{S})}
    }
  \end{inlinemath}
  is feasible for the SDP~\cref{eq:GW-polar-def-intro}, for every
  fractional cut cover \(y\) for \((G, z)\).
  The second inequality in~\cref{eq:GW-polar-approx-pledge} follows
  directly from \cref{prop:sdp-rounding}.
\end{proof}

\begin{remark}
\label{re:need-to-randomize}
We will define in \cref{eq:Chernoff-sampling-def} below the random
variable capturing the repeated sampling behavior of \Cref{alg:1}.
One might question the purpose of employing a randomized approach
when \Cref{prop:sdp-rounding} defines a precise solution with a
guaranteed approximation factor to the optimal value.
We provide two compelling reasons: in \cref{prop:sdp-rounding},
\begin{itemize}
\item[(i)] it is a challenging task to compute the probabilities
  \(\prob\paren{\GWrv(Y) = S}\) in~\cref{eq:6}: while the marginal
  probability \(\prob\paren{ij \in \delta(\GWrv(Y))}\) is determined
  by~\cref{eq:GW-Y-edge-marginal}, computing
  \(\prob\paren{\GWrv(Y) = S}\) requires joint probabilities for all
  the vertices in the shore~\(S\);

\item[(ii)] the fractional cut cover \(y\) obtained may have
  exponential support size: indeed there are instances (e.g., the
  complete graph with uniform weights) for which there exist optimal
  solutions for~\cref{eq:GW-polar-def-intro} that result in a vector
  \(y\) given as in~\cref{eq:6} with exponential support size; see
  \cref{sec:exp-supp}.
\end{itemize}
The randomized procedure helps to address these two issues.
\end{remark}

As \cite{NetoBen-Ameur2019} is only concerned with approximating
the value \(\fcc(G)\), as opposed to computing a fractional
cut cover, the randomized algorithms in their Section~4 do not
address the issues in \cref{re:need-to-randomize}.
One may interpret their results as using \cref{prop:sdp-rounding}
to say that feasible solutions to \cref{eq:GW-polar-def-intro} are
an \emph{implicit} representation of a fractional cut cover.

\Cref{prop:sdp-rounding} and Carathéodory's theorem imply that, given
a matrix~\(Y\) and a real number~\(\mu\) feasible
for~\cref{eq:GW-polar-def-intro}, there exists a fractional cut cover
with polynomial support size and objective value bounded above by
\((1/\GWalpha)\mu\).
The randomized approach below will in fact produce a feasible solution
whose support size is \(O(\ln(n))\), a reduction of \emph{two} orders
of magnitude compared to the fractional cut cover from
\cref{prop:sdp-rounding}.

Let \(T \in \Naturals \setminus \set{0}\) and let \(\Chernoff \in (0,1)\).
Let \(G = (V, E)\) be a graph, and let \(Y \in \Sym{V}\) and \(\mu \in
\Reals_+\) be such that \(Y \in \mu \elliptope{V}\).
Let \(S_1, \dotsc, S_T \subseteq V\) be independent, identically
distributed random shores sampled by~\(\GWrv(Y)\).
Define
\begin{equation}
  \label{eq:Chernoff-sampling-def}
  \Acal_{T, \Chernoff}(G, Y) \coloneqq (\Fcal, y),
  \text{ where }
  \Fcal \coloneqq \set{S_1, \dotsc, S_T} \subseteq \Powerset{V}
  \text{ and }
  y \coloneqq
  \frac{\mu}{(1 - \Chernoff) \GWalpha} \frac{1}{T} \sum_{t = 1}^T
  e_{S_t} \in \Reals_+^{\Fcal}.
\end{equation}
Informally, \(\Acal_{T, \Chernoff}(G, Y)\) produces a (scaled version of
a) sparse surrogate for the probability distribution of~\(\GWrv(Y)\).
It is immediate that
\begin{equation}
  \label{eq:Chernoff-sampling-properties}
  \iprodt{\ones}{y} = \frac{1}{\GWalpha} \frac{1}{1 - \Chernoff}\mu
  \qquad
  \text{and}
  \qquad
  \card{\Fcal} \leq T.
\end{equation}
The parameter \(\Chernoff\) regulates the deviation from the objective
value obtained in  \Cref{prop:sdp-rounding}, while \(T\) needs to be
chosen large enough so that concentration results imply that the
desired level of accuracy is achieved with high probability.

Let \(G = (V, E)\) be a graph.
\v{S}ámal \cite[Theorem~5.2]{Samal2015} uses Chernoff's bound to show
that, by sampling sufficiently many cuts, one can obtain a fractional
cut cover for \((G, \ones)\).
\Cref{prop:probabilistic-rounding-cover-pre} improves on this work, by
providing an explicit bound on the actual number of cuts that suffices
in the general weighted setting.
We~will use that, for every \(x \in [-1, 1]\) and \(y \in [0, 2]\),
\begin{equation}
  \label{eq:arccos-bound}
  \text{if }
  x \le 1 - y
  \text{ then }
  \arccos(x) \ge \sqrt{2y}.
\end{equation}
This follows from monotonicity of \(\arccos\), and the inequality
\(\arccos(1 - y) \ge \sqrt{2y}\) for each \(y \in [0, 2]\).

\begin{proposition}
\label{prop:probabilistic-rounding-cover-pre}
Let \(\xi, \kappa, \Chernoff \in \Reals\) be such that
\(0 < \xi \le 1 \le \kappa\) and \(0 < \Chernoff < 1\).
Let \(G = (V, E)\) be a graph on \(n\) vertices, and let
\(z \in \Reals_+^E\) be nonzero.
Let \((\mu,Y)\) be feasible in~\cref{eq:GW-polar-def-intro}.
Let
\begin{inlinemath}
  T
  \geq
  \ceil[\big]{
    3\pi
    \paren{
      \frac{\kappa}{\xi}
    }^{\half}
    \frac{1}{\Chernoff^2} \ln(n)
  }
\end{inlinemath}
be an integer and set
\((\Fcal, y) \coloneqq \Acal_{T, \Chernoff}(G, Y)\).
If
\begin{subequations}
  \begin{gather}
    \label{eq:prop-round-pre-2}
    \mu \le \kappa \norm[\infty]{z}
    \\
    \shortintertext{and}
    \label{eq:prop-round-pre-3}
    z \ge \xi \norm[\infty]{z}\ones,
  \end{gather}
\end{subequations}
then
\begin{equation}
\label{eq:prop-round-y}
  \prob\paren[\bigg]{
    \sum_{S \in \Fcal} y_S \incidvector{\delta(S)} \geq z
  } \ge 1 - \frac{1}{n}.
\end{equation}
\end{proposition}
\begin{proof}
Let \(S_1,\ldots, S_T\) be random shores defined as
in~\cref{eq:Chernoff-sampling-def}.
For every \(ij \in E\), set
\(X_{ij} \coloneqq \card{\setst{t \in [T]}{ij \in \delta(S_t)}}\) and
\(p_{ij} \coloneqq \prob\paren{ij \in \delta(\GWrv(Y))}\).
By construction, for each \(ij \in E\), the random variable~\(X_{ij}\)
has binomial distribution on \(T\) trials with success probability
\(p_{ij}\) and
\[
  \paren[\Big]{
    \sum_{S \in \Fcal} y_S \incidvector{\delta(S)}
  }_{ij}
  =\frac{\mu}{(1-\Chernoff)\GWalpha}\frac{1}{T} X_{ij}.
\]
By~\cref{eq:sdp-rounding}, we have that
\(p_{ij} \ge \frac{\GWalpha}{\mu} z_{ij}\) and so the expected value
is \(\Ebb(X_{ij}) \ge T \frac{\GWalpha}{\mu} z_{ij}\).
Thus, in order to prove \cref{eq:prop-round-y}, it suffices to show
that \(\prob\paren{\exists e \in E,\, X_{e} \le (1 - \Chernoff)
\Ebb(X_{e})} \leq 1/n\).

Set \(\bar{Y} \coloneqq \mu^{-1} Y\) and let \(ij \in E\).
Using \(z \neq 0\), \cref{eq:prop-round-pre-3},
\cref{eq:prop-round-pre-2}, feasibility of \((\mu,Y)\)
in~\cref{eq:GW-polar-def-intro}, and~\cref{eq:Laplacian-adjoint-def},
we obtain
\begin{equation*}
  \frac{2\xi}{\kappa}
  = \frac{
    2\xi \norm[\infty]{z}
  }{
    \kappa \norm[\infty]{z}
  }
  \le \frac{
    2 z_{ij}
  }{
    \kappa \norm[\infty]{z}
  }
  \le \frac{
    2 z_{ij}
  }{
    \mu
  }
  \le \frac{1}{2} \Laplacian_G^*(\bar{Y})_{ij}
  = 1 - \bar{Y}_{ij}.
\end{equation*}
Hence \(\bar{Y}_{ij} \le 1 -
\frac{2\xi}{\kappa}\).
Using~\cref{eq:GW-Y-edge-marginal} and~\cref{eq:arccos-bound} we see
that
\begin{equation*}
  p_{ij}
  = \frac{\arccos(\bar{Y}_{ij})}{\pi}
  \geq \frac{2}{\pi} \sqrt{\frac{\xi}{\kappa}}.
\end{equation*}
By Chernoff's bound,
\begin{equation*}
  \prob\paren[\big]{
    X_{ij} \le (1 - \Chernoff) \Ebb(X_{ij})
  }
  \le \exp\paren[\bigg]{
    -\frac{
      \Chernoff^2 \Ebb(X_{ij})
    }{2}
  }
  = \exp\paren[\bigg]{
    -\frac{\Chernoff^2 T p_{ij}}{2}
  }
  \le \exp\paren[\bigg]{
    -\frac{\Chernoff^2 T}{\pi}
    \sqrt{\frac{\xi}{\kappa}}\,
  }.
\end{equation*}
Hence, by union bound and the lower bound on~\(T\),
\begin{align*}
  \prob\paren[\big]{
    \exists e \in E,\,
    X_{e} \le (1 - \Chernoff) \Ebb(X_{e})
  }
  &\le n^2 \exp\paren[\bigg]{
    -\frac{\Chernoff^2 T}{\pi}
    \sqrt{\frac{\xi}{\kappa}}\,
  }
  \\
  &\le \exp\paren[\bigg]{
    2\ln(n)
    -
    \frac{3\pi \kappa^{\half} \ln(n)}
    {\xi^{\half}\Chernoff^2}
    \frac{\Chernoff^2}{\pi}
    \sqrt{\frac{\xi}{\kappa}}\,
  }
  = \frac{1}{n}.\qedhere
\end{align*}
\end{proof}

\begin{remark}
\label{re:small-edges}
When applying \Cref{prop:probabilistic-rounding-cover-pre}, it is
apparent that \cref{eq:prop-round-pre-3} is the most stringent
condition.
Indeed, the other requirements can be satisfied by a nearly optimal
solution to the SDP in~\cref{eq:GW-polar-def-intro}.
On the other hand, \cref{eq:prop-round-pre-3}~restricts the
applicability of our procedure to certain well-behaved values of \(z\).
This is not an artifact of our analysis, but an unavoidable
consequence of the repeated sampling approach.
\Cref{sec:small-edges} shows that, for every \(\eps \in (0, 2)\),
there exists an instance \((G, z)\) which has an optimal solution
\((\bar{\mu}, \bar{Y})\) for~\cref{eq:GW-polar-def-intro} such
that there is an edge \(ij\) with
\[
  \prob\paren[\big]{ij \in \delta(\GWrv(\bar{Y}))}
  = \frac{\arccos(1 - 2\eps + \eps^2/2)}{\pi},
\]
which can be made arbitrarily small.
This, in turn, increases the number of samples needed to produce a
cut~cover, i.e., a set of cuts whose union is the whole edge set of
the graph.
\Cref{thm:rounding-algorithm} solves this issue by perturbing the edge
weights \(z\) to a vector \(\hat{z}\) that satisfies
\cref{eq:prop-round-pre-3}.
\end{remark}

Let \(G = (V, E)\) be a graph.
We will use that
\begin{alignat}{3}
  \label{eq:GW-polar-monotone}
  \GW^{\polar}(G, z_0) & \le \GW^{\polar}(G, z_1),
  & \qquad &
  \text{for every \(z_0,z_1 \in \Lp{E}\) such that \(z_0 \leq z_1\)},
  \\*
  \label{eq:GW-polar-convex}
  \GW^{\polar}(G, z_0 + z_1)
  & \le \GW^{\polar}(G, z_0) + \GW^{\polar}(G, z_1)
  & \qquad &
  \text{for every \(z_0,z_1 \in \Lp{E}\), and}
  \\*
  \label{eq:GW-polar-norm-infty-lowerbound}
  \norm[\infty]{z} & \le \GW^{\polar}(G, z)
  & \qquad &
  \text{for every \(z \in \Lp{E}\)}.
\end{alignat}
These facts follow from SDP Strong Duality: if \(z \in \Reals_+^E\),
then
\begin{subequations}
  \label{eq:GW-polar-def}
  \begin{align}
    \GW^{\polar}(G, z)
    &=
    \label{eq:GW-polar-gaugef}
    \min\setst[\big]{
      \mu
    }{
      \mu \in \Reals_{+},\,
      Y \in \Psd{V},\,
      \diag(Y) = \mu \ones,\,
      \tfrac{1}{4} \Laplacian_G^*(Y) \ge z
    }
  \\*
  &=
  \label{eq:GW-polar-supf}
  \max\setst[\big]{
    \iprodt{z}{w}
  }{
    w \in \Reals_+^E,\,
    x \in \Reals^V,\,
    \tfrac{1}{4}\Laplacian_G(w) \preceq \Diag(x),\,
    \iprodt{\ones}{x} \le 1
  }.
  \end{align}
\end{subequations}
The optimization problems \cref{eq:GW-polar-gaugef}
and~\cref{eq:GW-polar-supf} form a primal-dual pair of SDPs.
Note that
\((\slater{\mu}, \slater{Y}) \coloneqq (2\norm[\infty]{z},
2\norm[\infty]{z}I)\) and
\((\slater{w}, \slater{x}) \coloneqq (0,\card{V}^{-1} \ones)\) are
relaxed Slater points of~\cref{eq:GW-polar-gaugef},
and~\cref{eq:GW-polar-supf}, respectively, so SDP Strong Duality
ensures both problems have optimal solutions attaining a common
optimal value; see, e.g., \cite[Theorem~7.1.2]{Nemirovski2012}.
The proofs of~\cref{eq:GW-polar-monotone,eq:GW-polar-convex} are
immediate from~\cref{eq:GW-polar-gaugef} and~\cref{eq:GW-polar-supf},
respectively.
To prove~\cref{eq:GW-polar-norm-infty-lowerbound}, let
\(z \in \Lp{E}\) and note that, for every \(ij \in E\),
\begin{equation*}
2 \Diag(e_i + e_j) - \Laplacian_G(e_{ij})
= 2 (\E{i}{i} + \E{j}{j}) - \oprod{(e_i - e_j)}{(e_i - e_j)}
= (e_i + e_j)(e_i + e_j)^\transp
\in \Psd{V}.
\end{equation*}
Hence, for every \(ij \in E\), the pair
\((\bar{w}, \bar{x}) \coloneqq (e_{ij}, \frac{1}{2} (e_i + e_j))\) is
feasible in~\cref{eq:GW-polar-supf} for \((G,z)\).
Thus~\cref{eq:GW-polar-norm-infty-lowerbound} holds.

\begin{theorem}
\label{thm:rounding-algorithm}
Let \(\beta \in (0, \GWalpha)\) and set
\begin{equation}
  \label{eq:5}
  \tau \coloneqq 1 - \dfrac{\beta}{\GWalpha} \in (0,1)
  \qquad
  \text{and}
  \qquad
  C
  \coloneqq
  {81\sqrt{2}\pi}/{\tau^{5/2}}.
\end{equation}
There exists a randomized polynomial-time algorithm which takes as
input a graph \(G = (V, E)\) on \(n\) vertices and a vector \(z \in
\Reals_+^E\) and outputs \((\Fcal, y)\), where \(\Fcal \subseteq
\Powerset{V}\) and \(y \in \Reals_+^{\Fcal}\) are such that
\[
  \prob\paren[\bigg]{\,
    \sum_{S \in \Fcal}
    y_S \incidvector{\delta(S)} \ge z
  } \ge 1 - \frac{1}{n},\qquad
  \iprodt{\ones}{y} \le \frac{1}{\beta} \fcc(G, z),
  \qquad
  \text{and}
  \qquad
  \card{\Fcal}
  \le
  \ceil{C \ln(n)}
  = O(\ln(n)).
\]
That~is, with high probability, \(y\) is a
\(\tfrac{1}{\beta}\)-approximately optimal solution
for~\cref{eq:cut-covering-problem} with logarithmic support size.
\end{theorem}
\begin{proof}
We start by setting up the constants that will be used in the proof
(and in \Cref{alg:1}).
Set
\begin{equation}
  \label{eq:3}
  \sigma \coloneqq
  \eps \coloneqq
  \Chernoff \coloneqq
  \dfrac{\tau}{3}
  \in (0,\tfrac{1}{3})
  \quad\text{and}\quad
  T \coloneqq
  \ceil[\big]{C\ln(n)},
\end{equation}
as in the preamble to \cref{alg:1}.
The constants \(\sigma\), \(\eps\), and \(\Chernoff\) are chosen so
that
\begin{equation}
  \label{eq:1}
  \frac{1-\Chernoff}{1+\eps+\sigma} \GWalpha
  =
  \paren[\Big]{
    \frac{1-\tau/3}{1+2\tau/3}
  }
  \GWalpha
  =
  \paren[\Big]{
    1-\frac{\tau}{1+2\tau/3}
  }
  \GWalpha
  \geq
  \paren{
    1 - \tau
  }
  \GWalpha
  =
  \beta.
\end{equation}
It is simple to verify that \cref{alg:1} works if \(z = 0\), so we
may assume that \(z \neq 0\).
We define \(\hat{z}\) by rounding up entries that are smaller than
\(\frac{1}{2}\eps \norm[\infty]{z}\).
Set \(\hat{z} \in \Reals_+^E\) by \(\hat{z}_{ij} \coloneqq
\max\set{z_{ij}, \frac{1}{2}\eps \norm[\infty]{z}}\) for every \(ij
\in E\).
(Note that this is done in \cref{step:preconditioning} of \Cref{alg:1}.)

Let \((\mu,Y)\) be a feasible solution for
\(\GW^{\polar}(G, \hat{z})\) in \cref{eq:GW-polar-def-intro} with
objective value
\begin{equation}
  \label{eq:prob-round-2}
  \mu
  \leq
  \GW^{\polar}(G, \hat{z}) +\sigma \norm[\infty]{z}
  =
  \GW^{\polar}(G, \hat{z}) +\sigma \norm[\infty]{\hat{z}}.
\end{equation}
Note that this appears in \cref{step:sdp-approximate-solving} in
\Cref{alg:1}.
Such a nearly optimal solution \((\mu, Y)\) can be computed in
polynomial time due to the existence of strict Slater points for the
SDPs in~\cref{eq:GW-polar-def}; see~\cref{sec:solvers}.

We now move to the final and randomized part of \Cref{alg:1}.
Set \((\Fcal, y) \coloneqq \Acal_{T, \Chernoff}(G, Y)\).
Note that \(y\) is the solution produced by \Cref{alg:1}.
To finish the proof, we will show that \(y\) is a fractional cut cover
for \((G, z)\) with probability at least \(1 - \tfrac{1}{n}\), and it
has support size at most \( \ceil{C \ln(n)}\) and objective value at
most \((1/\beta) \fcc(G, z)\).

We will apply \Cref{prop:probabilistic-rounding-cover-pre} to show
that \(y\) is a fractional cut cover for \((G,\hat{z})\) with
probability at least \(1 - \tfrac{1}{n}\).
Since \(\hat{z} \geq z\), this implies that \(y\) is a fractional cut
cover for \((G,z)\) with probability at least \(1 - \tfrac{1}{n}\).
Set \(\xi \coloneqq \eps/2 = \tau/6\) and \(\kappa \coloneqq 3\), and
recall the definition of \(\Chernoff\) in~\cref{eq:3}.
By construction, \cref{eq:prop-round-pre-3} holds.
Note that
\begin{equation*}
  T
  = \ceil{C \ln(n)}
  = \ceil[\bigg]{
    \frac{81 \sqrt{2} \pi}{\tau^{5/2}} \ln(n)
  }
  = \ceil[\bigg]{
    3\pi
    \paren[\bigg]{\frac{18}{\tau}}^{1/2}
    \paren[\bigg]{\frac{3}{\tau}}^2
    \ln(n)
  }
  = \ceil[\bigg]{
    3\pi
    \paren[\bigg]{\frac{\kappa}{\xi}}^{1/2}
    \paren[\bigg]{\frac{1}{\gamma}}^2
    \ln(n)
  },
\end{equation*}
so the lower bound on~\(T\) from
\cref{prop:probabilistic-rounding-cover-pre} is met.
We will check that \(\mu \le \kappa \norm[\infty]{\hat{z}}\), that~is,
that \cref{eq:prop-round-pre-2} holds.
We~claim that
\begin{equation}
\label{eq:unweighted-GW-polar-upperbound}
\GW^{\polar}(G, \ones) \le 2.
\end{equation}
This follows from feasibility of
\((\bar{\mu}, \bar{Y}) \coloneqq (2, 2I)\)
in~\cref{eq:GW-polar-def-intro} for \((G,\ones)\), as
\(
\tfrac{1}{4} \Laplacian_G^*(I)
= \tfrac{1}{2} \ones
\)
by~\cref{eq:Laplacian-adjoint-def}.
Hence, \cref{eq:prob-round-2}, \cref{eq:GW-polar-monotone},
\cref{eq:unweighted-GW-polar-upperbound}, and \(\sigma < 1\) imply
\begin{equation*}
  \mu
  \le \GW^{\polar}(G, \hat{z}) + \sigma \norm[\infty]{\hat{z}}
  \le \GW^{\polar}(G, \norm[\infty]{\hat{z}} \ones)  + \sigma
  \norm[\infty]{\hat{z}}
  = \paren{\GW^{\polar}(G, \ones) + \sigma}\norm[\infty]{\hat{z}}
  \le 3 \norm[\infty]{\hat{z}},
\end{equation*}
so \cref{prop:probabilistic-rounding-cover-pre} applies.

The support size of~\(y\) is \(\card{\Fcal} \leq T = \ceil{C \ln(n)}\)
by~\cref{eq:Chernoff-sampling-properties}.
Finally, we bound \(\iprodt{\ones}{y}\):
\begin{detailedproof}
\iprodt{\ones}{y}
&= \frac{1}{\GWalpha}\frac{1}{1 - \Chernoff} \mu
&&\text{by~\cref{eq:Chernoff-sampling-properties}}\\
&\le \frac{1}{\GWalpha}\frac{1}{1 - \Chernoff}
(\GW^{\polar}(G, \hat{z}) + \sigma \norm[\infty]{\hat{z}})
&&\text{by~\cref{eq:prob-round-2}}\\
&\le \frac{1}{\GWalpha}\frac{1}{1 - \Chernoff}
\paren[\big]{
\GW^{\polar}(G, z + \tfrac{1}{2}\eps \norm[\infty]{\hat{z}} \ones)
+ \sigma \norm[\infty]{\hat{z}}
}
&&\text{by~\cref{eq:GW-polar-monotone},
as \(\hat{z} \le z + \tfrac{1}{2} \eps \norm[\infty]{\hat{z}}\ones\)}\\
&\le \frac{1}{\GWalpha}\frac{1}{1 - \Chernoff}
\paren[\big]{
\GW^{\polar}(G, z)
+ \tfrac{1}{2}\eps \norm[\infty]{\hat{z}} \GW^{\polar}(G, \ones)
+ \sigma \norm[\infty]{\hat{z}}
}
&&\text{by~\cref{eq:GW-polar-convex}}\\
&\le \frac{1}{\GWalpha}\frac{1}{1 - \Chernoff}
\paren[\big]{
\GW^{\polar}(G, z)
+ \eps\, \norm[\infty]{\hat{z}}
+ \sigma \norm[\infty]{\hat{z}}
}
&&\text{by~\cref{eq:unweighted-GW-polar-upperbound}}\\
&= \frac{1}{\GWalpha}\frac{1}{1 - \Chernoff}
\paren[\big]{
\GW^{\polar}(G, z)
+ (\eps + \sigma)\norm[\infty]{z}
}
&&\text{since \(\norm[\infty]{z} = \norm[\infty]{\hat{z}}\)}\\
&\le \frac{1}{\GWalpha}\frac{1 + \eps + \sigma}{1 - \Chernoff}
\GW^{\polar}(G, z)
&&\text{by~\cref{eq:GW-polar-norm-infty-lowerbound}}\\
&\leq\frac{1}{\beta}\GW^{\polar}(G, z)
&&\text{by~\cref{eq:1}}.\qedhere
\end{detailedproof}
\end{proof}

\begin{corollary}
\label{cor:probabilistic-method}

Let \(\beta \in (0, \GWalpha)\), and set \(\tau \in (0, 1)\) and \(C
\in \Reals_{++}\) as in~\cref{eq:5}.
For every graph \(G = (V, E)\) and \(z \in \Lp{E}\), there exists a
fractional cut cover \(y \in \Lp{\Powerset{V}}\) with
\(\card{\supp(y)} \le \ceil{C \ln n}\) and \(\iprodt{\ones}{y} \le
(1/\beta) \fcc(G, z)\).
\end{corollary}
\begin{proof}
Immediate from \cref{thm:rounding-algorithm}.
\end{proof}

\begin{minipage}{1.0\linewidth}
\section{
A Primal-Dual Extension of the Goemans--Williamson
Algorithm\texorpdfstring{\\%
}{} With Certificates of Approximate Optimality}
\label{sec:certificates}
\end{minipage}

\begin{figure}
  \centering
  \begin{tabular}[c]{|c|c|c|c|c|}
    \hline
    \textbf{Problem}
    & \textbf{\begin{tabular}{c}SDP Solution\\Properties\end{tabular}}
    & \textbf{\begin{tabular}{c}Rounding\\Procedure\end{tabular}}
    & \textbf{\begin{tabular}{c}Rounding\\Analysis\end{tabular}}
    & \textbf{Algorithm}
    \bigstrut\\\cline{1-5}
    Fractional Cut-Covering
    & \cref{prop:sdp-rounding}
    & \cref{eq:Chernoff-sampling-def}
    & \cref{prop:probabilistic-rounding-cover-pre}
    & \cref{thm:rounding-algorithm}
    \bigstrut\\\cline{1-5}
    Simultaneous Certificates
    & \cref{prop:real-number-model-certificate}
    & \multirow{3}{*}[-.4em]{\cref{eq:GWOpt-approx-sampling}}
    & \multirow{3}{*}[-.4em]{\cref{prop:GWOpt-approx-sampling}}
    & \cref{prop:sampling-GWOpt}
    \bigstrut\\\cline{1-2}\cline{5-5}
    Maximum Cut Certificates
    & \cref{thm:certificates-from-w}
    &
    &
    & \cref{thm:certificates-from-w-approx}
    \bigstrut\\\cline{1-2}\cline{5-5}
    Fractional Cut-Covering Certificates
    & \cref{thm:certificates-from-z}
    &
    &
    & \cref{thm:certificates-from-z-approx}
    \bigstrut\\
    \hline
  \end{tabular}
  \caption{\Cref{sec:certificates} produces solutions accompanied by
    certificates of their approximate optimality.
    These developments parallel \cref{sec:rounding}: we exploit the
    properties of optimal solutions to SDP relaxations in a rounding
    procedure.
    As an auxiliary step, we study \emph{simultaneous} approximate
    solutions to both problems.}
  \label{fig:table-of-results}
\end{figure}

\Cref{sec:rounding} describes an approximation algorithm for the
fractional cut-covering problem.
The work of Goemans and Williamson is so ubiquitous in our reasoning
that one may claim \Cref{alg:1} to be ``dual'' to the algorithm
described in \cite{GoemansWilliamson1995}.
This language suggests a primal-dual approach, where cuts and
fractional cut covers simultaneously certify each other's
(approximate) optimality via a suitable notion of ``weak~duality''.
This section provides randomized polynomial-time algorithms exploiting
this idea.
Fix a desired approximation factor \(\beta \in (0, \GWalpha)\).
Given a fractional cut-covering instance, we produce a fractional cut
cover whose \((1/\beta)\)\nobreakdash-approximate optimality is certified by a
maximum cut instance with one of its \(\beta\)-approximately optimal
solutions.
Symmetrically, the input may be a maximum cut instance, and the
algorithm then produces a \(\beta\)-approximately optimal cut and
certifies it via a fractional cut-covering instance with one of its
\((1/\beta)\)-approximately optimal solutions.
The alignment of subtopics between this section and
\Cref{sec:rounding} are highlighted by
Figure~\hyperref[fig:table-of-results]{1}.

\subsection{Gauge Duality}
\label{sec:duality}
This subsection presents the gauge duality theory that permeates and
forms the foundational basis for our results.
In this manner, this subsection places our work within the literature
and equips readers with a theoretical framework that can lead to new
results.

\begin{proposition}
  \label{prop:1}
Let \(G = (V, E)\) be a graph. The functions \(\mc(G, \cdot)\) and
\(\fcc(G, \cdot)\) satisfy
\begin{subequations}
  \label{eq:mc-fcc-polar}
  \begin{alignat}{3}
  \label{eq:mc-as-polar}
  \mc(G, w)
  &= \max\setst{
    \iprodt{w}{z}
  }{
    z \in \Reals_+^E,\,
    \fcc(G, z) \le 1
  }
  & \qquad & \text{for every }w \in \Reals_+^E,
  \\
  \label{eq:fcc-as-polar}
  \fcc(G, z)
  &= \max\setst{
    \iprodt{z}{w}
  }{
    w \in \Reals_+^E,\,
    \mc(G, w) \le 1
  }
  & \qquad & \text{for every }z \in \Reals_+^E,
  \\
   \label{eq:fcc-mc-CS}
   \iprodt{w}{z}
   &\le \mc(G, w) \fcc(G, z)
   &\qquad
  &\text{for every }w, z \in \Reals_+^E.
  \end{alignat}
\end{subequations}
The functions \(\GW(G, \cdot)\) and
\(\GW^{\polar}(G, \cdot)\) satisfy
\begin{subequations}
  \label{eq:gw-gw-polar-polar}
  \begin{alignat}{3}
  \label{eq:GW-as-polar}
  \GW(G, w)
  &= \max\setst{
    \iprodt{w}{z}
  }{
    z \in \Reals_+^E,\,
    \GW^{\polar}(G, z) \le 1
  },
  & \qquad & \text{for every }w \in \Reals_+^E,
  \\
  \label{eq:gw-polar-as-polar}
  \GW^{\polar}(G, z)
  &= \max\setst{
    \iprodt{z}{w}
  }{
    w \in \Reals_+^E,\,
    \GW(G, w) \le 1
  },
  & \qquad & \text{for every }z \in \Reals_+^E,
  \\
  \label{eq:GW-CS}
  \iprodt{w}{z}
  &\le \GW(G, w) \GW^{\polar}(G, z)
  &\qquad
  &\text{for every }
  w,z \in \Lp{E}.
  \end{alignat}
\end{subequations}
\end{proposition}

\begin{remark}
  The striking similarities between~\cref{eq:mc-fcc-polar}
  and~\cref{eq:gw-gw-polar-polar} underscore the existence of a
  theoretical framework that explains this phenomenon, rather than being
  a fortunate coincidence.
  It turns out that the functions \(\mc(G, \cdot)\), \(\fcc(G, \cdot)\),
  \(\GW(G, \cdot)\), and \(\GW^{\polar}(G, \cdot)\) are positive
  definite monotone gauges, which we shall define presently.
  Furthermore, \(\mc(G, \cdot)\) and \(\fcc(G, \cdot)\) form a dual
  pair, as do \(\GW(G, \cdot)\) and \(\GW^{\polar}(G, \cdot)\).
\end{remark}

\begin{proof}[Proof of \cref{prop:1}]
  \Cref{eq:fcc-as-polar} follows directly from Linear Programming
  Strong Duality, as
  \begin{align*}
    \fcc(G, z)
    &=
    \min\setst[\Big]{
      \iprodt {\ones} y
    }{
      y \in \Reals_+^{\Powerset{V}},\,
      \sum_{S \subseteq V} y_S \incidvector{\delta(S)} \ge z
    }
    \\*
    &= \max \setst{
      \iprodt{z}{w}
    }{
      w \in \Reals_+^E,\,
      \forall S \subseteq V,\,
      \iprodt{w}{\incidvector{\delta(S)}} \le 1
    }
    \\*
    &=
    \max\setst{
      \iprodt{z}{w}
    }{
      w \in \Reals_+^E,\,
      \mc(G, w) \le 1
    },
  \end{align*}
  and \cref{eq:fcc-mc-CS} is then straightforward.
  Next, we show that \cref{eq:mc-as-polar} holds.
  We have that
  \begin{detailedproof}
    \mc(G, w)
    &= \max\setst{
      \iprodt{w}{\incidvector{\delta(S)}}
    }{
      S \subseteq V
    }\\
    &\le \max\setst[\Big]{
      \iprodt{w}{z}
    }{
      z \in \Reals_+^E,\,
      y \in \Reals_+^{\Powerset{V}},\,
      \iprodt{\ones}{y} \leq 1,\,
      z \le \sum_{S \subseteq V} y_S\incidvector{\delta(S)}
    }
    &&\text{take \((z, y) \coloneqq (\incidvector{\delta(S)}, e_S)\)}\\
    &= \max\setst{
      \iprodt{w}{z}
    }{
      z \in \Reals_+^E,\,
      \fcc(G, z) \le 1
    }
    &&\text{by~\cref{eq:cut-covering-problem}}\\
    &\le \mc(G, w)
    &&\text{by~\cref{eq:fcc-mc-CS}.}
  \end{detailedproof}
  The proof of~\cref{eq:gw-gw-polar-polar} follows a similar
  structure.
  For every \(w \in \Reals_+^E\), equation~\eqref{eq:gw-polar-as-polar} follows
  from SDP Strong Duality via \cref{eq:GW-intro-pd}
  and~\cref{eq:GW-polar-supf}.
  The Cauchy-Schwarz inequality in~\cref{eq:GW-CS} then follows
  from~\cref{eq:gw-polar-as-polar}.
  Finally,~\cref{eq:GW-as-polar} follows, since
    \begin{detailedproof}
      \GW(G, w)
      &= \max\setst{
        \iprodt{w}{\paren[\big]{\tfrac{1}{4}\Laplacian_G^*(Y)}}
      }{
        Y \in \Psd{V},\,
        \diag(Y) = \ones
      }
      &&\text{by~\cref{eq:GW-intro-def}}\\
      &\le \max\setst{
        \iprodt{w}{z}
      }{
        z \in \Reals_+^E,\,
        z \le \tfrac{1}{4}\Laplacian_G^*(Y),\,
        Y \in \Psd{V},\,
        \diag(Y) = \ones
      }\\
      &= \max\setst{
        \iprodt{w}{z}
      }{
        z \in \Reals_+^E,\,
        \GW^{\polar}(G, z) \le 1
      }
      &&\text{by~\cref{eq:GW-polar-def-intro}}\\*
      &\le \GW(G, w)
      &&\text{by~\cref{eq:GW-CS}.}
      \qedhere
    \end{detailedproof}
\end{proof}

In the following discussion, we will elaborate on the above concepts
along with their associated implications.
Let \(E\) be a finite set, and let \(\phi \colon \Lp{E} \to \Reals_+\)
be a function such that \(\phi(0) = 0\).
We say that
\begin{minipage}{1.0\linewidth}
  \begin{itemize}
  \item \(\phi\) is \emph{positive definite} if \(\phi(x) > 0\) for
    every nonzero \(x \in \Lp{E}\);
  \item \(\phi\) is \emph{monotone} whenever \(x \le y\) implies
    \(\phi(x) \le \phi(y)\) for every \(x, y \in \Lp{E}\);
  \item \(\phi\) is \emph{positively homogeneous} if
    \(\phi(\alpha x) = \alpha \phi(x)\) for every \(x \in \Lp{E}\) and
    \(\alpha \in \Reals_{++}\);
  \item \(\phi\) is a \emph{gauge} if it is convex and positively
    homogeneous.
  \end{itemize}
  \smallskip
\end{minipage}
If \(\phi \colon \Lp{E} \to \Reals_+\) is a positive definite monotone
gauge, we define its \emph{dual}
\(\phi^{\polar} \colon \Lp{E} \to \Reals_+\) by
\begin{equation}
  \label{eq:gauge-polar-def}
  \phi^{\polar}(z)
  \coloneqq
  \max \setst{
    \iprodt{z}{w}
  }{
    w \in \Lp{E},\,
    \phi(w) \le 1
  }
  \qquad
  \text{ for every }
  z \in \Lp{E}.
\end{equation}
It is routine to check that \(\phi^{\polar}\) is a positive definite
monotone gauge.
One can also exploit a hyperplane separation theorem to show that
\(\phi^{\polar\polar} = \phi\).

\begin{corollary}
  Let \(G = (V, E)\) be a graph.
  Then the functions \(\mc(G, \cdot)\) and \(\GW(G, \cdot)\) are
  positive definite monotone gauges, and their duals are
  \(\fcc(G,\cdot)\) and~\(\GW^{\polar}(G,\cdot)\), respectively.
\end{corollary}
\begin{proof}
  This follows directly from \cref{prop:1}.
\end{proof}

We have already exploited the properties of positive definite monotone
gauges throughout our work: recall~\cref{eq:GW-polar-monotone}
and~\cref{eq:GW-polar-convex}, for example, which state monotonicity
and convexity of \(\GW^{\polar}(G, \cdot)\), respectively.
More importantly, gauge duality immediately establishes a \emph{bound
  conversion}
procedure~\cite[Section~6]{BenedettoProencadeCarliSilvaEtAl2021} on
which this work is~based.
Recall~\cref{eq:GW-approximation} and \cref{eq:GW-polar-approx-pledge}:
\begin{alignat*}{3}
  \GWalpha \GW(G, w)
  & \le \mc(G, w)
  && \le \GW(G, w)
  &
  \mathrlap{
    \qquad
    \text{for each \(w \in \Reals_+^{E}\)},
  }
  \tag{\ref{eq:GW-approximation}}
  \\
  \GW^{\polar}(G, z)
  & \le \,\fcc(G, z)
  && \le \frac{1}{\GWalpha} \GW^{\polar}(G, z)
  &
  \mathrlap{
    \qquad
    \text{for each \(z \in \Reals_+^{E}\)}.
  }
  \tag{\ref{eq:GW-polar-approx-pledge}}
\end{alignat*}
The relationship between these inequalities are instances of the
following result.
\begin{proposition}
  \label{prop:bound-conversion}
  Let \(E\) be a finite set, and let
  \(\phi, \psi \colon \Lp{E} \to \Reals_+\) be positive definite
  monotone gauges.
  Let \(\alpha, \beta \in \Reals_{++}\).
  Then the following are equivalent:
  \begin{alignat*}{3}
    \alpha \psi(w)
    &\le
    \phi(w)
    &&\le \beta \psi(w)
    & \mathrlap{
      \qquad
      \text{ for every } w \in \Lp{E};
    }
    \\
    \tfrac{1}{\beta} \psi^{\polar}(z)
    &\le
    \phi^{\polar}(z)
    &&\le
    \tfrac{1}{\alpha} \psi^{\polar}(z)
    & \mathrlap{
      \qquad
      \text{ for every } z \in \Lp{E}.
    }
  \end{alignat*}
\end{proposition}
\begin{proof}
Since \(\phi^{\polar}, \psi^{\polar} \colon \Lp{E} \to \Reals_+\) are
positive definite monotone gauges and \(\phi^{\polar\polar} =
\phi\) and \(\psi^{\polar\polar} = \psi\), it~suffices to show that
\begin{equation}
\label{eq:24}
  \text{ if }
  \phi(w) \le \beta \psi(w)
  \text{ for every \(w \in \Lp{E}\), then }
  \frac{1}{\beta}\psi^{\polar}(z) \le \phi^{\polar}(z)
  \text{ for every \(z \in \Lp{E}\).}
\end{equation}
This follows from the fact that, for every \(z \in \Lp{E}\),
\begin{equation*}
  \phi^{\polar}(z)
  = \max\setst{\iprodt{z}{w}}{w \in \Lp{E},\, \phi(w) \le 1}
  \text{ and }
  \tfrac{1}{\beta} \psi^{\polar}(z)
  = \max\setst{\iprodt{z}{w}}{w \in \Lp{E},\, \beta \psi(w) \le 1}.
  \qedhere
\end{equation*}
\end{proof}

From a geometric viewpoint, positive definite monotone gauges are
deeply related to convex corners.
Let \(E\) be a finite
set. The \emph{lower-comprehensive hull of \(\Ccal \subseteq
  \Reals_+^E\)} is defined by \(\lc(\Ccal) \coloneqq \setst{z \in
  \Reals_+^E}{\exists x \in \Ccal, z \le x}\), and \(\Ccal\) is
\emph{lower comprehensive} if \(\lc(\Ccal) = \Ccal\). A \emph{convex
  corner} is a lower-comprehensive compact convex set \(\Ccal
\subseteq \Lp{E}\) with nonempty interior.

Every positive definite monotone gauge \(\phi \colon \Lp{E} \to
\Reals_+\) is associated to a convex corner
\begin{equation}
  \label{eq:assoc-convex-corner}
  \Ccal_{\phi}
  \coloneqq
  \setst{x \in \Lp{E}}{\phi^{\polar}(x) \le 1}
\end{equation}
which satisfies
\begin{equation*}
  \phi(w) = \max\setst{\iprodt{w}{x}}{x \in \Ccal_{\phi}}
  \qquad
  \text{for each \(w \in \Lp{E}\)}.
\end{equation*}
The \emph{antiblocker of \(\Ccal \subseteq \Lp{E}\)} is defined by
\(\abl(\Ccal) \coloneqq \setst{y \in \Lp{E}}{\forall x \in \Ccal,\,
  \iprodt{y}{x} \le 1}\).
From these definitions and the fact that
\(\phi^{\polar\polar} = \phi\), it~is~clear that
\(\Ccal_{\phi^{\polar}} = \abl(\Ccal_{\phi})\).
This correspondence allows one to recast gauge duality results in
terms of \emph{antiblocking duality}
\cite{Fulkerson1971,Fulkerson1972}.

Let \(G = (V,E)\) be a graph.
Define the \emph{cut polytope of~\(G\)} as
\begin{subequations}
  \begin{align}
    \label{eq:2}
    \CUT(G)
    & \coloneqq
    \conv\setst{\incidvector{\delta(S)}}{S \subseteq V}
    \subseteq \Reals^E,
    \\
    \intertext{where \(\conv(\cdot)\) denotes the convex hull, and
    define its semidefinite relaxation as}
    \GWcut(G)
    & \coloneqq
    \setst{\tfrac{1}{4}\Laplacian_G^*(Y)}{Y \in \elliptope{V}}
    \subseteq \Reals^E.
  \end{align}
\end{subequations}
The convex corners associated to
\(\mc(G, \cdot)\), \(\fcc(G,\cdot)\), \(\GW(G, \cdot)\), and
\(\GW^{\polar}(G, \cdot)\) as in \cref{eq:assoc-convex-corner}
are the following:
\begin{alignat}{3}
  \Ccal_{\mc(G, \cdot)} &= \lc(\CUT(G))
  &\quad\text{and}\quad&
  \Ccal_{\fcc(G, \cdot)} &&=\abl(\CUT(G)),
  \\
  \label{eq:GW-convex-corners}
  \Ccal_{\GW(G, \cdot)} &= \lc(\GWcut(G))
  &\quad\text{and}\quad&
  \Ccal_{\GW^{\polar}(G, \cdot)} &&= \abl(\GWcut(G)).
\end{alignat}
Basic convex analysis shows that studying \(\fcc(G,\cdot)\) for
\emph{all} weights \(z \in \Reals_+^{E}\) corresponds to studying the
\emph{whole} boundary structure of \(\Ccal_{\fcc(G,\cdot)}\), not just in the direction \(z = \ones\).
The set \(\lc(\CUT(G))\) above has appeared previously in the
literature as the bipartite subgraph polytope of~\(G\); see, e.g.,
\cite{GroetschelP81a}.

Moreover, the inequalities in~\cref{eq:GW-approximation} and
\cref{eq:GW-polar-approx-pledge} can be interpreted as the following
set inclusions, respectively:
\begin{subequations}
  \begin{alignat}{3}
    \GWalpha \lc(\GWcut(G))
    &\subseteq \,\,\lc(\CUT(G))
    &&\subseteq \lc(\GWcut(G)),
    \\
    \abl(\GWcut(G))
    &\subseteq
    \abl(\CUT(G))
    &&\subseteq
    \tfrac{1}{\GWalpha}\abl(\GWcut(G)).
  \end{alignat}
\end{subequations}

We refer the reader to
\cite[Sections~2--7]{BenedettoProencadeCarliSilvaEtAl2021} or
\cite[Sections~4.1--4.3]{BenedettoProenca2021} for an in-depth
discussion about gauge duality, with elementary proofs of the
aforementioned results.
One may regard gauge duality as a manifestation of convex duality.
For example, \nameandcite{Freund1987} formulates pairs of primal and
dual gauge optimization problems, and proves a strong duality result
using a hyperplane separation theorem.
In this form, gauge duality has received a lot of attention in the
optimization community recently; see
\cite{FriedlanderMacedoEtAl2014,AravkinBurkeEtAl2018}.
We remark that the work of
\citeauthor{GrotschelLovaszEtAl1981}~\cite[Corollary~3.5]{GrotschelLovaszEtAl1981},
together with the remark that~\cref{eq:GW-intro-def} defines a
positive definite monotone gauge, immediately implies that one can
approximate the optimal value of~\cref{eq:GW-polar-def-intro} to any
given precision in polynomial time.
The algorithm in \cref{sec:rounding} refines this by showing that,
beyond the polynomial-time computable lower bound to the value of the
fractional cut-covering number, one has a suitable approximation
algorithm leveraging the work of \nameandcite{GoemansWilliamson1995}
that actually constructs an approximately optimal fractional cut cover.

\subsection{\texorpdfstring{\(\beta\)}{β}-pairings}

Let \(G = (V, E)\) be a graph.
For every vector \(w \in \Lp{E}\), one has an instance~\((G, w)\) of
the maximum cut problem.
Similarly, for every vector \(z \in \Lp{E}\), one has an instance
\((G, z)\) of the fractional cut-covering problem.
From a computational complexity point of view, it is remarkable
how~\cref{eq:fcc-mc-CS} relates the optimal values \(\mc(G, w)\) and
\(\fcc(G, z)\).
It is natural then to try to find pairs \((w, z) \in \Lp{E} \times
\Lp{E}\) which mutually certify the optimality of each other.
That such pairs exist is a consequence of~\cref{eq:mc-fcc-polar}.
As we are interested in approximation algorithms, we then parameterize
this relationship between instances by a real number \(\beta \in
\halfclosed{0, 1}\), interpreted as an approximation factor.

\begin{definition}
  \label{def:beta-pairing}
  Let \(G=(V,E)\) be a graph and let \(\beta\in\halfclosed{0,1}\).
  A \emph{\(\beta\)-pairing} on \(G\) is a pair
  \((w, z) \in \Lp{E} \times \Lp{E}\) such that there exist
  \(\rho, \mu \in \Lp{}\), such that \(\rho = 0 = \mu\) if and only if
  \(w = 0 = z\), and
\begin{equation}
  \label{eq:beta-pairing-def}
  \iprodt{w}{z}
  \,\overset{\text{\hyperref[eq:beta-pairing-def]{(\ref{eq:beta-pairing-def}a)}}}{\mathclap{=}}
  \rho\mu
  \qquad
  \text{and}
  \qquad
  \beta\rho\mu
  \overset{\text{\hyperref[eq:beta-pairing-def]{(\ref{eq:beta-pairing-def}b)}}}{\mathclap{\leq}}
  \mc(G, w)\mu
  \overset{\text{\hyperref[eq:beta-pairing-def]{(\ref{eq:beta-pairing-def}c)}}}{\mathclap{\leq}}
  \rho\mu
  \overset{\text{\hyperref[eq:beta-pairing-def]{(\ref{eq:beta-pairing-def}d)}}}{\mathclap{\leq}}
  \rho\fcc(G, z)
  \overset{\text{\hyperref[eq:beta-pairing-def]{(\ref{eq:beta-pairing-def}e)}}}{\mathclap{\leq}}\,
  \frac{1}{\beta}\rho\mu.
\end{equation}
We define an \emph{exact pairing} on \(G\) to be a \(1\)-pairing on
\(G\).
\end{definition}
When \(\rho>0\) and \(\mu>0\), which we regard as the ``typical''
case, we may restate
\cref{eq:beta-pairing-def} as
\begin{gather*}
  \iprodt{w}{z}
  \overset{\text{\hyperref[eq:beta-pairing-def]{(\ref{eq:beta-pairing-def}a)}}}{\mathclap{=}}
  \rho\mu,\qquad
  \beta\rho
  \overset{\text{\hyperref[eq:beta-pairing-def]{(\ref{eq:beta-pairing-def}b)}}}{\mathclap{\leq}}
  \mc(G, w)
  \overset{\text{\hyperref[eq:beta-pairing-def]{(\ref{eq:beta-pairing-def}c)}}}{\mathclap{\leq}}
  \rho,
  \quad\text{and}\quad
  \mu
  \overset{\text{\hyperref[eq:beta-pairing-def]{(\ref{eq:beta-pairing-def}d)}}}{\mathclap{\leq}}
  \fcc(G, z)
  \overset{\text{\hyperref[eq:beta-pairing-def]{(\ref{eq:beta-pairing-def}e)}}}{\mathclap{\leq}}
  \tfrac{1}{\beta}\mu.
\end{gather*}
The definition is made to accommodate the case \(0 \in \set{\rho,
  \mu}\).
The nonzero conditions on~\(\rho\) and \(\mu\) are meant only to
avoid ``spurious'' \(\beta\)-pairings.
In~fact, one can easily check that, in \cref{def:beta-pairing},
\begin{equation*}
  \rho = 0 \text{ if and only if }w = 0
  \qquad
  \text{and}
  \qquad
  \mu = 0 \text{ if and only if }z = 0.
\end{equation*}

Let \(G = (V, E)\) be a graph, and let \(w,\,z \in \Lp{E}\).
If we take \(\rho \coloneqq \GW(G, w)\) and \(\mu \coloneqq
\GW^{\polar}(G, z)\), \Cref{prop:bound-conversion}
states that
\hyperref[eq:beta-pairing-def]{(\ref{eq:beta-pairing-def}b)} and
\hyperref[eq:beta-pairing-def]{(\ref{eq:beta-pairing-def}c)} are
the ``dual inequalities'' to
\hyperref[eq:beta-pairing-def]{(\ref{eq:beta-pairing-def}e)} and
\hyperref[eq:beta-pairing-def]{(\ref{eq:beta-pairing-def}d)},
respectively.
\Cref{def:beta-pairing} abstracts some of the concepts from positive
definite monotone gauges mentioned in~\cref{sec:duality}, while
preserving the relevant duality.
\Cref{def:beta-pairing} also shifts the focus to the relation
established between \(w\) and \(z\), which is crucial to
our certification approach.
Note that for an exact pairing \((w,z)\), the above solutions are
optimal with a precise relationship:
\begin{equation*}
  \iprodt{w}{z}
  = \mc(G, w)\fcc(G, z),\qquad
  \mc(G, w) = \rho,
  \quad\text{and}\quad
  \fcc(G, z) = \mu.
\end{equation*}

While the definition of a \(\beta\)-pairing captures the notion of
simultaneous \(\beta\)-approximations of the \emph{numbers}
\(\mc(G, w)\) and \(\fcc(G, z)\), it is important to consider which
objects might \emph{certify} that a pair
\((w, z) \in \Lp{E} \times \Lp{E}\) is indeed a \(\beta\)-pairing for
a fixed \(\beta \in \halfclosed{0, 1}\).
The lower
bound~\hyperref[eq:beta-pairing-def]{(\ref{eq:beta-pairing-def}b)} on
the maximum cut value and the upper
bound~\hyperref[eq:beta-pairing-def]{(\ref{eq:beta-pairing-def}e)} on
the fractional cut-covering value in~\cref{eq:beta-pairing-def} can be
naturally certified by (the shore of) a cut and a fractional cut
cover, respectively.
However, certifying the upper
bound~\hyperref[eq:beta-pairing-def]{(\ref{eq:beta-pairing-def}c)} on
\(\mc(G, w)\) and the lower
bound~\hyperref[eq:beta-pairing-def]{(\ref{eq:beta-pairing-def}d)} on
\(\fcc(G, z)\), i.e., \(\mc(G, w)\mu \le \rho\mu \le \rho\fcc(G, z)\),
poses a more complex question.
We achieve this certification by using semidefinite programming weak
duality.
Since \(\mc(G,w) \leq \GW(G, w)\), we have that
\begin{equation}
  \label{eq:maxcut-sdp-uppperbound}
  \mc(G, w) \le \rho
  \text{ for each feasible solution }
  (\rho, x)
  \text{ of~\cref{eq:GW-gaugef-rho}}.
\end{equation}
From the viewpoint of \(\GW(G, w)\) as a semidefinite relaxation for
\(\mc(G,w)\), it is very natural to regard \(x\) as (the key part of)
a feasible solution for its dual SDP~\cref{eq:GW-gaugef-rho}.
However, the vector \(x \in \Reals^V\) also has a combinatorial
interpretation.
Direct computation using~\cref{eq:Laplacian-adjoint-def} shows that
\begin{subequations}
  \begin{alignat}{3}
    \Laplacian_G^*(\oprod{\incidvector{S}}{\incidvector{S}})
    &=
    \incidvector{\delta(S)}
    & \qquad &
    \text{for each \(S \subseteq V\)},
    \\
    \Laplacian_G^*(\oprod{\ones}{h})
    &=
    \Laplacian_G^*(\oprod{h}{\ones})
    =
    0
    & \qquad &
    \text{for each \(h \in \Reals^V\)}.
  \end{alignat}
\end{subequations}
Thus, the inequality \(\tfrac{1}{4}\Laplacian_G(w) \preceq \Diag(x)\)
from~\cref{eq:GW-gaugef-rho} implies the middle inequality in:
\begin{equation}
  \label{eq:comb-weak-duality}
  \iprodt{w}{\incidvector{\delta(S)}}
  = \iprod{
    \tfrac{1}{4}\Laplacian_G(w)
  }{
    \oprod{
      \paren{\ones - 2\incidvector{S}}
    }{
      \paren{\ones - 2\incidvector{S}}
    }
  }
  \le \iprod{
    \Diag(x)
  }{
    \oprod{
      \paren{\ones - 2\incidvector{S}}
    }{
      \paren{\ones - 2\incidvector{S}}
    }
  }
  =
  \iprodt{\ones}{x}
  \leq
  \rho.
\end{equation}
This shows that, while the semidefinite inequality
\(\tfrac{1}{4}\Laplacian_G(w) \preceq \Diag(x)\) just used may look at
a first glance not quite combinatorial, the only property used in the
proof of \cref{eq:comb-weak-duality} is the more combinatorial-looking
\begin{equation}
  \label{eq:comb-weak-duality-core}
  \iprod{\tfrac{1}{4}\Laplacian_G(w)}{\oprod{h}{h}}
  \le
  \iprod{\Diag(x)}{\oprod{h}{h}}
  \qquad
  \text{for each \(h \in \set{\pm 1}^V\)}.
\end{equation}
Hence, any certificate for the inequality
\(\tfrac{1}{4}\Laplacian_G(w) \preceq \Diag(x)\) certifies the
inequality in \cref{eq:comb-weak-duality-core} for arbitrary
\(h \in \Reals^V\), and in particular for each \(h \in \set{\pm1}^V\),
i.e., for each \(h \in \Reals^{V}\) of the form
\(\ones - 2\incidvector{S}\) for some \(S \subseteq V\).
Certifying a richer family of inequalities can be seen as dual to
solving a relaxation of \(\mc(G,w)\).
Finally, the semidefinite inequality
\(\tfrac{1}{4}\Laplacian_G(w) \preceq \Diag(x)\) can be certified by
providing an \(LDL^{\transp}\) factorization (that~is,
a~square-root-free Cholesky decomposition) of the positive
semidefinite matrix \(\Diag(x) - \tfrac{1}{4}\Laplacian_G(w)\).
With~\cref{eq:comb-weak-duality} we have showed how to certify
\(\mc(G,w) \leq \rho\); it remains to discuss certification of
\(\fcc(G, z) \ge \mu\).
However, as we shall prove in the upcoming results, the latter
inequality can be certified by the \emph{same} objects that certify
the former inequality for appropriately paired edge weights~\(w\)
and~\(z\).
Such a \emph{simultaneous certification} is a key aspect of our work.
We have now gathered all the ingredients we need to define certificates
for \(\beta\)-pairings.

\begin{definition}
  \label{def:beta-cert}
  Let \(G=(V,E)\) be a graph and let \(\beta\in\halfclosed{0,1}\).
  Let \((w,z) \in \Lp{E} \times \Lp{E}\).
  A \emph{\(\beta\)-certificate for \((w, z)\)} is a tuple
  \((\rho, \mu, S, y, x)\) such that \(\rho = 0 = \mu\) if and only if
  \(w = 0 = z\), and
  \begin{subequations}
    \label{eq:beta-cert-items}
    \renewcommand{\theequation}{\theparentequation.\roman{equation}}
    \begin{flalign}
      \label{item:cert-1}
      &\rho,\mu \in \Reals_+
      \text{ are such that }
      \rho\mu = \iprodt{w}{z},&&
      \\
      \label{item:cert-2}
      &S \subseteq V
      \text{ is such that }
      \iprodt{w}{\incidvector{\delta(S)}} \ge \beta \rho,&&
      \\
      \label{item:cert-3}
      &y \in \Reals_+^{\Powerset{V}}
      \text{ is such that }
      {\textstyle\sum_{U \subseteq V}} y_U \incidvector{\delta(U)} \ge z
      \text{ and }
      \iprodt{\ones}{y} \le \tfrac{1}{\beta}\mu,
      \text{ and}&&
      \\
      \label{item:cert-4}
      &x \in \Reals^V
      \text{ is such that }
      \rho \geq \iprodt{\ones}{x}
      \text{ and }\Diag(x) \succeq \tfrac{1}{4}\Laplacian_G(w).&&
    \end{flalign}
  \end{subequations}
\end{definition}

\begin{remark}
\label{re:beta-certificate-alternatives}
Recalling the discussion preceding \Cref{def:beta-pairing}, note how
items \cref{item:cert-1}, \cref{item:cert-2}, and \cref{item:cert-3}
are the natural certificates for the inequalities
\hyperref[eq:beta-pairing-def]{(\ref{eq:beta-pairing-def}a)},
\hyperref[eq:beta-pairing-def]{(\ref{eq:beta-pairing-def}b)}, and
\hyperref[eq:beta-pairing-def]{(\ref{eq:beta-pairing-def}e)},
respectively.
Whereas we work with an SDP certificate in \cref{item:cert-4}, this
setup opens the possibility of using other techniques which upper
bound the maximum value of the weighted maximum cut problem.
Concretely, one could substitute \cref{item:cert-4} by appropriate
certificates arising from \cite{Trevisan2012,HopkinsSchrammEtAl2020},
for example.
\end{remark}

Next, we prove that a \(\beta\)-certificate for \((w,z)\) does indeed
certify that \((w,z)\) is a \(\beta\)-pairing.
\begin{proposition}
  \label{prop:2}
  Let \(G=(V,E)\) be a graph and let \(\beta\in\halfclosed{0,1}\).
  Let \((w,z) \in \Lp{E} \times \Lp{E}\).
  If there exists a \(\beta\)-certificate for \((w, z)\), then \((w,z)\)
  is a \(\beta\)-pairing.
\end{proposition}
\begin{proof}
  Let \((\rho, \mu, S, y, x)\) be a \(\beta\)-certificate for
  \((w, z)\).
  Item~\eqref{item:cert-1} proves
  \hyperref[eq:beta-pairing-def]{(\ref{eq:beta-pairing-def}a)}.
  One has
  \begin{equation*}
    \beta \rho \mu
    \le
    \mu \iprodt{w}{\incidvector{\delta(S)}}
    \le
    \mu \mc(G, w)
  \end{equation*}
  by \cref{item:cert-2}.
  This proves
  \hyperref[eq:beta-pairing-def]{(\ref{eq:beta-pairing-def}b)}.
  Similarly,
  \begin{equation*}
    \rho \fcc(G, z)
    \le
    \rho \iprodt{\ones}{y}
    \le
    \tfrac{1}{\beta} \rho \mu
  \end{equation*}
  by~\cref{item:cert-3}.
  This proves
  \hyperref[eq:beta-pairing-def]{(\ref{eq:beta-pairing-def}e)}.

  So far we have only used feasible solutions for \(\mc(G,w)\)
  and~\(\fcc(G,z)\) to obtain bounds for the optimal values.
  Finally, \cref{item:cert-4} shows
  that~\cref{eq:maxcut-sdp-uppperbound} applies, so
  \begin{equation*}
    \mc(G, w) \mu
    \overset{\text{\cref{eq:maxcut-sdp-uppperbound}}}{\mathclap{\le}}
    \rho \mu
    \overset{\text{\cref{item:cert-1}}}{\mathclap{=}}
    \iprodt{w}{z}
    \overset{\text{\cref{eq:fcc-mc-CS}}}{\mathclap{\le}}
    \mc(G, w) \fcc(G, z)
    \overset{\text{\cref{eq:maxcut-sdp-uppperbound}}}{\mathclap{\le}}
    \rho \fcc(G, z),
  \end{equation*}
  thus proving
  \hyperref[eq:beta-pairing-def]{(\ref{eq:beta-pairing-def}c)} and
  \hyperref[eq:beta-pairing-def]{(\ref{eq:beta-pairing-def}d)}.
\end{proof}

\subsection{Existence of \texorpdfstring{\(\GWalpha\)}{α}-Certificates}
\label{sec:existence-of-certificates}

Having motivated the definition of \(\beta\)-certificates as objects
simultaneously proving approximate optimality for the maximum cut and
fractional cut-covering problems, a next step would be to determine
conditions on a \(\beta\)-pairing \((w, z) \in \Lp{E} \times \Lp{E}\)
that guarantee the existence of a \(\beta\)-certificate for \((w, z)\).
Set
\begin{equation}
  \label{eq:GWOpt-as-tight-CS}
  \GWOpt(G)
  \coloneqq \setst{
    (w, z) \in \Lp{E} \times \Lp{E}
  }{
    \iprodt{w}{z} = \GW(G, w)\GW^{\polar}(G, z)
  }.
\end{equation}
We now expand~\cref{eq:GWOpt-as-tight-CS} into a more
convenient characterization.
Recalling~\cref{eq:GW-gaugef-rho} and \cref{eq:GW-polar-def-intro},
\begin{align}
  \GW(G, w)
  &=
  \min\setst[\big]{
    \rho
  }{
    \rho \in \Reals_{+},\,
    x \in \Reals^V,\,
    \rho \geq \iprodt{\ones}{x},\,
    \Diag(x) \succeq \tfrac{1}{4}\Laplacian_G(w)
  },
  \tag{\ref{eq:GW-gaugef-rho}}
  \\
  \GW^{\polar}(G, z)
  &=
  \min\setst[\big]{
    \mu
  }{
    \mu \in \Reals_+,\,
    Y \in \Psd{V},\,
    \diag(Y) = \mu \ones,\,
    \tfrac{1}{4}\Laplacian_G^*(Y) \ge z
  },
  \tag{\ref{eq:GW-polar-def-intro}}
\end{align}
we claim that
\begin{equation}
  \label{eq:GWOpt-def}
  \GWOpt(G)
  =
  \setst*{
    (w, z)
    \in \Lp{E} \times \Lp{E}
  }{
    \begin{array}{c}
      \exists (\rho,x)\text{ feasible
        for~\cref{eq:GW-gaugef-rho} for }(G,w),\\[2pt]
      \exists (\mu,Y) \text{ feasible
        for~\cref{eq:GW-polar-def-intro} for }(G,z),\\[2pt]
      \text{and } \iprodt{w}{z} \ge \rho \mu
    \end{array}
  }.
\end{equation}
For each \((w,z) \in \GWOpt(G)\), we shall refer to pairs \((\rho,x)\)
and \((\mu,Y)\) assumed to exist as in the RHS of~\cref{eq:GWOpt-def},
as~\emph{witnesses of the membership} \((w,z) \in \GWOpt(G)\).
To prove `\(\subseteq\)' in~\cref{eq:GWOpt-def}, it suffices to choose
as witnesses an optimal solution \((\rho,x)\)
for~\cref{eq:GW-gaugef-rho} and an optimal solution \((\mu,Y)\)
for~\cref{eq:GW-polar-def-intro}.
Next we prove `\(\supseteq\)'.
Note that,
\begin{equation}
  \label{eq:GWOpt-strong-duality}
  \text{%
    if \((\rho, x)\) and \((\mu,Y)\) witness the membership \((w, z)
    \in \GWOpt(G)\), then
  }
  \rho\mu
  =
  \iprodt{w}{z}
  =
  \iprod{\tfrac{1}{4}\Laplacian_G(w)}{Y},
\end{equation}
since
\(\rho\mu \le \iprodt{w}{z} \le \iprodt{w}{\paren[\big]{\tfrac{1}{4}
    \Laplacian_G^*(Y)}} = \iprod{\tfrac{1}{4}\Laplacian_G(w)}{Y} \le
\iprod{\Diag(x)}{Y} = \mu \iprodt{\ones}{x} \le \rho \mu\), so
equality holds throughout.
In particular, we further have that \(\GW(G, w) \leq \rho\) and
\(\GW^{\polar}(G, z) \leq \mu\), so equality holds in both cases
by~\cref{eq:GW-CS}.
Thus, `\(\supseteq\)' is proved in~\cref{eq:GWOpt-def}.
To see the connection between~\cref{eq:GWOpt-def}
and~\cref{def:beta-cert}, note that in the Goemans--Williamson
approximation algorithm for the maximum cut problem and in our
approximation algorithm for the fractional cut-covering problem, a
crucial step is obtaining a matrix \(Y \in \Psd{V}\) that is feasible
in the semidefinite relaxation, which is then used in the sampling of
shores.
In this way, the matrix~\(Y\) encodes both the shore in~\cref{item:cert-2}
and the fractional cut~cover in~\cref{item:cert-3} featured in
\(\beta\)-certificates.

In this subsection, we present two main results.
The first result (stated in \Cref{prop:real-number-model-certificate})
is that there is an \(\GWalpha\)-certificate for every nonzero
\((w, z) \in \GWOpt(G)\); in~particular, by \cref{prop:2}, every
nonzero \((w,z) \in \GWOpt(G)\) is an \(\GWalpha\)-pairing.
The second result (stated in \Cref{prop:optz-optw-cones}) shows that,
given an instance \((G,w)\) of the maximum cut problem, we have that
\((w,z) \in \GWOpt(G)\) for \(z \in \Reals_+^E\) if and only if \(z\)
is in the convex cone generated by the optimal solutions of a certain
formulation for \(\GW(G, w)\).
Symmetrically, given an instance \((G,z\)) of the fractional
cut-covering problem, we~have that \((w,z) \in \GWOpt(G)\) for
\(w \in \Reals_+^E\) if and only if \(w\) is in the convex cone
generated by the optimal solutions of a certain formulation for
\(\GW^{\polar}(G, z)\).

Let \(G = (V, E)\) be a graph and let \(z \in \Reals_+^E\).
Goemans and Williamson's analysis~\cite{GoemansWilliamson1995} implies
that
\begin{equation}
  \label{eq:gap:GW-cut}
  \text{if \((\mu,Y)\) is feasible for~\cref{eq:GW-polar-def-intro}
    and \(\mu > 0\), then }
  \Ebb\paren[\big]{
    \iprodt{w}{\incidvector{\delta(\GWrv(Y))}}
  } \ge \frac{\GWalpha}{\mu} \iprod{\tfrac{1}{4} \Laplacian_G(w)}{Y}.
\end{equation}

Now we show the existence of an \(\GWalpha\)-certificate for every
nonzero pair \((w, z) \in \GWOpt(G)\).
\begin{proposition}
\label{prop:real-number-model-certificate}
Let \(G = (V, E)\) be a graph.
For every \((w, z) \in \GWOpt(G)\) such that \(w \neq 0 \neq z\),
there exists an \(\GWalpha\)-certificate \((\rho, \mu, S, y, x)\) for
\((w,z)\) such that \(\rho \neq 0 \neq \mu\).
In~particular, \((w,z)\) is an \(\GWalpha\)-pairing.
\end{proposition}
\begin{proof}
  Let \((w, z) \in \GWOpt(G)\) be such that \(w \neq 0 \neq z\).
  Let \((\rho, x)\) be an optimal solution for~\cref{eq:GW-gaugef-rho}
  and let \((\mu,Y)\) be an optimal solution
  for~\cref{eq:GW-polar-def-intro}, so that \(\rho = \GW(G, w) > 0\)
  and \(\mu = \GW^{\polar}(G, z) > 0\).
  Thus, \(\iprodt{w}{z} = \GW(G,w)\GW^{\polar}(G,z) = \rho\mu\), so
  \cref{item:cert-1} holds; we take \(\beta \coloneqq \GWalpha\)
  whenever referring to the items of~\cref{eq:beta-cert-items} in this
  proof.
  \Cref{item:cert-4} holds by the feasibility of \((\rho, x)\)
  for~\cref{eq:GW-gaugef-rho}.
  Define \(y \in \Reals^{\Powerset{V}}\) as in \cref{eq:6}.
  \Cref{item:cert-3} follows from \cref{prop:sdp-rounding}.
  Since by~\cref{eq:gap:GW-cut} one has
  \begin{equation*}
    \Ebb\paren[\big]{
      \iprodt{w}{\incidvector{\delta(\GWrv(Y))}} }
    \geq
    \frac{\GWalpha}{\mu} \iprod{\tfrac{1}{4}\Laplacian_G(w)}{Y}
    \geq
    \frac{\GWalpha}{\mu} \iprod{\Diag(x)}{Y}
    =
    \GWalpha \iprodt{\ones}{x}
    =
    \GWalpha \rho,
  \end{equation*}
  there exists \(S\subseteq V\) as in \cref{item:cert-2}.
  Thus, \((w,z)\) is an \(\GWalpha\)-pairing by \cref{prop:2}.
\end{proof}

The set \(\GWOpt(G)\) defines a relation between maximum cut and
fractional cut-covering instances.
Suppose the starting point is an instance of one of these problems and
one builds an instance for the other one so that the pair is in
\(\GWOpt(G)\).
This motivates the following definitions.
Define
\begin{alignat}{3}
  \label{eq:optz-def}
  \optz_G(w)
  &\coloneqq
  \setst{
    z \in \Reals_+^E
  }{
    (w, z) \in \GWOpt(G)
  }
  & \quad &
    \text{for every }
    w \in \Reals_+^E,
  \\*
  \label{eq:optw-def}
  \optw_G(z)
  &\coloneqq
  \setst{
    w \in \Reals_+^E
  }{
    (w, z) \in \GWOpt(G)
  }
  & \quad &
  \text{for every }
  z \in \Reals_+^E.
\end{alignat}
The upcoming \cref{prop:optz-optw-cones}, which describes the encoding
of the optimal solutions to our SDPs in \(\GWOpt(G)\), is~more
conveniently stated using a slight variant of~\cref{eq:GW-intro-pd}:
\begin{subequations}
  \label{eq:GW-explicit}
  \begin{align}
    \label{eq:GW-supf-explicit-z}
    \GW(G, w)
    &=
    \max\setst[\big]{
      \iprodt{w}{z}
    }{
      z \in \Reals_+^E,\,
      Y \in \Psd{V},\,
      \diag(Y) = \ones,\,
      z \le \tfrac{1}{4} \Laplacian_G^*(Y)
    }
    \\
    \label{eq:GW-gaugef-explicit-z}
    &=
    \min\setst[\big]{
      \rho
    }{
      \rho \in \Reals,\,
      x \in \Reals^V,\,
      \Diag(x) \succeq \tfrac{1}{4}\Laplacian_G(w),\,
      \rho \geq \iprodt{\ones}{x}
    }.
    \\
    \intertext{%
      Recall~\cref{eq:GW-polar-gaugef} and~\cref{eq:GW-polar-supf}:
    }
    \GW^{\polar}(G, z)
    &=
    \min\setst[\big]{
     \mu
    }{
      \mu \in \Reals_{+},\,
      Y \in \Psd{V},\,
      \diag(Y) = \mu \ones,\,
      \tfrac{1}{4} \Laplacian_G^*(Y) \ge z
    }
    \tag{\ref{eq:GW-polar-gaugef}}
    \\
    &\eqaligned
    \max\setst[\big]{
      \iprodt{z}{w}
    }{
      w \in \Reals_+^E,\,
      x \in \Reals^V,\,
      \tfrac{1}{4}\Laplacian_G(w) \preceq \Diag(x),\,
      \iprodt{\ones}{x} \le 1
    }.
    \tag{\ref{eq:GW-polar-supf}}
  \end{align}
\end{subequations}
Note the various symmetries relating these SDPs.
The constraint `\(\Diag(x) \succeq \tfrac{1}{4}\Laplacian_G(w)\)'
occurs in~\cref{eq:GW-gaugef-explicit-z,eq:GW-polar-supf} and
`\(\tfrac{1}{4}\Laplacian_G^*(Y) \geq z\)' occurs
in~\cref{eq:GW-supf-explicit-z,eq:GW-polar-gaugef}, although in each
constraint one of \(w\) or~\(z\) is a variable in one but not in the
other SDP.
Additionally, the constraint `\(\diag(Y) = \ones\)'
from~\cref{eq:GW-supf-explicit-z} appears homogenized in
\cref{eq:GW-polar-gaugef} with the variable~\(\mu\), whereas the
constraint `\(\iprodt{\ones}{x} \leq 1\)' from \cref{eq:GW-polar-supf}
appears homogenized in~\cref{eq:GW-gaugef-explicit-z} with the
variable~\(\rho\).

Let \(G = (V, E)\) be a graph.
We now relate \(\optz_G\) and \(\optw_G\) to optimal solutions
of~\cref{eq:GW-supf-explicit-z} and of~\cref{eq:GW-polar-supf}, resp.
For any finite set \(U\) and for any set \(S \subseteq \Reals^{U}\),
we~denote by \(\cone(S)\) the convex cone generated by~\(S\), i.e.,
the smallest convex cone containing~\(S\) and the origin.

\begin{proposition}
\label{prop:optz-optw-cones}

Let \(G = (V, E)\) be a graph.
For every \(w \in \Reals_+^E\),
\begin{align}
  \label{eq:optz-as-cone}
  \optz_G(w)
  &=
  \cone\paren[\big]{
    \setst{
      z \in \Reals_+^E
    }{
      \exists Y \in \Psd{V}
      \text{ s.t. \((z, Y)\) is optimal in~\cref{eq:GW-supf-explicit-z} for }
      \GW(G, w)
    }
  }.%
  \intertext{Similarly, for every \(z \in \Reals_+^E\),}
  \label{eq:optw-as-cone}
  \optw_G(z)
  &=
  \cone\paren[\big]{
    \setst{
      w \in \Reals_+^E
    }{
      \exists x \in \Reals^V
      \text{ s.t. }
      (w, x)
      \text{ is optimal in~\cref{eq:GW-polar-supf} for }
      \GW^{\polar}(G, z)
    }
  }.
\end{align}
\end{proposition}
\begin{proof}

We first prove~\cref{eq:optz-as-cone}.
Let \(w \in \Lp{E}\).
Note that
\begin{equation}
  \label{eq:20}
  \setst{z \in \Lp{E}}{\iprodt{w}{z} = \GW(G, w)\GW^{\polar}(G, z)}
  = \cone\paren{\setst{z \in \Lp{E}}{
    \GW^{\polar}(G, z) \le 1,\,
    \iprodt{w}{z} = \GW(G, w)
  }}.
\end{equation}
Indeed, `\(\subseteq\)' holds in~\cref{eq:20} since \(\GW^{\polar}(G,\cdot)\)
is positively homogeneous.
For the proof of `\(\supseteq\)', first note that the LHS is clearly
closed under positive scalar multiplication since
\(\GW^{\polar}(G,\cdot)\) is positively homogeneous.
To see that the LHS is a convex cone, let \(z_1,z_2\) be elements of
the LHS.
Sublinearity of \(\GW^{\polar}(G,\cdot)\) implies that
\begin{inlinemath}
  \iprodt{w}{(z_1+z_2)}
  =
  \GW(G,w)
  \paren[\big]{
    \GW^{\polar}(G,z_1)
    +
    \GW^{\polar}(G,z_2)
  }
  \geq
  \GW(G,w)
  \GW^{\polar}(G,z_1+z_2),
\end{inlinemath}
whence equality holds by~\cref{eq:GW-CS}, which proves that
\(z_1+z_2\) lies in the LHS.
Since
\(
  \iprodt{w}{z}
  = \GW(G, w)
  \ge \GW(G, w) \GW^{\polar}(G, z)
  \ge \iprodt{w}{z}
\)
for every \(z \in \Lp{E}\) in the RHS of~\cref{eq:20}, by~\cref{eq:GW-CS},
this concludes the proof of~\cref{eq:20}.
We can now prove~\cref{eq:optz-as-cone}:
\begin{detailedproof}
  \optz_G(w)
  &= \setst{z \in \Lp{E}}{(w, z) \in \GWOpt(G)}
  &&\text{by~\cref{eq:optz-def}}\\
  &= \setst{z \in \Lp{E}}{\iprodt{w}{z} = \GW(G, w)\GW^{\polar}(G, z)}
  &&\text{by~\cref{eq:GWOpt-as-tight-CS}}\\
  &= \cone\paren{\setst{z \in \Lp{E}}{
    \GW^{\polar}(G, z) \le 1,\,
    \iprodt{w}{z} = \GW(G, w)
  }}
  &&\text{by~\cref{eq:20}}\\
  &= \cone\paren{\setst{z \in \Lp{E}}{
    \exists Y \in \Psd{V},\,
    \diag(Y) = \ones,\,
    \tfrac{1}{4}\Laplacian_G^*(Y) \ge z,\,
    \iprodt{w}{z} = \GW(G, w)
  }}
  &&\text{by~\cref{eq:GW-polar-gaugef}}\\
  &= \cone\paren{\setst{z \in \Lp{E}}{
    \exists Y \in \Psd{V}
    \text{ s.t. \((z, Y)\) is optimal in~\cref{eq:GW-supf-explicit-z} for }
    \GW(G, w)
  }}.
\end{detailedproof}

For~\cref{eq:optw-as-cone}, let \(z \in \Lp{E}\).
One can prove, analogously to~\cref{eq:20}, that
\begin{equation}
  \label{eq:22}
  \setst{w \in \Lp{E}}{\iprodt{w}{z} = \GW(G, w)\GW^{\polar}(G, z)}
  = \cone\paren{\setst{w \in \Lp{E}}{
    \GW(G, w) \le 1,\,
    \iprodt{w}{z} = \GW^{\polar}(G, z)
  }}.
\end{equation}
Then
\begin{detailedproof}
  \optw_G(z)
  &= \setst{w \in \Lp{E}}{(w, z) \in \GWOpt(G)}
  &&\text{by~\cref{eq:optw-def}}\\*
  &= \setst{w \in \Lp{E}}{\iprodt{w}{z} = \GW(G, w)\GW^{\polar}(G, z)}
  &&\text{by~\cref{eq:GWOpt-as-tight-CS}}\\*
  &= \cone\paren{\setst{w \in \Lp{E}}{
    \GW(G, w) \le 1,\,
    \iprodt{w}{z} = \GW^{\polar}(G, z)
  }}&&\text{by~\cref{eq:22}}\\*
  &= \cone\paren{\setst{w \in \Lp{E}}{
    \exists x \in \Reals^V,\,
    1 \ge \iprodt{\ones}{x},\,
    \Diag(x) \succeq \tfrac{1}{4}\Laplacian_G(w),\,
    \iprodt{w}{z} = \GW^{\polar}(G, z)
  }}
  &&\text{by~\cref{eq:GW-gaugef-rho}}\\*
  &= \cone\paren{\setst{w \in \Lp{E}}{
    \exists x \in \Reals^V
    \text{ s.t. \((w, x)\) is optimal in~\cref{eq:GW-polar-supf} for }
    \GW(G, w)
  }}.&&\qedhere
\end{detailedproof}
\end{proof}

\begin{remark}

Let \(G = (V, E)\) be a graph.
\Cref{prop:optz-optw-cones} establishes a stronger result than
nonemptiness of \(\GWOpt(G)\), by showing that the projections
\(\optz_G(w)\) and \(\optw_G(z)\) of \(\GWOpt(G)\) are nontrivial
convex cones for every \(w,\, z \in \Lp{E}\).
These cones are tightly connected to normal cones of the relevant
convex corners.
Recall the sets \(\Ccal_{\GW(G, \cdot)}\) and
\(\Ccal_{\GW^{\polar}(G, \cdot)}\) defined
in~\cref{eq:GW-convex-corners}.
Fix a nonzero \(w \in \Lp{E}\), and set \(\rho \coloneqq \GW(G, w) > 0\).
Then \(\GW(G, \rho^{-1} w) \le 1\), so \(\rho^{-1} w \in
\Ccal_{\GW^{\polar}(G, \cdot)} = \abl(\Ccal_{\GW(G, \cdot)})\).
Then one can show that
\begin{detailedproof}
 \optz_G(w)
  &= \setst{z \in \Lp{E}}{
    \iprodt{w}{z} = \GW(G, w) \GW^{\polar}(G, z)
  }\\
  &= \setst{z \in \Lp{E}}{
    \iprodt{(\rho^{-1} w)}{z} = \GW^{\polar}(G, z)
  }\\
  &= \setst{z \in \Lp{E}}{
    z \in \Normal(\Ccal_{\GW^{\polar}(G, \cdot)}, \rho^{-1} w)
  },
\end{detailedproof}
where
\(\Normal(S, \bar{x}) \coloneqq \setst{c \in \Reals^E}{\forall x
  \in S,\,\iprodt{c}{\bar{x}} \geq \iprodt{c}{x}}\) denotes the
normal cone of \(S \subseteq \Reals^E\) at \(\bar{x} \in S\).
Thus
\begin{alignat*}{2}
  \optz_G(w)
  &= \Lp{E}
    \cap \Normal\paren[\big]{
      \Ccal_{\GW^{\polar}(G, \cdot)}, \GW(G,w)^{-1} w
    }
  & \qquad &
    \text{ for every }
    w \in \Lp{E} \setminus \set{0}.\\*
  \shortintertext{Dually,}
  \optw_G(z)
  &= \Lp{E}
    \cap \Normal\paren[\big]{
      \Ccal_{\GW(G, \cdot)}, \GW^{\polar}(G,z)^{-1} z
    }
  & \qquad &
    \text{ for every }
    z \in \Lp{E} \setminus \set{0}.
\end{alignat*}
\end{remark}

By combining \cref{prop:real-number-model-certificate,%
  prop:optz-optw-cones}, we show how to obtain
\(\GWalpha\)-certificates when the starting point is either a maximum
cut instance or a fractional cut-covering instance.
In \Cref{subsec:algorithimic_certicates}, we present algorithmic
versions of these results: from nearly optimal SDPs solutions and a
randomized sampling procedure, we~produce a heavy cut and a light
fractional cut cover.

\begin{theorem}
\label{thm:certificates-from-w}
Let \(G = (V, E)\) be a graph and let \(w \in \Reals_+^E\) be
nonzero.
Then there exist a nonzero \(z \in \optz_G(w)\) and an
\(\GWalpha\)-certificate for \((w,z)\).
\end{theorem}
\begin{proof}
  There exists an optimal solution \((\bar{z},\bar{Y})\)
  for~\cref{eq:GW-supf-explicit-z}.
  Thus, \((w, \bar{z}) \in \GWOpt(G)\) by \Cref{prop:optz-optw-cones}.
  As \(w \neq 0\), we have that \(\bar{z} \neq 0\).
  By \cref{prop:real-number-model-certificate}, there exists an
  \(\GWalpha\)-certificate for \((w,\bar{z})\).
\end{proof}

\begin{theorem}
\label{thm:certificates-from-z}
Let \(G = (V, E)\) be a graph and let \(z \in \Reals_+^E\) be
nonzero.
Then there exist a nonzero \(w \in \optw_G(z)\) and an
\(\GWalpha\)-certificate for \((w,z)\).
\end{theorem}
\begin{proof}
  There exists an optimal solution \((\bar{w},\bar{x})\)
  for~\cref{eq:GW-polar-supf}.
  Thus, \((\bar{w}, z) \in \GWOpt(G)\) by \Cref{prop:optz-optw-cones}.
  As \(z \neq 0\), we have that \(\bar{w} \neq 0\).
  By \cref{prop:real-number-model-certificate}, there exists an
  \(\GWalpha\)-certificate for \((\bar{w},z)\).
\end{proof}

\subsection{Algorithmically Obtaining \texorpdfstring{\(\beta\)}{β}-Certificates}
\label{subsec:algorithimic_certicates}

\Cref{thm:certificates-from-w,thm:certificates-from-z} are pleasantly
symmetric, both in statement and in proof.
However, they are not directly suitable for algorithmic use.
The first issue is in the definition of \(\GWOpt(G)\) itself: as
\cref{prop:optz-optw-cones} states, its elements arise as optimal
solutions to~SDPs, whereas in algorithms one must work with nearly
optimal solutions.
The second issue is that, in
\Cref{prop:real-number-model-certificate}, the cut \(\delta(S)\) and
the fractional cut cover \(y\) obtained have significant caveats: the
shore~\(S\) is not explicitly constructed and \(y\) may have
exponential support size.
In particular, the same issues arising at~\cref{sec:rounding}
resurface here again, requiring one to ``thicken'' edges and to settle
for sparse surrogates of the probability distribution described in
\Cref{prop:sdp-rounding}.

Rather than parameterizing our rounding procedure as we have done in
\cref{sec:rounding}, we parameterize \emph{our optimization problems},
thus obtaining a family of geometric objects to study.
Let \(G = (V, E)\) be a graph and let \(\eps \in \halfopen{0,1}\).
For each \(w \in \Reals_+^E\), define
\begin{samepage}
  \begin{subequations}
    \label{eq:GW-eps-def-outer}
    \begin{align}
      \label{eq:GW-eps-def}
      \GW_{\eps}(G, w)
      &\coloneqq (1 - \eps)\GW(G, w) + \tfrac{\eps}{2}\norm[1]{w}
      \\
      \label{eq:GW-eps-def-dual}
      &\eqaligned
      \min\setst[\big]{
        \rho
      }{
        \rho \in \Reals_{+},\,
        x \in \Reals^V,\,
        \rho \geq (1-\eps)\iprodt{\ones}{x} + \tfrac{\eps}{2}\iprodt{\ones}{w},\,
        \Diag(x) \succeq \tfrac{1}{4}\Laplacian_G(w)
      }.
    \end{align}
  \end{subequations}
  Then \(\GW_{\eps}(G,\cdot) \colon \Lp{E} \to \Reals_+\) is a
  positive definite monotone gauge, and by~\cref{eq:gauge-polar-def}
  its gauge dual can be written as
  \begin{subequations}
    \label{eq:GW-eps-polar-def}
    \begin{align}
      \label{eq:GW-eps-polar-supf}
      \GW_{\eps}^{\polar}(G, z)
      &= \max\setst{
        \iprodt{z}{w}
        }{
        w \in \Reals_+^E,\,
        x \in \Reals^V,\,
        \tfrac{1}{4}\Laplacian_G(w) \preceq \Diag(x),\,
        (1 - \eps)\iprodt{\ones}{x} +  \iprodt{\tfrac{\eps}{2}\ones}{w} \le 1
        }
      \\
      \label{eq:GW-eps-polar-gaugef}
      &= \min\setst{
        \mu
        }{
        \mu \in \Reals_+,\,
        Y \in \Sym{V},\,
        Y \succeq \eps \mu I,\,
        \tfrac{1}{4}\Laplacian_G^*(Y) \ge z,\,
        \diag(Y) = \mu \ones
        }
    \end{align}
  \end{subequations}
  for every \(z \in \Reals_+^{E}\).
  The equality in~\cref{eq:GW-eps-polar-gaugef} follows from SDP
  Strong Duality, as the relaxed Slater points we exhibited
  in~\cref{eq:GW-polar-def} remain relaxed Slater points
  in~\cref{eq:GW-eps-polar-def}.
  Let \(\sigma \in (0, 1)\) and set
  \begin{equation}
    \label{eq:GWOpt-approx-def}
    \GWOpt_{\eps, \sigma}(G)
    \coloneqq
    \setst*{
      (w, z)
      \in \Reals_+^E \times \Reals_+^E
    }{
      \begin{array}{@{} l @{} l @{} l @{}}
        \exists (\rho,x)
        & \text{ feasible for~\cref{eq:GW-eps-def-dual} }
        & \text{for }(G,w),
        \\[2pt]
        \exists (\mu,Y)
        & \text{ feasible for~\cref{eq:GW-eps-polar-gaugef} }
        & \text{for }(G,z),
        \\[2pt]
        \multicolumn{3}{c}{%
        \text{and } \iprodt{w}{z} \ge (1 - \sigma) \rho \mu
        }
      \end{array}
    }.
  \end{equation}
\end{samepage}%
As for \(\GWOpt(G)\), for each \((w,z) \in \GWOpt_{\eps, \sigma}(G)\),
we refer to pairs \((\rho,x)\) and \((\mu,Y)\) assumed to exist as in
the RHS of~\cref{eq:GWOpt-approx-def}, as~\emph{witnesses of the membership}
\((w,z) \in \GWOpt_{\eps, \sigma}(G)\).
Note the following approximate version of~\cref{eq:GWOpt-strong-duality}:
\begin{equation}
  \label{eq:GWOpt-Cauchy-Schwarz}
  \text{%
    if \((\rho, x)\) and \((\mu,Y)\) witness the membership \((w, z)
    \in \GWOpt_{\eps, \sigma}(G)\),
    then
  }
  (1-\sigma)\rho\mu
  \le
  \iprodt{w}{z}
  \le
  \rho\mu.
\end{equation}
This holds since
\begin{detailedproof}
  \iprodt{w}{z}
  \le
  \iprodt{w}{\paren[\big]{\tfrac{1}{4} \Laplacian_G^*(Y)}}
  &=
  \iprod{\tfrac{1}{4}\Laplacian_G(w)}{Y - \eps\mu I}
  +
  \iprod{\tfrac{1}{4}\Laplacian_G(w)}{\eps\mu I}
  \\
  &\le
  \iprod{\Diag(x)}{Y - \eps\mu I}
  +
  \tfrac{\eps\mu}{2} \iprodt{\ones}{w}
  =
  \paren{
    (1 - \eps) \iprodt{\ones}{x} + \tfrac{\eps}{2} \iprodt{\ones}{w}
  }\mu
  \le
  \rho\mu.
\end{detailedproof}
The analogue of the expression we took as definition of \(H(G)\)
in~\cref{eq:GWOpt-as-tight-CS} is
\begin{equation*}
  \GWOpt_{\eps,\sigma}(G)
  = \setst[\bigg]{
    (w, z) \in \Lp{E} \times \Lp{E}
  }{
    \abs[\bigg]{
      \frac{
        \GW_{\eps}(G,w)\GW_{\eps}^{\polar}(G, z) - \iprodt{w}{z}
      }{
        \GW_{\eps}(G,w)\GW_{\eps}^{\polar}(G, z)
      }
    } \le \sigma
  },
\end{equation*}
where the expression inside the absolute value is taken to be zero
whenever \(0 \in \set{w, z}\).

Let \(G = (V, E)\) be a graph, let \(w,z \in \Lp{E}\), and let
\(\eps \in \halfopen{0,1}\).
We show in \cref{thm:approx-GW-bound-conversion} below that
\(\GW_{\eps}(G, w)\) and \(\GW_{\eps}^{\polar}(G, z)\) are
approximations for \(\GW(G, w)\) and \(\GW^{\polar}(G, z)\),
respectively.
Before that, we state two important monotonicity properties of
the Laplacian of a graph \(G\):
\begin{subequations}
  \begin{alignat}{3}
    \label{eq:Laplacian-monotone}
    v \le w
    & \text{ implies }
    \Laplacian_G(v) \preceq \Laplacian_G(w),
    & \quad &
    \text{ for every } v, w \in \Reals^E,
    \\*
    \label{eq:Laplacian-adjoint-monotone}
    A \preceq B
    & \text{ implies }
    \Laplacian_G^*(A) \le \Laplacian_G^*(B),
    & \quad &
    \text{ for every } A, B \in \Sym{V}.
  \end{alignat}
\end{subequations}
Both follow from the fact that \(\Laplacian_G(e_{ij}) \succeq 0\) for
every \(ij \in E\).
This is immediate for~\cref{eq:Laplacian-monotone} by using the
definition of \(\Laplacian_G\).
For~\cref{eq:Laplacian-adjoint-monotone}, one has
\begin{equation*}
  (\Laplacian_G^*(A))_{ij}
  = \iprod{\Laplacian_G^*(A)}{e_{ij}}
  = \iprod{A}{\Laplacian_G(e_{ij})}
  \le \iprod{B}{\Laplacian_G(e_{ij})}
  = \iprod{\Laplacian_G^*(B)}{e_{ij}}
  = (\Laplacian_G^*(B))_{ij}
\end{equation*}
for every \(A, B \in \Sym{V}\) such that \(A \preceq B\), and for
every \(ij \in E\).
These results imply one of the motivating properties of
\(\GW_{\eps}^{\polar}\):
\begin{equation}
  \label{eq:gap:GW-eps-cover}
  \text{if \((\mu,Y)\) is feasible for~\cref{eq:GW-eps-polar-gaugef}
    and \(\mu > 0\), then }
  \prob(ij \in \GWrv(Y)) \ge \frac{\sqrt{2\eps}}{\pi}
  \text{ for every }
  ij \in E.
\end{equation}
Let \((\mu, Y)\) be feasible in~\cref{eq:GW-eps-polar-gaugef}
such that \(\mu > 0\), and let \(ij \in E\).
Then
\(
  \mu - Y_{ij}
  = \tfrac{1}{2}\Laplacian_G^*(Y)_{ij}
  \ge \tfrac{1}{2}\Laplacian_G^*(\eps \mu I)_{ij}
  = \eps\mu
\)
by~\cref{eq:Laplacian-adjoint-monotone}.
Thus \(\mu(1 - \eps) \ge Y_{ij}\).
Hence~\cref{eq:GW-Y-edge-marginal} and~\cref{eq:arccos-bound} imply
that
\[
  \prob\paren{ij \in \GWrv(Y)}
  = \frac{\arccos(\mu^{-1} Y_{ij})}{\pi}
  \ge \frac{\sqrt{2\eps}}{\pi},
\]
so~\cref{eq:gap:GW-eps-cover} holds.

\begin{theorem}
  \label{thm:approx-GW-bound-conversion}
  Let \(G = (V, E)\) be a graph, and let \(\eps \in \halfopen{0,1}\).
  Then, for each \(w,z \in \Reals_+^{E}\),
  \begin{subequations}
    \begin{alignat}{3}
      \label{eq:GW-eps-bound}
      (1 - \eps) \GW(G, w)
      &\le \GW_{\eps}(G, w)
      &&\le \GW(G, w),
      &
      \mathrlap{
        \qquad
        \text{for each \(w \in \Reals_+^{E}\)},
      }
      \\
      \label{eq:GW-eps-polar-bound}
      \GW^{\polar}(G, z)
      &\le \GW_{\eps}^{\polar}(G, z)
      &&\le \frac{1}{1 - \eps} \GW^{\polar}(G, z),
      &
      \mathrlap{
        \qquad
        \text{for each \(z \in \Reals_+^{E}\)}.
      }
    \end{alignat}
  \end{subequations}
\end{theorem}
\begin{proof}

As the identity matrix is feasible in~\cref{eq:GW-intro-def}, the
rightmost inequality below holds:
\begin{equation*}
  (1 - \eps) \GW(G, w)
  \le (1 - \eps) \GW(G, w) + \tfrac{\eps}{2} \norm[1]{w}
  = (1 - \eps) \GW(G, w) + \eps\iprod{\tfrac{1}{4}\Laplacian_G(w)}{I}
  \le \GW(G, w).
\end{equation*}
Thus~\cref{eq:GW-eps-bound} follows from~\cref{eq:GW-eps-def}.
Now \cref{eq:GW-eps-polar-bound} follows from~\cref{eq:GW-eps-bound}
and \cref{prop:bound-conversion}.
Alternatively, one can prove~\cref{eq:GW-eps-polar-bound} directly.
As the feasible region of~\cref{eq:GW-polar-def-intro} contains the
feasible region of~\cref{eq:GW-eps-polar-gaugef}, the first inequality
in~\cref{eq:GW-eps-polar-bound} holds.
Let \((\mu,Y)\) be feasible for~\cref{eq:GW-polar-def-intro}.
Set \(\mu_{\eps} \coloneqq \mu/(1-\eps)\) and
\(Y_{\eps} \coloneqq Y + \eps \mu_{\eps} I \succeq \eps\mu_{\eps} I\).
Then
\(\diag(Y_{\eps}) = \mu \ones + \eps\mu_{\eps}\ones =
\mu_{\eps}\ones\).
Moreover, from~\cref{eq:Laplacian-adjoint-monotone} and
\(Y_{\eps} \succeq Y\) we obtain
\(\tfrac{1}{4}\Laplacian_G^*(Y_{\eps}) \ge
\tfrac{1}{4}\Laplacian_G^*(Y) \ge z\).
Thus \((\mu_{\eps},Y_{\eps})\) is feasible
for~\cref{eq:GW-eps-polar-gaugef} with objective value \(\mu_{\eps}\).
Hence~\cref{eq:GW-eps-polar-bound} holds.
\end{proof}

Let \(T \in \Naturals \setminus \set{0}\) and let \(\Chernoff \in (0,1)\).
We use the shore sampling procedure~\cref{eq:Chernoff-sampling-def}
defined in \Cref{sec:rounding}.
Let \(G = (V, E)\) be a graph, let \(\mu \in \Reals_+\), and let
\(Y \in \mu \Ecal_V\).
Let \(S_1, \dotsc, S_T \subseteq V\) be independent
identically-distributed random shores sampled by \(\GWrv(Y)\).
Recall the definition in \cref{eq:Chernoff-sampling-def}:
\begin{equation}
  \tag{\ref{eq:Chernoff-sampling-def}}
  \Acal_{T, \Chernoff}(G, Y) = (\Fcal, y),
  \text{ where }
  \Fcal = \set{S_1, \dotsc, S_T} \subseteq \Powerset{V}
  \text{ and }
  y =
  \frac{\mu}{(1 - \Chernoff) \GWalpha} \frac{1}{T} \sum_{t = 1}^T
  e_{S_t} \in \Reals_+^{\Fcal}.
\end{equation}
A shore of a sampled cut with largest weight is
\begin{inlinemath}
  \argmax \setst[\big]{
    \iprodt{w}{\incidvector{\delta(S)}}
  }{
    S \in \Fcal
  }.
\end{inlinemath}
These objects give rise to the sampling procedure in
\Cref{alg:sampling}, which we analyze next.
Similar to \cref{alg:1}, the pseudocode of \Cref{alg:sampling}
abstracts away important implementation choices, including the choice
of data structures.

\begin{algorithm}
  \caption{Certification procedure}
  \label{alg:sampling}
  \begin{algorithmic}[1]
    \algrenewcommand\algorithmicrequire{\textbf{Parameters:}}
    \Require a constant approximation factor \(\beta \in (0, \GWalpha)\)
    parameterizes the algorithm \textsc{Certify${}_{\beta}$}.
    As~in~\cref{eq:sampling-GWOpt-constants-1,%
      eq:sampling-GWOpt-constants-2}, define the following constants in
    terms of \(\beta\):
    \begin{equation*}
      \tau \coloneqq 1 - \dfrac{\beta}{\GWalpha} \in (0,1),
      \quad
      \sigma \coloneqq \frac{2}{3}\tau
      \in (0, 2/3),
      \quad
      \eps \coloneqq \frac{\tau}{3(3 - 2\tau)} \in (0,1/3),
      \quad
      \text{and}
      \quad
      \Chernoff \coloneqq \frac{2\tau}{9-7\tau} \in (0,1).
    \end{equation*}
    \noindent
    \algrenewcommand\algorithmicrequire{\textbf{Input:}} \Require a
    graph \(G = (V,E)\), a pair \((w,z) \in \Reals_+^{E}
    \times \Reals_+^{E}\) of nonzero edge weights, and witnesses
    \((\bar{\rho},x)\) and \((\bar{\mu},Y)\) of the membership of
    \((w,z)\) in \(\GWOpt_{\eps, \sigma}(G)\).
    \algrenewcommand\algorithmicensure{\textbf{Output:}}
    \Ensure\Call{Certify${}_{\beta}$}{$G, (w,z), (\bar{\rho},x),
      (\bar{\mu},Y)$} returns a \(\beta\)-certificate
    \((\rho, \mu, S, y, x)\) with high probability, where the support
    of \(y\) has size bounded above by
    \(T \coloneqq
    \ceil[\big]{\frac{2187\pi}{2\GWalpha^2\tau^3}\ln(\card{V})}\), as
    in~\cref{eq:23}.
    \Procedure{Certify${}_{\beta}$}{$G, (w,z), (\bar{\rho},x), (\bar{\mu},Y)$}
    \State \(\Fcal \gets \varnothing\)
    \State \(\bar{y} \gets 0 \in \Lp{\Powerset{V}}\)
    \CountUp{\(T\)}
    \State\label{step:GW-sampling-S}%
    \(S \gets \GWrv(Y)\)
    \Comment Sample a shore \(S \subseteq V\) via the random hyperplane technique
    \State \(\Fcal \gets \Fcal \cup \{S\}\)
    \State \(\bar{y}_S \gets \bar{y}_S + 1\)
    \EndCountUp
    \State \(S^{\max} \gets \arg\max \setst[\big]{\iprodt{w}{\incidvector{\delta(S)}}}{S \in \Fcal}\)
    \State
    \(y \gets \frac{\mu}{(1 - \Chernoff) \GWalpha} \frac{1}{T} \bar{y}\)
    \State \(\rho \gets (1-\eps)^{-1}\bar{\rho}\)
    \State \(\mu \gets \rho^{-1} \iprodt{w}{z}\)
    \State \textbf{return} \(\paren{\rho,\mu,S^{\max},y,x}\)
    \EndProcedure
  \end{algorithmic}
\end{algorithm}

\begin{proposition}
\label{prop:GWOpt-approx-sampling}
Let \(\eps, \sigma, \Chernoff \in (0, 1)\).
Let \(G = (V, E)\) be a graph on \(n\) vertices and let
\((w, z) \in \GWOpt_{\eps, \sigma}(G)\) be such that
\(w \neq 0 \neq z\).
Set \(\beta \coloneqq \GWalpha(1 - \Chernoff)(1 - \sigma) (1-\eps)\).
Let \((\bar{\rho},x)\) and \((\bar{\mu},Y)\) witness the membership
\((w, z) \in \GWOpt_{\eps, \sigma}(G)\).
For each integer
\begin{equation*}
  T \geq \ceil[\Big]{
    \frac{6\pi}{
      \paren{\GWalpha\Chernoff(1 - \sigma)(1 - \eps)}^2\eps
    }
    \ln(n)
  },
\end{equation*}
the randomized polynomial-time procedure \(\Acal_{T, \Chernoff}(G,
Y)\) satisfies the following: with probability at least \(1-2/n\),
we have that \((\rho, \mu, S^{\max}, y, x)\) is a
\(\beta\)-certificate for \((w, z)\), where
\begin{equation}
  \label{eq:GWOpt-approx-sampling}
\begin{gathered}
  \rho \coloneqq (1-\eps)^{-1}\bar{\rho},
  \quad
  \mu \coloneqq {\rho}^{-1} \iprodt{w}{z},
  \\
  (\Fcal, y) \coloneqq \Acal_{T, \Chernoff}(G, Y),
  \quad
  \text{ and }
  \quad
  S^{\max} \coloneqq \argmax \setst[\big]{
    \iprodt{w}{\incidvector{\delta(S)}}
  }{
    S \in \Fcal
  }.
\end{gathered}
\end{equation}
In~particular, \((w,z)\) is a \(\beta\)-pairing.
\end{proposition}

\begin{remark}
  Alternatively, rather than sampling from \(\GWrv(Y)\) in the call to
  \(\Acal_{T, \Chernoff}\) in \cref{prop:GWOpt-approx-sampling}, using
  the definition of \(\GW_{\eps}\) as a starting point, one may use a
  perturbed sampling \(\GWrv_{\eps}(Y)\) obtained by sampling from
  \(\GWrv(Y)\) with probability \((1-\eps)\), and by sampling
  uniformly among all shores with probability~\(\eps\).
\end{remark}

\begin{proof}[Proof of \cref{prop:GWOpt-approx-sampling}]
  Let \((w, z) \in \GWOpt_{\eps, \sigma}(G)\).
  Let \((\bar{\rho},x)\) and \((\bar{\mu}, Y)\) witness the membership
  \((w, z) \in \GWOpt_{\eps, \sigma}(G)\).
  Note that \(\bar{\rho}, \bar{\mu} > 0\) as \(w \neq 0 \neq z\).
  Set~\(\rho \coloneqq (1-\eps)^{-1} \bar{\rho}\) and
  \(\mu \coloneqq (1/\rho) \iprodt{w}{z}\).
  Item \cref{item:cert-1} in the \cref{def:beta-cert} of
  \(\beta\)-certificates holds trivially.
  We~also have \cref{item:cert-4}, since
  \(\Diag(x) \succeq \tfrac{1}{4}\Laplacian_G(w)\) and
\begin{equation*}
  \rho = \frac{\bar{\rho}}{1-\eps}
  \geq\frac{1}{1-\eps} \paren[\Big]{
    (1-\eps)\iprodt{\ones}{x} +
    \frac{\eps}{2}\norm[1]{w}}
  \geq \iprodt{\ones}{x}.
\end{equation*}
In particular, \((\rho, x)\) is feasible in~\cref{eq:GW-gaugef-rho} so
\begin{equation}
  \label{eq:upperbound_cut_rho}
  \mc(G, w)
  \le \GW(G, w)
  \le \rho.
\end{equation}
Next we prove \cref{item:cert-2} and \cref{item:cert-3}.

Let \(S_1,\ldots, S_T\) and \((\Fcal,y)\) be defined as
in~\cref{eq:Chernoff-sampling-def}, so that
\begin{inlinemath}
  S^{\max}
  =
  \argmax \setst[\big]{
    \iprodt{w}{\incidvector{\delta(S)}}
  }{
    S \in \Fcal
  }.
\end{inlinemath}
We will now prove~\cref{item:cert-2} for the shore \(S^{\max}\).
More precisely, we show that
\(\iprodt{w}{\incidvector{\delta(S^{\max})}} \ge \beta \rho\) with
probability at least \(1-1/n\).
As \((\bar{\mu}, Y)\) is feasible in~\cref{eq:GW-eps-polar-gaugef}, it
is feasible in~\cref{eq:GW-polar-def-intro}, so from~\cref{eq:gap:GW-cut}
we have that
\begin{equation*}
  \Ebb\paren[\big]{
    \iprodt{w}{\incidvector{\delta(\GWrv(Y))}}
  }
  \ge \frac{\GWalpha}{\bar{\mu}}\iprod{\tfrac{1}{4}\Laplacian_G(w)}{Y}
  \ge \GWalpha\frac{\iprodt{w}{z}}{\bar{\mu}}
  \geq \GWalpha(1-\sigma)\bar{\rho}
  = \GWalpha(1-\sigma) (1-\eps) \rho,
\end{equation*}
and by \cref{eq:upperbound_cut_rho}, we can bound the range of
\(\iprodt{w}{\incidvector{\delta(\GWrv(Y))}}\) as \(0 \le
\iprodt{w}{\incidvector{\delta(\GWrv(Y))}} \le \rho\).
Define \(X_t \coloneqq \iprodt{w}{\incidvector{\delta(S_t)}}\) for
every \(t\in[T]\).
Define \(S \coloneqq \sum_{t=1}^T X_t\).
The random variables \(X_1,\dotsc, X_T\) are independent and satisfy
\(0\leq X_t\leq \rho\) and \(\Ebb\paren{X_t}\geq
\GWalpha(1-\sigma)(1-\eps)\rho\) for each \(t\in[T]\).
Using Hoeffding's inequality,
\begin{multline*}
  \prob\paren[\Big]{
    X_t \leq (1-\Chernoff)\Ebb\paren{X_t}
    \text{ for all }t\in[T]}
  \leq
  \prob\paren[\Big]{S \leq (1-\Chernoff)\Ebb\paren{S}}
  \leq
  \exp\paren[\bigg]{\frac{-\Chernoff^2 \Ebb\paren{S}^2}
    {T \rho^2}}
  \\
  \leq
  \exp\paren[\bigg]{\frac{-\Chernoff^2
      (\GWalpha(1-\sigma)(1-\eps)\rho T)^2}
    {T\rho^2}}
  =
  \exp\paren[\bigg]{
    -\paren[\big]{\GWalpha\Chernoff(1-\sigma)(1-\eps)}^2 T
  }
  \leq 1/n,
\end{multline*}
since
\begin{equation*}
  T \ge \ceil[\bigg]{\frac{
    \ln(n)
  }{
    \paren{\GWalpha\Chernoff(1 - \sigma)(1 - \eps)}^2
  }}.
\end{equation*}
Thus \(\iprodt{w}{\incidvector{\delta(S^{\max})}} \ge \GWalpha(1 -
\Chernoff)(1 - \sigma)(1 - \eps)\rho = \beta\rho\) with probability
at least \(1 - 1/n\), so the proof of \cref{item:cert-2} is complete.

We now prove \cref{item:cert-3}.
By \cref{eq:Chernoff-sampling-properties},
\begin{equation}
  \label{eq:7}
  \iprodt{\ones}{y}
  =
  \frac{1}{(1 - \Chernoff) \GWalpha} \bar{\mu}
  =
  \frac{1}{\beta} (1-\eps)(1-\sigma)\bar{\mu}
  \leq
  \frac{1}{\beta} (1-\eps)
  \frac{\iprodt{w}{z}}{\bar{\rho}}
  =
  \frac{1}{\beta}\frac{\iprodt{w}{z}}{\rho}
  =
  \frac{1}{\beta}\mu.
\end{equation}
Next we will use \cref{prop:probabilistic-rounding-cover-pre} to show
that \(\sum_{S \in \Fcal} y_S \incidvector{\delta(S)} \ge z\) with
probability at least \(1-1/n\).
Set \(\hat{z} \coloneqq \tfrac{1}{4}\Laplacian_G^*(Y)\).
Set \(\kappa \coloneqq 2/\eps\) and \(\xi \coloneqq \eps/2\).
We claim that the hypotheses from
\cref{prop:probabilistic-rounding-cover-pre} for \(\hat{z}\) are met
by~\((\bar{\mu}, Y)\) as well as \(\xi\), \(\kappa\), and~\(\Chernoff\).
It is immediate that \((\bar{\mu}, Y)\) is feasible
in~\cref{eq:GW-polar-def-intro} for \((G,\hat{z})\).
We claim that
\begin{equation}
  \label{eq:GWOpt-approx-sampling-1}
  \frac{\eps}{2} \bar{\mu}
  \le \hat{z}_{ij}
  \le \bar{\mu}
  \mathrlap{
    \qquad
    \text{for every }
    ij \in E.
  }
\end{equation}
By using~\cref{eq:Laplacian-adjoint-monotone},
\cref{eq:GW-polar-norm-infty-lowerbound},
\cref{eq:GW-eps-polar-bound}, and that feasibility of
\((\bar{\mu}, Y)\) in~\cref{eq:GW-eps-polar-gaugef} for \((G,\hat{z})\)
implies that \(\GW_{\eps}^{\polar}(G, \hat{z}) \le \bar{\mu}\), we see
that
\begin{equation*}
  \frac{\eps}{2}\bar{\mu} \ones
  = \tfrac{1}{4} \Laplacian_G^*(\eps\bar{\mu} I)\\
  \le \tfrac{1}{4} \Laplacian_G^*(Y)
  = \hat{z}
  \le \norm[\infty]{\hat{z}} \ones
  \le \GW^{\polar}(G, \hat{z}) \ones
  \le \GW_{\eps}^{\polar}(G, \hat{z}) \ones
  \le \bar{\mu} \ones.
\end{equation*}
Thus~\cref{eq:GWOpt-approx-sampling-1} holds.
Condition \cref{eq:prop-round-pre-2} of
\cref{prop:probabilistic-rounding-cover-pre} holds by the first
inequality in~\cref{eq:GWOpt-approx-sampling-1};
condition \cref{eq:prop-round-pre-3} follows from noting that
\(\tfrac{\eps}{2}\norm[\infty]{\hat{z}} \le \frac{\eps}{2} \bar{\mu} \le
\hat{z}_{ij}\) for every \(ij \in E\).
Hence, with probability at least \(1 - 1/n\), \((\Fcal, y) \coloneqq
\Acal_{T,\Chernoff}(G, Y)\) covers \({z}\) --- by covering
\(\hat{z} \ge z\).
Thus~\cref{eq:7} proves~\cref{item:cert-3}, as
\begin{equation*}
  T
  \geq
  \ceil[\bigg]{
    \frac{6\pi}{\Chernoff^2\eps}
    \ln(n)
  }
  =
  \ceil[\bigg]{
    \frac{3\pi}{\Chernoff^2}
    \paren[\bigg]{
      \frac{\kappa}{\xi}
    }^{\half}
    \ln(n)
  }.
\end{equation*}
Hence, the probability that \((\rho, \mu, S^{\max}, y, x)\) is not a
\(\beta\)-certificate for \((w,z)\) is at most \(2/n\).
\end{proof}

\begin{proposition}
\label{prop:sampling-GWOpt}
Let \(\beta \in (0, \GWalpha)\) and set
\begin{equation}
  \label{eq:sampling-GWOpt-constants-1}
  \tau \coloneqq 1 - \dfrac{\beta}{\GWalpha} \in (0, 1),
  \qquad
  \sigma \coloneqq \frac{2}{3}\tau
  \in (0,2/3),
  \quad
  \text{and}
  \quad
  \eps \coloneqq
  \frac{\tau - \sigma}{3(1 - \sigma)}
  = \frac{\tau}{3(3 - 2\tau)}
  \in (0, 1/3).
\end{equation}
There exists a polynomial-time randomized algorithm that takes as
input a graph \(G = (V, E)\) on \(n\) vertices, a~nonzero pair
\((w, z) \in \GWOpt_{\eps, \sigma}(G)\), and objects
\((\bar{\rho}, x)\) and \((\bar{\mu}, Y)\) witnessing the membership
\((w, z) \in \GWOpt_{\eps,\sigma}(G)\), and, with probability at least
\(1-2/n\), outputs a \(\beta\)-certificate
\((\rho, \mu, S^{\max}, y, x)\) for~\((w,z)\), where
\begin{equation}
  \label{eq:23}
  \card{\supp(y)}
  \le T
  \coloneqq
  \ceil[\bigg]{
    \frac{2187 \pi}{2\GWalpha^2\tau^3} \ln(n)
  }.
\end{equation}
This algorithm may be implemented so that it makes:
\begin{enumerate}
\item a single call to an oracle producing a Cholesky factorization of
  an \(n \times n\) matrix;
\item \(T n\) calls to an oracle sampling from a standard Gaussian
  distribution; and
\item \(O(Tn^2)\) extra work.
\end{enumerate}
\end{proposition}
\begin{proof}
Set
\begin{equation}
  \label{eq:sampling-GWOpt-constants-2}
  \Chernoff
  \coloneqq \frac{2\eps}{1 - \eps}
  = \frac{2\tau}{9-7\tau}
  \in (0,1).
\end{equation}
The constants \(\sigma\), \(\eps\), and \(\Chernoff\) are defined so
that
\begin{equation*}
  (1 - \Chernoff)(1 - \sigma) (1-\eps)
  = \paren[\bigg]{
    1 - \frac{2\eps}{1 - \eps}
  }
  (1 - \eps)
  (1 - \sigma)
  = (1 - 3\eps)(1 - \sigma)
    = \paren[\bigg]{
    1 - \frac{\tau - \sigma}{1 - \sigma}
  }(1 - \sigma)
  = 1 - \tau.
\end{equation*}
Thus \(\beta = \GWalpha(1 - \tau) = \GWalpha(1 - \Chernoff)(1 -
\sigma)(1 - \eps)\).
Moreover,
\begin{equation}
  \label{eq:sampling-GWOpt-1}
  \paren{(1 - \eps)\Chernoff}^2\eps
  = (2\eps)^2\eps
  = 4 \eps^3
  = \frac{4}{27}\paren[\bigg]{
    \frac{\tau - \sigma}{1 - \sigma}
  }^3.
\end{equation}
The proof is then an application of \cref{prop:GWOpt-approx-sampling}, since
\begin{detailedproof}
  \ceil[\bigg]{
    \frac{6\pi}{
      \paren{\GWalpha\Chernoff(1 - \sigma)(1 - \eps)}^2\eps
    }
    \ln(n)
  }
  &=
  \ceil[\bigg]{
    \frac{6\pi}{
      \paren{\GWalpha(1 - \sigma)}^2
    }
    \frac{27}{4}\paren[\bigg]{
      \frac{1 - \sigma}{\tau - \sigma}
    }^3
    \ln(n)
  }
  &&\text{by~\cref{eq:sampling-GWOpt-1}}\\
  &=
  \ceil[\bigg]{
    \frac{81\pi}{2 \GWalpha^2}
    \frac{1 - \sigma}{(\tau - \sigma)^3}
    \ln(n)
  }\\
  &\le
  \ceil[\bigg]{
    \frac{2187\pi}{2\GWalpha^2}
    \frac{1}{\tau^3}
    \ln(n)
  }
  &&\text{since \(\tau - \sigma = \tfrac{1}{3}\tau\) and \(\sigma > 0\)}\\
  &= T.
\end{detailedproof}

Let \(B \in \Reals^{[n] \times V}\) be a Cholesky factorization of
\(Y\), i.e., such that \(Y = B^\transp B\).
Let \(\Omega \in \Reals^{[T] \times [n]}\) have each entry independently
sampled from a standard Gaussian distribution.
\Cref{alg:sampling} may be easily implemented by computing the matrix
product \(X \coloneqq \Omega B\) and checking the signs of the entries
of \(X\).
This matrix product can be computed in \(O(Tn^2)\) time.
Each row of this matrix defines a shore of a cut, and one must keep
track of which shores have appeared.
Simply storing a list of the vertices in each shore allows one to
compute the fractional cut cover \(y\) in~\cref{alg:sampling} in
\(O(T^2 n)\) time.
Finally, computing  \(S^{\max}\) costs \(O(Tm)\).
Since \(T = O(\log n)\), the matrix multiplication cost \(O(Tn^2)\)
dominates.
\end{proof}

\begin{algorithm}
  \caption{Primal-dual randomized rounding approximation algorithm for
    $\mc$, building an instance for $\fcc$ and certificates}
  \label{alg:mc_fcc}
  \begin{algorithmic}[1]
    \algrenewcommand\algorithmicrequire{\textbf{Parameters:}}
    \Require a constant approximation factor \(\beta \in (0, \GWalpha)\)
    parameterizes the algorithm \textsc{ApproxMcFcc${}_{\beta}$}.
    Define the following constants in terms of~\(\beta\):
    \begin{equation*}
      \tau \coloneqq 1 - \dfrac{\beta}{\GWalpha},
      \quad
      \sigma \coloneqq \frac{2}{3}\tau,
      \quad
      \eps \coloneqq \frac{\tau}{3(3 - 2\tau)},
      \quad
      \Chernoff \coloneqq \frac{2\tau}{9-7\tau},
      \quad
      \text{and}
      \quad
      C \coloneqq \frac{2187\pi}{2\GWalpha^2\tau^3}.
    \end{equation*}
    \noindent
    \algrenewcommand\algorithmicrequire{\textbf{Input:}}
    \Require a graph $G = (V,E)$ and nonzero edge weights $w \in \Lp{E}$.
    \algrenewcommand\algorithmicensure{\textbf{Output:}}
    \Ensure \Call{ApproxMcFcc${}_{\beta}$}{$G, w$} returns a nonzero \(z \in
    \Lp{E}\) and a \(\beta\)-certificate \((\rho, \mu, S, y, x)\)
    for \((w, z)\) with probability at least \(1-2/\card{V}\), such
    that \(\card{\supp(y)} \le \ceil{C \ln(\card{V})}\).
    The algorithm runs in strongly polynomial time.
    \Procedure{ApproxMcFcc${}_{\beta}$}{$G, w$}
      \State Find a feasible pair $\paren[\big]{\tilde{Y}, (\tilde{\rho}, x)}$
      for~\cref{eq:GW-intro-pd} such that
      $\tilde{\rho} \le \iprod{\tfrac{1}{4}\Laplacian_G(w)}{\tilde{Y}}
      + \sigma\norm[\infty]{w}$.
      \State $\bar{\mu} \gets 1$
      \State $Y \gets (1 - \eps) \tilde{Y} + \eps I$
      \State $z \gets \tfrac{1}{4}\Laplacian_G^*(Y)$
      \State $\bar{\rho} \gets (1 - \eps)\iprodt{\ones}{x} + \tfrac{\eps}{2}\norm[1]{w}$
      \Comment Solve \(\GW_{\eps}(G, w)\) within \(\sigma\norm[\infty]{w}\)
      \State \textbf{return} $\big(z, $ \Call{Certify${}_{\beta}$}{$G, (w, z), (\bar{\rho}, x), (\bar{\mu}, Y)$}$\big)$
    \EndProcedure
  \end{algorithmic}
\end{algorithm}

\begin{algorithm}
  \caption{Primal-dual randomized rounding approximation algorithm for
    $\fcc$, building an instance for $\mc$ and certificates for both instances}
  \label{alg:fcc_mc}
  \begin{algorithmic}[1]
    \algrenewcommand\algorithmicrequire{\textbf{Parameters:}}
    \Require a constant approximation factor \(\beta \in (0, \GWalpha)\)
    parameterizes the algorithm \textsc{ApproxFccMc${}_{\beta}$}.
    Define the following constants in terms of~\(\beta\):
    \begin{equation*}
      \tau \coloneqq 1 - \dfrac{\beta}{\GWalpha},
      \quad
      \sigma \coloneqq \frac{2}{3}\tau,
      \quad
      \eps \coloneqq \frac{\tau}{3(3 - 2\tau)},
      \quad
      \Chernoff \coloneqq \frac{2\tau}{9-7\tau},
      \quad
      \text{and}
      \quad
      C \coloneqq \frac{2187\pi}{2\GWalpha^2\tau^3}.
    \end{equation*}
    \noindent
    \algrenewcommand\algorithmicrequire{\textbf{Input:}}
    \Require a graph $G = (V,E)$ and nonzero edge weights $z \in \Lp{E}$.
    \algrenewcommand\algorithmicensure{\textbf{Output:}}
    \Ensure \Call{ApproxFccMc${}_{\beta}$}{$G, z$} returns a nonzero \(w \in
    \Lp{E}\) and a \(\beta\)-certificate \((\rho, \mu, S, y, x)\)
    for \((w, z)\) with probability at least \(1-2/\card{V}\), such
    that \(\card{\supp(y)} \le \ceil{C \ln(\card{V})}\).
    The algorithm runs in strongly polynomial time.
    \Procedure{ApproxFccMc${}_{\beta}$}{$G, z$}
    \State Find a feasible pair $(w,x)$, $(\bar{\mu}, Y)$
    for~\cref{eq:GW-eps-polar-def} {s.t.} $\bar{\mu} \le \iprodt{z}{w} +
    \sigma\norm[\infty]{z}$.
    \Comment Solve $\GW^{\polar}_{\eps}(G, z)$ within \(\sigma\norm[\infty]{z}\)
    \State \textbf{return} $\big(w, $ \Call{Certify${}_{\beta}$}{$G, (w, z), (1, x), (\bar{\mu}, Y)$}$\big)$
  \EndProcedure
  \end{algorithmic}
\end{algorithm}

\begin{theorem}
\label{thm:certificates-from-w-approx}
Let \(\beta \in (0, \GWalpha)\).
Set \(\tau \coloneqq 1- \beta/\GWalpha\) and \(C \coloneqq
2187\pi/(2\GWalpha^2\tau^3)\).
There exists a randomized polynomial-time algorithm that, given a
graph \(G = (V, E)\) on \(n\) vertices and a nonzero \(w \in \Lp{E}\),
computes with probability at least \(1-2/n\) a nonzero vector \(z \in
\Lp{E}\) and a \(\beta\)-certificate \((\rho, \mu, S, y, x)\) for
\((w,z)\) such that the support of the fractional cut cover \(y\) has
size at most \(\ceil[\big]{C\ln(n)}\).
\end{theorem}
\begin{proof}
  Let \(w \in \Lp{E}\) be nonzero.
  Define \(\sigma \coloneqq \frac{2}{3}\tau\) and \(\eps \coloneqq
  \tau/(9 - 6\tau)\) as in \Cref{prop:sampling-GWOpt}.
  We follow \cref{alg:mc_fcc}.
  By (nearly) solving the primal-dual SDPs in~\cref{eq:GW-intro-pd},
  one can compute in polynomial time a feasible solution \(\tilde{Y}\)
  for~\cref{eq:GW-intro-def} and a feasible solution
  \((\tilde{\rho}, x)\) for~\cref{eq:GW-gaugef-rho} such that
  \begin{equation*}
    \tilde{\rho}
    \le
    \iprod{\tfrac{1}{4}\Laplacian_G(w)}{\tilde{Y}} + \sigma \norm[\infty]{w}.
  \end{equation*}
  Set
  \begin{equation}
    \label{eq:10}
    Y
    \coloneqq
    (1 - \eps) \tilde{Y} + \eps I,
    \qquad
    z \coloneqq \tfrac{1}{4}\Laplacian_G^*(Y),
    \qquad
    \bar{\rho}
    \coloneqq
    (1 - \eps) \iprodt{\ones}{x} + \tfrac{\eps}{2} \norm[1]{w}.
  \end{equation}
  We claim that
  \begin{equation}
    \label{eq:8}
    \text{%
      \((\bar{\rho}, x)\) and \((1, Y)\)
      witness the membership \((w, z) \in \GWOpt_{\eps, \sigma}(G)\).
    }
  \end{equation}
  It is immediate that \((\bar{\rho}, x)\) is feasible
  in~\cref{eq:GW-eps-def-dual}, whereas \(Y \succeq \eps I\) implies
  \((1, Y)\) is feasible in~\cref{eq:GW-eps-polar-gaugef}.
  Since \(\norm[\infty]{w} \le \GW(G, w)\), we can use
  \cref{eq:GW-eps-bound} and~\cref{eq:GW-eps-def-dual} to see
\begin{equation}
  \label{eq:certificates-from-w-approx}
  (1 - \eps) \norm[\infty]{w}
  \le (1 - \eps) \GW(G, w)
  \le \GW_{\eps}(G, w)
  \le \bar{\rho}.
\end{equation}
Then
\begin{detailedproof}
  \bar{\rho}
  &= (1 - \eps) \iprodt{\ones}{x} + \tfrac{\eps}{2} \norm[1]{w}
  &&\text{by \cref{eq:10}}\\
  &\le (1 - \eps)\paren[\big]{
    \iprod{\tfrac{1}{4}\Laplacian_G(w)}{\tilde{Y}}
    + \sigma \norm[\infty]{w}
  }
  + \tfrac{\eps}{2} \norm[1]{w}
  &&\text{since \(
    \iprodt{\ones}{x}
    \le \tilde{\rho}
    \le \iprod{\tfrac{1}{4}\Laplacian_G(w)}{\tilde{Y}}
    + \sigma \norm[\infty]{w}
  \)}\\
  &= \iprod{\tfrac{1}{4}\Laplacian_G(w)}{(1 - \eps) \tilde{Y} + \eps I}
  + (1 - \eps) \sigma \norm[\infty]{w}
  &&\text{since \(\tfrac{\eps}{2}\norm[1]{w} =
    \iprod{\tfrac{1}{4}\Laplacian_G(w)}{\eps I}\)}\\
  &= \iprodt{w}{z}
  + (1 - \eps) \sigma \norm[\infty]{w}
  &&\text{by the definitions of \(Y\) and \(z\)}\\
  &\le \iprodt{w}{z} + \sigma \bar{\rho}
  &&\text{by~\cref{eq:certificates-from-w-approx}}.
\end{detailedproof}
This completes the proof of~\cref{eq:8}.
Moreover, as \(w \neq 0\), we get
from~\cref{eq:certificates-from-w-approx} that \(\bar{\rho} > 0\).
Hence \(\iprodt{w}{z} \ge (1 - \sigma)\bar{\rho} > 0\), as
\(\sigma < 1\).
Thus \(z \neq 0\).
\Cref{prop:sampling-GWOpt} then finishes the proof.
\end{proof}

\begin{theorem}
\label{thm:certificates-from-z-approx}
Let \(\beta \in (0, \GWalpha)\).
Set \(\tau \coloneqq 1- \beta/\GWalpha\) and \(C \coloneqq
2187\pi/(2\GWalpha^2\tau^3)\).
There exists a randomized polynomial-time algorithm that, given a
graph \(G = (V, E)\) on \(n\) vertices and a nonzero \(z \in \Lp{E}\),
computes with probability at least \(1-2/n\) a nonzero vector \(w \in
\Lp{E}\) and a \(\beta\)-certificate \((\rho, \mu, S, y, x)\) for
\((w,z)\) such that the support of the fractional cut cover \(y\) has
size at most \(\ceil[\big]{C\ln(n)}\).
\end{theorem}
\begin{proof}
  Let \(z \in \Lp{E}\) be nonzero.
  Define \(\sigma \coloneqq \tfrac{2}{3}\tau\) and \(\eps \coloneqq \tau/(9 -
  6\tau)\), as in \Cref{prop:sampling-GWOpt}.
  We follow \cref{alg:fcc_mc}.
  By (nearly) solving the primal-dual SDPs
  in~\cref{eq:GW-eps-polar-def}, one can compute in polynomial time a
  feasible solution \((w,x)\) for~\cref{eq:GW-eps-polar-supf} and a
  feasible solution \((\bar{\mu}, Y)\) for~\cref{eq:GW-eps-polar-gaugef}
  such that
  \begin{equation*}
    \bar{\mu} \le \iprodt{z}{w} + \sigma \norm[\infty]{z}.
  \end{equation*}
  By combining~\cref{eq:GW-polar-norm-infty-lowerbound},
  \cref{eq:GW-eps-polar-bound}, and~\cref{eq:GW-eps-polar-gaugef}, we
  see that
  \begin{equation}
    \label{eq:norm-infty-le-GW-eps-polar}
    \norm[\infty]{z}
    \le
    \GW^{\polar}(G,z)
    \le
    \GW_{\eps}^{\polar}(G,z)
    \le
    \bar{\mu}.
  \end{equation}
  We claim that
  \begin{equation}
    \label{eq:9}
    \text{%
      \((1, x)\) and \((\bar{\mu}, Y)\)
      witness the membership \((w, z) \in \GWOpt_{\eps, \sigma}(G)\).
    }
  \end{equation}
  Feasibility of \((1,x)\) and \((\bar{\mu},Y)\) in the appropriate SDPs are
  easily verified.
  Moreover, by~\cref{eq:norm-infty-le-GW-eps-polar}, one has
  \begin{inlinemath}
    \bar{\mu}
    \le \iprodt{w}{z}
    + \sigma \norm[\infty]{z}
    \le \iprodt{w}{z}
    + \sigma \bar{\mu}.
  \end{inlinemath}
  This proves~\cref{eq:9}.
  Moreover, as \(z \neq 0\), we get
  from~\cref{eq:norm-infty-le-GW-eps-polar} that \(\bar{\mu} > 0\).
  Hence \(\iprodt{w}{z} \ge \bar{\mu}(1 - \sigma) > 0\) as \(\sigma < 1\).
  Thus \(w \neq 0\).
  \cref{prop:sampling-GWOpt} finishes the proof.
\end{proof}

\begin{remark}

\Cref{thm:certificates-from-w-approx} uses that the
SDP~\cref{eq:GW-intro-pd} is nearly solvable in polynomial time.
However, \Cref{thm:certificates-from-z-approx} relies on nearly
solving the SDP~\cref{eq:GW-eps-polar-def}, which is introduced in
this work.
\Cref{sec:solvers} proves this can be done in polynomial time.
\end{remark}

\section{Geometric Representation of Graphs}
\label{sec:geometric-graphs}

\begin{definition}
A \emph{hypersphere representation of} a graph \(G = (V, E)\) is a
map \(u \colon V \to \Reals^d\) for some \(d \in \Naturals\) such that
the map \(i \in V \mapsto \norm{u_i} \in \Reals\) is constant.
Such constant is the \emph{radius} of~\(u\), denoted by~\(\radius(u)\).
We~denote by~\(\Hcal(G)\) the set of all hypersphere representations
of~\(G\).
\end{definition}

Let \(G = (V, E)\) be a graph, and let \(\mu \in \Reals_+\).
Hypersphere representations of~\(G\) with squared radius~\(\mu\) are
directly related to \(\mu\elliptope{V}\) via their Gram matrices: if
\(u \in \Hcal(G)\) has squared radius~\(\mu\) and one defines a
matrix~\(U\) with columns \(\set{u_i}_{i \in V}\), then
\(U^{\transp} U \in \mu \elliptope{V}\); conversely, if
\(X \in \mu\elliptope{V}\), then the columns of~\(X^{1/2}\) form a
hypersphere representation of~\(G\) with squared radius~\(\mu\).

Let \(X \in \mu\elliptope{V}\) be the Gram matrix corresponding to a
hypersphere representation \(u \in \Hcal(G)\).
Then for every \(z \in \Lp{E}\),
\begin{equation}
  \label{eq:4}
  \Laplacian_G^*(X) \ge z
  \text{ if and only if }
  \norm{u_i - u_j}^2 \ge z_{ij}
  \text{ for every }
  ij \in E.
\end{equation}
In this way, the constraint defined by the adjoint of the Laplacian
has a natural geometric interpretation.

The study of geometric representations of graphs has been very
fruitful~\cite{Lovasz19a}.
To~our knowledge, geometric representations provided the first proof
that specializing the SDP in~\cref{eq:GW-polar-def-intro} to the case
\(z = \ones\) recovers the relationship with the vector chromatic
number.
A \emph{unit-distance representation of} a graph \(G = (V,E)\) is a hypersphere
representation \(u \colon V \to \Reals^d\) of~\(G\) such that
\(\norm{u_i - u_j}^2 = 1\) for every \(ij \in E\).
The~\emph{hypersphere number} of \(G\), denoted by \(t(G)\), is the
smallest squared radius of a unit-distance representation of \(G\).
Using semidefinite programming, one can write
\begin{align*}
  t(G)
  &= \inf\setst{
    \radius(u)^2
  }{
    u \in \Hcal(G)
    \text{ such that }
    \norm{u_i - u_j}^2 = 1
    \text{ for every }
    ij \in E
  }\\
  &= \inf\setst{
    \mu
  }{
    \mu \in \Reals_+,\,
    X \in \Psd{V},\,
    \diag(X) = \mu\ones,\,
    \Laplacian_G^*(X) = \ones
  }.
\end{align*}
The only differences between this optimization problem and the
problem~\cref{eq:GW-polar-def} specialized to \(z = \ones\) appear in
the constraint featuring the adjoint of the Laplacian: the \(\tfrac{1}{4}\)
factor is gone, and `\(\ge\)' was changed to `\(=\)'.
Lovász \cite[p.~23]{Lovasz2003} proved that
\begin{equation}
  \label{eq:Lovasz-theta-unit-distance-representations}
  2t(G) + \frac{1}{\theta(\overline{G})} = 1,
\end{equation}
where \(\theta\) denotes the \emph{Lovász theta function}
\cite{Lovasz1979}.
The similarities between the Lovász theta function \(\theta\) and the
vector chromatic number are already discussed in the work introducing
\(\chivec\) \cite{KargerMotwaniEtAl1998}.
In fact, \(\chivec(G) = \theta'(\overline{G})\), where
\(\overline{G}\) denotes the complement of a graph and \(\theta'\),
commonly referred to as Schrijver's \(\theta'\) function, denotes a
variant of \(\theta\) introduced independently
in~\cite{McElieceRodemichEtAl1978} and \cite{Schrijver1979}.
A natural variation
of~\cref{eq:Lovasz-theta-unit-distance-representations} that
involves~\(\theta'\) (see, e.g.,
\cite[Sec.~4]{deCarliSilvaTuncel2013}) is
\[
  2t'(G) + \frac{1}{\chivec(G)} = 1,
\]
where
\[
  t'(G)
  \coloneqq \inf\setst{
    \mu
  }{
    X \in \Psd{V},\,
    \diag(X) = \mu\ones,\,
    \Laplacian_G(X) \ge \ones
  }
  = \tfrac{1}{4}\GW^{\polar}(G).
\]
It is then immediate that
\begin{equation}
  \label{eq:GW-polar-theta-prime}
  \GW^{\polar}(G, \ones)
  = 2\paren[\bigg]{1 - \frac{1}{\chivec(G)}}.
\end{equation}
Our introduction presented~\cref{eq:vector-coloring-approximates-fcc}
as a motivating fact of our work.
Note that~\cref{eq:GW-polar-theta-prime} and \Cref{cor:1} provide an
alternative proof of~\cref{eq:vector-coloring-approximates-fcc}.
\Cref{eq:GW-polar-theta-prime} is another manifestation of the
well-known connection between the Lovász theta function (and its
variants) with the maximum cut problem --- see, e.g.,
\cite{LaurentPoljakEtAl1997,deCarliSilvaTuncel2015}.

The optimization problems we have considered optimize different
objective functions over certain hypersphere representations of a
graph \(G = (V, E)\).
For each \(z \in \Lp{E}\), an \emph{optimal fcc representation of \((G,
z)\)} is a hypersphere representation \(u \colon V \to \Reals^d\)
such that \(\tfrac{1}{4}\norm{u_i - u_j}^2 \ge z_{ij}\) for every edge
\(ij \in E\), and with minimal radius among such representations.
For each \(w \in \Lp{E}\), an \emph{optimal mc representation of \((G,
  w)\)} is a hypershere representation \(u \colon V \to \Reals^d\)
with radius~1 which maximizes \(\sum_{ij \in E} w_{ij} \norm{u_i -
  u_j}^2\).
Define, for every \(w, z \in \Lp{E}\),
\begin{align*}
  \FCC(G, z)
  &\coloneqq \argmin \setst[\big]{
    \radius(u)^2
  }{
    u \in \Hcal(G)
    \text{ such that }
    \tfrac{1}{4}\norm{u_i - u_j}^2 \ge z_{ij}
    \text{ for every }
    ij \in E
  }, \text{ and }\\
  \MC(G, w)
  &\coloneqq \argmax \setst[\bigg]{
    \frac{1}{4}\sum_{ij \in E} w_{ij}\norm{u_i - u_j}^2
  }{
    u \in \Hcal(G),\,
    \radius(u) = 1
  }.
\end{align*}
The connection between these two sets of optimal geometric
representations, which we illustrate with
\Cref{fig:hypersphere-representation}, is captured by the following
theorem.

\begin{theorem}
\label{thm:geometric-representation-equivalence}
Let \(G = (V, E)\) be a graph.
If \(w \in \Lp{E}\) is nonzero, then for every \(u \in \MC(G, w)\)
there exists \(z \in \optz_G(w)\) such that
\begin{equation}
  \label{eq:12}
  (\sqrt{\mu}u_i)_{i \in V} \in \FCC(G, z),
  \text{ where }
  \mu
  \coloneqq \min\setst{
    \rho \in \Lp{}
  }{
    \rho\tfrac{1}{4}\norm{u_i - u_j}^2 \ge z_{ij}
    \text{ for every } ij \in E
  }.
\end{equation}
Conversely, if \(z \in \Lp{E}\) is nonzero, then for every \(v \in
\FCC(G, z)\) there exists \(w \in \optw_G(z)\) such that
\begin{equation}
  \label{eq:28}
  \paren[\big]{\radius(v)^{-1} v_i}_{i \in V} \in \MC(G, w).
\end{equation}
\end{theorem}
\begin{remark}
Let \(G = (V, E)\) be a graph.
Note that~\cref{eq:12} describes an algorithm producing an element in
\(\FCC(G, z)\) from inputs \(G\), \(w \in \Lp{E}\), \(u \in \MC(G,
w)\), and \(z \in \optz_G(w)\).
Similarly,~\cref{eq:28} describes an algorithm producing an element in
\(\MC(G, w)\) from inputs \(G\), \(z \in \Lp{E}\), \(v \in \FCC(G,
z)\), and \(w \in \optw_G(z)\).
If one is given \(w \in \Lp{E}\) and an optimal mc representation \(u
\in \MC(G,w)\) and simply wants to obtain an associated optimal fcc
representation, there is an easily computable choice of \(z \in
\optz_G(w)\) --- namely, \(z_{ij} \coloneqq \tfrac{1}{4}\norm{u_i -
u_j}^2\) for every \(ij \in E\).
\end{remark}
\begin{proof}[Proof of \cref{thm:geometric-representation-equivalence}]

Let \(w \in \Lp{E}\) be nonzero, and let \(u \in \MC(G, w)\).
Let \(Y \in \elliptope{V}\) be the Gram matrix corresponding to~\(u\),
i.e., \(Y_{ij} = \iprodt{u_i}{u_j}\) for every \(i,j \in V\).
Since \(u \in \MC(G, w)\), there exists \(z \in \Lp{E}\) and \((\rho,
x) \in \Lp{} \times \Reals^V\) such that  \((z, Y)\) and \((\rho, x)\)
are optimal solutions
for the SDPs~\cref{eq:GW-supf-explicit-z,eq:GW-gaugef-explicit-z}, respectively.
We will prove that \(v \coloneqq (\sqrt{\mu}u_i)_{i \in V} \in \FCC(G, z)\).
From optimality of \((z,Y)\),
\begin{equation}
  \label{eq:27}
  \radius(v)^2
  = \min\setst{
      \rho \in \Lp{}
      }{
      \rho\tfrac{1}{4}\norm{u_i - u_j}^2 \ge z_{ij}
      \text{ for every } ij \in E
      }
  = \min \setst{\rho \in \Lp{}}{\tfrac{1}{4}\Laplacian_G^*(\rho Y) \ge z}
  = 1.
\end{equation}
Now let \(\tilde{v} \in \Hcal(G)\) be such that
\(\tfrac{1}{4}\norm{\tilde{v}_i - \tilde{v}_j}^2 \ge z_{ij}\) for
every \(ij \in E\), and let \(\tilde{Y}\) be the Gram matrix of
\(\tilde{v}\).
Since \((\rho, x)\) is feasible in~\cref{eq:GW-gaugef-explicit-z},
\begin{detailedproof}
  \rho \radius(\tilde{v})^2
  \ge \iprodt{\ones}{x} \radius(\tilde{v})^2
  = \iprod{\Diag(x)}{\tilde{Y}}
  \ge \iprod{\tfrac{1}{4}\Laplacian_G(w)}{\tilde{Y}}
  = \sum_{ij \in E} w_{ij} \tfrac{1}{4}\norm{\tilde{v}_i - \tilde{v}_j}^2
  \ge \iprodt{w}{z}
  = \rho.
\end{detailedproof}
As \(w \neq 0\), we have that \(\rho > 0\), so
\(\radius(\tilde{v})^2 \ge 1\).
As the latter inequality holds with equality at~\(v\) by~\cref{eq:27},
we~conclude that \(v \in \FCC(G, z)\).

Let \(z \in \Lp{E}\) be nonzero, let \(v \in \FCC(G, z)\), and let
\(Y\) be the Gram matrix of \(v\).
Since \(v \in \FCC(G, z)\), there exists \((w, x) \in \Lp{E} \times
\Reals^V\) such that \((\radius(v)^2, Y)\) and \((w, x)\) are optimal
solutions to~\cref{eq:GW-polar-gaugef,eq:GW-polar-supf}, respectively.
We will prove that
\(u \coloneqq \paren[\big]{\radius(v)^{-1} v_i}_{i \in V} \in \MC(G,w)\).
We claim that
\begin{equation}
  \label{eq:29}
  1
  = \tfrac{1}{4}\sum_{ij \in E} w_{ij} \norm{u_i - u_j}^2.
\end{equation}
Indeed, by optimality of \((\radius(v)^2, Y)\) and \((w, x)\), we have
that
\(
  \radius(v)^2
  \geq \radius(v)^2\iprodt{\ones}{x}
  = \iprod{Y}{\Diag(x)}
  \ge \iprod{Y}{\tfrac{1}{4}\Laplacian_G(w)}
  \ge \iprodt{z}{w}
  = \radius(v)^2
\).
Thus
\begin{detailedproof}
  \radius(v)^2
  = \iprod{Y}{\tfrac{1}{4}\Laplacian_G(w)}
  = \tfrac{1}{4} \sum_{ij \in E} w_{ij}\norm{v_i - v_j}^2
  = \radius(v)^2\tfrac{1}{4} \sum_{ij \in E} w_{ij}\norm{u_i - u_j}^2.
\end{detailedproof}
As \(z \neq 0\), we have that \(\radius(v)^2 > 0\), so~\cref{eq:29}
holds.
Now let \(\tilde{u} \in \Hcal(G)\) be such that \(\radius(\tilde{u}) =
1\).
Let \(\tilde{Y}\) be the Gram matrix corresponding to \(\tilde{u}\).
Then \(u \in \MC(G, w)\) follows from~\cref{eq:29}, as
\begin{equation*}
  \tfrac{1}{4} \sum_{ij \in E} w_{ij} \norm{\tilde{u}_i - \tilde{u}_j}^2
  = \iprod{\tilde{Y}}{\tfrac{1}{4} \Laplacian_G(w)}
  \le \iprod{\tilde{Y}}{\Diag(x)}
  = \radius(\tilde{u})^2\iprodt{\ones}{x}
  \le 1
  = \tfrac{1}{4}\sum_{ij \in E} w_{ij} \norm{u_i - u_j}^2.
  \qedhere
\end{equation*}
\end{proof}

\begin{figure}
  \centering
  \begin{subfigure}{0.4\textwidth}
    \centering
    \vspace{-1.1em}
    \begin{tikzpicture}[tdplot_main_coords]
  \pgfmathsetmacro{\radius}{1.3}
  
  \newcommand{\getCoordinate}[2]{
    ({\radius*sin(#1)*cos(#2)},{\radius*sin(#1)*sin(#2)},{\radius*cos(#1)})
  }
  
  \shade[ball color=gray!20, opacity=0.4] (0, 0, 0) circle [radius=\radius cm];
  \draw[gray,dashed] plot[variable=\t,domain=-180:180,smooth] (\getCoordinate{90}{\t});
  \draw[gray,dashed] plot[variable=\t,domain=-180:180,smooth] (\getCoordinate{\t}{0});
  \draw[-latex, gray, opacity=.5] (0,0,0) -- \getCoordinate{0}{0};
  \draw[-latex, gray, opacity=.5] (0,0,0) -- \getCoordinate{90}{0};
  \draw[-latex, gray, opacity=.5] (0,0,0) -- \getCoordinate{90}{90};
  \draw[gray,dashed,opacity=.5] plot[variable=\t,domain=-180:180,smooth] (\getCoordinate{90}{\t});
  \draw[gray,dashed,opacity=.5] plot[variable=\t,domain=-180:180,smooth] (\getCoordinate{\t}{0});
  
  \pgfmathsetmacro{\phiOne}{90} 
  \pgfmathsetmacro{\thetaOne}{120} 
  \pgfmathsetmacro{\phiTwo}{270} 
  \pgfmathsetmacro{\thetaTwo}{0} 
  \pgfmathsetmacro{\phiThree}{25} 
  \pgfmathsetmacro{\thetaThree}{0} 
  \pgfmathsetmacro{\phiFour}{90} 
  \pgfmathsetmacro{\thetaFour}{35} 

  \draw[-latex, thick] (0,0,0) -- \getCoordinate{\phiOne}{\thetaOne};
  \draw[-latex, thick] (0,0,0) -- \getCoordinate{\phiTwo}{\thetaTwo};
  \draw[-latex, thick] (0,0,0) -- \getCoordinate{\phiThree}{\thetaThree};
  \draw[-latex, thick] (0,0,0) -- \getCoordinate{\phiFour}{\thetaFour};
  
  \draw[blue, thick] \getCoordinate{\phiOne}{\thetaOne} -- \getCoordinate{\phiTwo}{\thetaTwo};
  \draw[blue, thick] \getCoordinate{\phiTwo}{\thetaTwo} -- \getCoordinate{\phiFour}{\thetaFour};
  \draw[blue, thick] \getCoordinate{\phiFour}{\thetaFour} --\getCoordinate{\phiThree}{\thetaThree};
  \draw[blue, thick] \getCoordinate{\phiThree}{\thetaThree} -- \getCoordinate{\phiOne}{\thetaOne};

  \node[right] at \getCoordinate{\phiOne}{\thetaOne} {$u_1$};
  \node[right] at \getCoordinate{\phiTwo}{\thetaTwo} {$u_2$};
  \node[above left] at \getCoordinate{\phiThree}{\thetaThree} {$u_3$};
  \node[below] at \getCoordinate{\phiFour}{\thetaFour} {$u_4$};
  
\end{tikzpicture}
    \caption{mc representation \(u \colon V \to \Reals^3\)}
  \end{subfigure}
  \hspace{.05\textwidth}
  \begin{subfigure}{0.4\textwidth}
    \centering
    \begin{tikzpicture}[tdplot_main_coords]
  \pgfmathsetmacro{\radius}{3}

  \newcommand{\getCoordinate}[2]{
    ({\radius*sin(#1)*cos(#2)},{\radius*sin(#1)*sin(#2)},{\radius*cos(#1)})
  }

  \shade[ball color=gray!20, opacity=0.4] (0, 0, 0) circle [radius=\radius cm];
  \draw[gray,dashed] plot[variable=\t,domain=-180:180,smooth] (\getCoordinate{90}{\t});
  \draw[gray,dashed] plot[variable=\t,domain=-180:180,smooth] (\getCoordinate{\t}{0});
  \draw[gray,dashed,opacity=.5] plot[variable=\t,domain=-180:180,smooth] (\getCoordinate{90}{\t});
  \draw[gray,dashed,opacity=.5] plot[variable=\t,domain=-180:180,smooth] (\getCoordinate{\t}{0});
  
  \pgfmathsetmacro{\phiOne}{90} 
  \pgfmathsetmacro{\thetaOne}{120} 
  \pgfmathsetmacro{\phiTwo}{270} 
  \pgfmathsetmacro{\thetaTwo}{0} 
  \pgfmathsetmacro{\phiThree}{25} 
  \pgfmathsetmacro{\thetaThree}{0} 
  \pgfmathsetmacro{\phiFour}{90} 
  \pgfmathsetmacro{\thetaFour}{35} 
  
  \draw[-latex, thick] (0,0,0) -- \getCoordinate{\phiOne}{\thetaOne};
  \draw[-latex, thick] (0,0,0) -- \getCoordinate{\phiTwo}{\thetaTwo};
  \draw[-latex, thick] (0,0,0) -- \getCoordinate{\phiThree}{\thetaThree};
  \draw[-latex, thick] (0,0,0) -- \getCoordinate{\phiFour}{\thetaFour};
  
  \draw[blue, thick] \getCoordinate{\phiOne}{\thetaOne} -- \getCoordinate{\phiTwo}{\thetaTwo};
  \draw[blue, thick] \getCoordinate{\phiTwo}{\thetaTwo} -- \getCoordinate{\phiFour}{\thetaFour};
  \draw[blue, thick] \getCoordinate{\phiFour}{\thetaFour} --\getCoordinate{\phiThree}{\thetaThree};
  \draw[blue, thick] \getCoordinate{\phiThree}{\thetaThree} -- \getCoordinate{\phiOne}{\thetaOne};

  \node[right] at \getCoordinate{\phiOne}{\thetaOne} {\(v_1\)};
  \node[right] at \getCoordinate{\phiTwo}{\thetaTwo} {\(v_2\)};
  \node[above left] at \getCoordinate{\phiThree}{\thetaThree} {\(v_3\)};
  \node[below] at \getCoordinate{\phiFour}{\thetaFour} {\(v_4\)};

  \pgfmathsetmacro{\radius}{1.3}
  \shade[ball color=gray!20, opacity=0.3] (0, 0, 0) circle [radius=\radius cm];
  \draw[gray,dashed] plot[variable=\t,domain=-180:180,smooth] (\getCoordinate{90}{\t});
  \draw[gray,dashed] plot[variable=\t,domain=-180:180,smooth] (\getCoordinate{\t}{0});
  
  \draw[-latex, thick] (0,0,0) -- \getCoordinate{\phiOne}{\thetaOne};
  \draw[-latex, thick] (0,0,0) -- \getCoordinate{\phiTwo}{\thetaTwo};
  \draw[-latex, thick] (0,0,0) -- \getCoordinate{\phiThree}{\thetaThree};
  \draw[-latex, thick] (0,0,0) -- \getCoordinate{\phiFour}{\thetaFour};
\end{tikzpicture}%
    \caption{and its corresponding fcc representation \(v \colon V \to
      \Reals^3\)}
  \end{subfigure}
  \caption{Geometric equivalence between hypersphere representations
    illustrating \cref{thm:geometric-representation-equivalence}.}
  \label{fig:hypersphere-representation}
\end{figure}

\section{Tightness of our Results}
\label{sec:possible-improvements}

This section discusses several aspects in which our algorithms and
analyses are best possible.
We collect instances of the maximum cut and fractional cut-covering
problems that justify several aspects of our algorithms, including
\begin{enumerate}[(i)]
\item\label{it:obstruction-1} the need to sparsify the cut cover defined in
  \Cref{prop:sdp-rounding},
\item\label{it:obstruction-2} the need to either ``thicken'' edges, as
  in~\cref{sec:rounding}, or to perturb the SDP, as
  in~\cref{subsec:algorithimic_certicates},
\item\label{it:obstruction-3} the asymptotic support size of the
  fractional cut cover we have obtained,
\item\label{it:obstruction-4} the approximation factor.
\end{enumerate}
These aspects are intertwined.
\Cref{it:obstruction-1} motivates the use of repeated sampling.
\Cref{it:obstruction-2} shows that naively sampling cuts until a
fractional cut cover is obtained takes too long for some choice of
weights, and
\cref{it:obstruction-3} discusses the amount of cuts needed to be
sampled.
The whole algorithmic set up grounds our analysis and guides our
discussion of~\cref{it:obstruction-4}.

\subsection{Sparsification of Rounded Solution}
Let \(G = (V, E)\) be a graph, and let \(z \in \Lp{E}\).
Let \((\mu, Y)\) be feasible in~\cref{eq:GW-polar-def-intro}, and set
\begin{equation}
  \label{eq:13}
  y
  \coloneqq \bar{\nu}p,
  \text{ where }
  p_S \coloneqq \prob\paren{\GWrv(Y) = S}
  \text{ for every }
  S \subseteq V,
  \text{ and }
  \bar{\nu} \coloneqq
  \min\setst[\Big]{\nu \in \Lp{}}{
    \nu \sum_{S \subseteq V} p_S \incidvector{\delta(S)} \ge z
  }.
\end{equation}
\Cref{prop:sdp-rounding} implies that \(\bar{\nu} \le \mu/\GWalpha\).
Via~\cref{eq:13} we have a deterministic fractional
cut cover \(y \in \Lp{\Powerset{V}}\) for every feasible solution
\((\mu, Y)\) of~\cref{eq:GW-polar-def-intro}.
It~is then natural to question the necessity of the repeated sampling
approach.
\Cref{re:need-to-randomize} mentions the difficulty of computing the
vector \(p \in \Lp{\Powerset{V}}\), and we now exhibit an instance
where \(y\) has exponential support size.
One may check that
\begin{inlinemath}
  (\bar{\mu},\bar{Y})
  \coloneqq
  \paren[\big]{
    2-\tfrac{2}{n}
    ,
    2I-\tfrac{2}{n}\oprodsym{\ones}
  }
\end{inlinemath}
is an optimal solution to \(\GW^{\polar}(K_n)\) for every nonzero
\(n \in \Naturals\), certified by the dual optimal solution
\begin{inlinemath}
  \paren{\bar{w},\bar{x}}
  \coloneqq
  \paren[\big]{
    \tfrac{4}{n^2}\ones
    ,
    \tfrac{1}{n}\ones
  }
\end{inlinemath}
for \cref{eq:GW-polar-supf}.
For every \(i \in [n]\), let \(g_i\) be independently sampled from the
standard normal distribution, and set \(h \coloneqq \norm{g}^{-1}g\).
One can prove --- see \cref{sec:exp-supp} --- that
\begin{equation}
  \label{eq:exp-supp-pledge}
  \supp(y)
  = \supp(p)
  = \setst{S \subseteq V}{\prob(\GWrv(\bar{Y},h) = S) > 0}
  = \setst{S \subseteq V}{S \neq \emptyset,\, S \neq V}.
\end{equation}
Hence, the vector \(y \in \Lp{\Powerset{V}}\) defined in~\cref{eq:13}
may have exponential support size.

\subsection{Edge Thickening or SDP Perturbation}
The repeated sampling approach naturally arises as a sparsification of
the probability distribution in~\cref{eq:13}.
Let \(G = (V, E)\) be a graph, let \(z \in \Lp{E}\), and let \((\mu,
Y)\) be feasible in~\cref{eq:GW-polar-def-intro}.
Set \(\hat{y}_0 \coloneqq 0\).
For every nonzero \(t \in \Naturals\), set
\begin{equation}
  \label{eq:hat-y-def}
  \hat{y}_t \coloneqq \hat{y}_{t - 1} + e_{S(t)}
  \text{ and }
  \mu_t \coloneqq
  \inf\setst[\Big]{\nu \in \Lp{}}{
    \nu \sum_{S \subseteq V} (\hat{y}_t)_S \incidvector{\delta(S)} \ge z
  },
\end{equation}
where each \(S(t)\) is independently sampled from \(\GWrv(Y)\).
These objects capture what we mean by ``repeated sampling''.
We claim that
\begin{equation}
  \label{eq:y-as-asymptotic-behavior}
  \prob\paren[\Big]{\,
    \lim_{t \to \infty} \mu_t \hat{y}_t
     = y
   }
   = 1,
\end{equation}
where \(y \in \Lp{\Powerset{V}}\) is defined as in~\cref{eq:13}.
In words, \cref{eq:y-as-asymptotic-behavior} states that almost
surely, the vector \(y\) describes the behavior of the
repeated sampling approach as more samples are taken.
We now prove~\cref{eq:y-as-asymptotic-behavior}.
Let \(B \in \Reals^{E \times \Powerset{V}}\) be the incidence matrix
of the cuts of \(G\), so \(Bx = \sum_{S \subseteq V} x_S
\incidvector{\delta(S)}\) for every \(x \in \Reals^{\Powerset{V}}\).
Set \(\Dcal \coloneqq \setst[\big]{x \in \Lp{\Powerset{V}}}{\supp(Bx)
\supseteq \supp(z)}\), and set \(f(x) \coloneqq \max
\setst{z_{ij}/(Bx)_{ij}}{ij \in \supp(z)}\) for every \(x \in \Dcal\).
Observe that \(f\) is continuous on \(\Dcal\),
that \(p \in \Dcal\) by \cref{prop:sdp-rounding}
and \(f(p) = \bar{\nu}\), and that \(f(\alpha x) = \tfrac{1}{\alpha}
x\) for every \(\alpha \in \Reals_{++}\) and \(x \in \Dcal\).
The Strong Law of Large Numbers implies that \(\lim_{t \to \infty}
\tfrac{1}{t}\hat{y}_t = p\) almost surely.
Assume that this event happens.
For every sufficiently large \(t \in \Naturals\), one has
\(\hat{y}_t \in \Dcal\),
as~\(\supp(B\hat{y}_t) = \supp(Bp) \supseteq \supp(z)\).
Thus~\cref{eq:y-as-asymptotic-behavior} holds, as
\[
  \lim_{t \to \infty} \mu_t \hat{y}_t
  = \lim_{t \to \infty} f(\hat{y}_t) \hat{y}_t
  = \lim_{t \to \infty} f(\hat{y}_t)t \tfrac{1}{t}\hat{y}_t
  = \lim_{t \to \infty} f\paren[\big]{\tfrac{1}{t} \hat{y}_t} \tfrac{1}{t}\hat{y}_t
  = f(p) p
  = \bar{\nu} p
  = y.
\]

\begin{proposition}
\label{prop:small-edges}

Let \(\eps \in (0, 2)\).
There exist a graph \(G = (V, E)\), vectors \(w \in \Lp{E}\) and \(z
\in \Lp{E}\), as well as \((1, x)\) and \((\mu, Y)\) witnesssing \((w,
z) \in \GWOpt(G)\) such that, for \(\mu_t\) and \(\hat{y}_t\) defined
as in~\cref{eq:hat-y-def},
\[
  \Ebb\paren{
    \min\setst{t \in \Naturals}{\mu_t < +\infty}
  }
  \ge
  \paren[\bigg]{
    \frac{2\sqrt{\eps}}{\pi} + O(\eps^{3/2})
  }^{-1}.
\]
\end{proposition}
\begin{proof}

Let \(G \coloneqq K_3\) be the complete graph on three vertices, with
\(V(K_3) = \set{1, 2, 3}\).
Set \((z_{12}, z_{13}, z_{23}) \coloneqq (1, 1, \eps) \in \Lp{E}\).
\Cref{sec:small-edges} shows there exist \(w \in \Lp{E}\) and \(x \in
\Reals^V\) such that \((1, x)\) and \((\mu, Y)\) witness the
membership \((w, z) \in \GWOpt(G)\), where \(\mu \coloneqq
4/(4-\eps)\) and
\[
  Y \coloneqq
  \frac{4}{4 - \eps}
  \begin{bmatrix}
    1 & \eps/2 -1 & \eps/2 - 1\\
    \eps/2 - 1 & 1 & 1 - 2\eps + \eps^2/2\\
    \eps/2 -1 & 1 - 2\eps + \eps^2/2 & 1
  \end{bmatrix}.
\]
If \(t \in \Naturals\) is such that \(\mu_t < +\infty\), then the edge
\(23 \in E\) was covered, and hence
\[
  \Ebb\paren{
    \min\setst{t \in \Naturals}{\mu_t < +\infty}
  }
  \ge
  \frac{1}{\prob\paren{
      \set{2,3} \in \delta(\GWrv(Y))
    }}
  = \frac{\pi}{\arccos(1 - 2\eps + \eps^2/2)}
  = \paren[\bigg]{
    \frac{2\sqrt{\eps}}{\pi} + O(\eps^{3/2})
  }^{-1}.
  \qedhere
\]
\end{proof}

Although~\cref{eq:y-as-asymptotic-behavior} ensures repeated sampling
converges to the solution in~\cref{eq:13} almost surely,
algorithmically it~is necessary to bound the number of samples one has
to take.
\Cref{re:small-edges} mentions that edges with relatively small
weights can force exponentially many samples to be taken just to
enable feasibility.
\Cref{prop:small-edges} formalizes that remark: it defines a family of
instances where naively sampling from an optimal solution to the
SDP~\cref{eq:GW-polar-def-intro} may require, in expectation,
exponentially (on the size of the input \((G, z)\)) many cuts just to
output a feasible fractional cut cover.
\Cref{prop:small-edges} also motivates the perturbed SDPs introduced
in~\cref{eq:GW-eps-def-outer} and~\cref{eq:GW-eps-polar-def}, as it
shows, in~particular, that one cannot take \(\eps = 0\) in
\cref{prop:sampling-GWOpt}.

\subsection{Asymptotic Support Size}
Let \(G = (V, E)\) be a graph, let \(z \in \Lp{E}\), and set \(n
\coloneqq \card{V}\).
\Cref{prop:probabilistic-rounding-cover-pre} shows that, by assuming
the ratio between the largest and smallest entry of \(z\) is bounded,
we can produce in polynomial time a fractional cut cover with support
size \(O(\ln(n))\).
Our algorithms then perturb the input so this hypothesis is met.
The~logarithmic bound may not be asymptotically improved without
further assumptions on the input.
Assume that \(\supp(z) = E\).
We~claim that
\begin{equation}
  \label{eq:chi-G-le-supp-y}
  \ceil{\lg(\chi(G))}
  \le \card{\supp(y)}
  \text{ for every fractional cut cover~\(y\) of~\((G,z)\)},
\end{equation}
where \(\chi(G)\) denotes the chromatic number of~\(G\).
Since \(\supp(z) = E\), every edge must be in some cut defined by a
shore in \(\supp(y)\).
The minimum number of cuts necessary to cover the edges of a graph
is \(\ceil{\lg(\chi(G))}\) --- see, for example,
\cite[Section~6]{DeVosNesetrilEtAl2007}.
Thus~\cref{eq:chi-G-le-supp-y} holds.
In particular, the bound on \(\card{\supp(y)}\) given by
\cref{thm:rounding-algorithm} or \cref{prop:sampling-GWOpt} is
asymptotically best possible for graphs such that
\(\chi(G) = \Theta(\card{V(G)})\) --- in particular, for complete
graphs.

\subsection{Computational Complexity of Fractional Cut Covering}

\Cref{ssec:approximation-factor} below addresses how the
approximation factor from our algorithms is tight.
Prior to this discussion is the computational complexity status of the
problem we are attempting to solve.

\begin{proposition}

Let \(G = (V, E)\) be a graph, let \(z \in \Rationals_+^E\), and let
\(\mu \in \Rationals\).
Consider the problem
\begin{equation}
  \label{eq:fcc-decision-version}
  \text{%
    given \((G, z, \mu)\) as input,
    decide if \(\fcc(G,z) \leq \mu\).
  }
\end{equation}
This problem is in NP.
\end{proposition}
\begin{proof}

Since the set of optimal solutions to~\cref{eq:cut-covering-problem}
is non-empty, Carathéodory's Theorem, implies that there exists
\(\Fcal \subseteq \Powerset{V}\) with \(\card{\Fcal} \le \card{E} +
1\) and \(y \in \Lp{\Fcal}\) such that
\begin{inlinemath}
  \sum_{S \in \Fcal} y_S \incidvector{\delta(S)} \ge z
\end{inlinemath}
and
\begin{inlinemath}
  \iprodt{\ones}{y} = \fcc(G, z).
\end{inlinemath}
Further note that \(y\) can be taken as an optimal solution of a
rational LP of polynomial size (namely the RHS
of~\cref{eq:cut-covering-problem} restricted to the columns~\(\Fcal\)).
We conclude that \(y \in \Rationals^{\Fcal}\) can be represented using
polynomially many bits on the size of the input
\cite[Corollary~10.2a]{Schrijver1986}.
Thus \cref{eq:fcc-decision-version}~is in NP in the Turing
Machine model.
\end{proof}

Let \(G = (V, E)\) be a graph.
\nameandcite{GrotschelLovaszEtAl1981} show that the strong optimization
problem for the class of polytopes \(\CUT(G)\) is solvable if and only
if it is solvable for the class of polytopes \(\abl(\CUT(G))\).
This implies that the following problem is NP-hard in the Turing Machine model:
\begin{equation}
  \label{eq:fcc-optimization-version}
  \text{given an instance }
  (G, z)
  \text{ of \(\fcc\), compute }
  w \in \Rationals_+^{E(G)}
  \text{ such that }
  \mc(G, w) \le 1
  \text{ and }
  \fcc(G, z) = \iprodt{w}{z}.
\end{equation}
In particular, \cite[Section~7]{GrotschelLovaszEtAl1981} proves
intractability of computing the fractional chromatic number from
intractability of computing the weighted maximum clique problem --- a
completely analogous situation to~\cref{eq:fcc-optimization-version}.

Our approximation algorithms rely on the fact that we have a
\emph{tractable} positive definite monotone gauge \(\GW^{\polar}\)
which approximates \(\fcc\).
\cref{thm:ellipsoid-conversion} implies that gauges which approximate
the value of the fractional cut-covering number ``too well'' cannot be
tractable unless P = NP.

\begin{theorem}
\label{thm:ellipsoid-conversion}

Let \(\alpha \in (0, 1)\).
Assume that, for every graph \(G = (V, E)\), there exists a positive
definite monotone gauge \(\psi_G \colon \Lp{E} \to \Lp{}\) such that
\begin{equation}
  \label{item:ellipsoid-gauge-close}
  \psi_G(z)
  \le \fcc(G, z)
  \le \frac{1}{\alpha} \psi_G(z)
  \text{ for every } z \in \Lp{E}.
\end{equation}
Assume a polynomial-time algorithm for either of the following problems:
\begin{enumerate}
\item\label{item:gauge-pairing}
  Given \((G, z, \sigma)\) as input, output \(w \in \Rationals^E\)
  such that \(
    \min\setst{\norm[2]{w - w_0}}{w_0 \in \Lp{E},\,
      \psi_G^{\polar}(w_0) \le 1} \le \sigma
  \) and \(\psi_G(z) \le \iprodt{z}{w} + \sigma \norm[2]{z}\);
\item\label{item:gauge-deciding}
  Given \((G, z, \sigma)\) as input, either
  \begin{enumerate}
  \item conclude that
    \(
    \min\setst{\norm[2]{z - z_0}}{
      z_0 \in \Lp{E},\, \psi_G(z_0) \le 1
    } \le \sigma
    \) or
  \item output \(w \in \Rationals_+^E\) such that
    \(\psi_G^{\polar}(w) \le \iprodt{w}{z} + \sigma \norm[2]{w}\).
  \end{enumerate}
\end{enumerate}
In both cases, the input is a graph \(G = (V, E)\), a vector \(z \in
\Rationals_+^E\), and \(\sigma \in \Rationals_+\).
Then for every \(\beta \in (0, \alpha)\), there exists a polynomial-time
algorithm that, given a graph \(G = (V, E)\) and \(w \in
\Rationals_+^E\) as input, outputs \(q \in \Rationals\) such that
\(\beta q \le \mc(G, w) \le q\).
\end{theorem}
\begin{proof}

Let \(G = (V, E)\) be a graph.
Define \(K \coloneqq \setst{z \in \Lp{E}}{\psi_G(z) \le 1}\), so
\(\max\setst{\iprodt{w}{z}}{z \in K} = \psi_G^{\polar}(w)\) for every
\(w \in \Lp{E}\).
Moreover, as \(\abl(K) = \setst{w \in \Lp{E}}{\psi_G^{\polar}(w) \le
1}\), we have that \(\max\setst{\iprodt{z}{w}}{w \in \abl(K)} =
\psi_G(z)\) for every \(z \in \Lp{E}\).
The weak optimization problem for \(\abl(K)\) is
\begin{equation}
  \label{eq:weak-optimization-problem-abl-K}
  \begin{aligned}
    \text{Given }
    z \in \Rationals_+^E
    \text{ and }
    \sigma > 0
    &\text{ as input, compute }
      \bar{w} \in \Rationals^E
      \text{ such that }\\
    & \min\setst{\norm[2]{\bar{w} - w}}{w \in \abl(K)} \le \sigma, \text{ and}\\
    & \max\setst{\iprodt{z}{w}}{w \in \abl(K)} \le \iprodt{\bar{w}}{z} + \sigma \norm[2]{z}.
  \end{aligned}
\end{equation}
Whereas \cite{GrotschelLovaszEtAl1981} does not multiply \(\sigma\) by
\(\norm[2]{z}\) in the second guarantee, we can easily accomplish so
by normalizing \(z\) before using the oracle.
The weak separation problem for \(K\) is
\begin{equation}
  \label{eq:weak-separation-problem-1}
  \begin{aligned}
    \text{Given }
    \bar{z} \in \Rationals_+^E
    \text{ and }
    \sigma > 0
    &\text{ as input, either }\\
    & \text{conclude that } \min\setst{\norm[2]{\bar{z} - z}}{z \in K} \le \sigma, \text{ or}\\
    & \text{compute } w \in \Rationals_+^E \text{ such that }
      \max\setst{\iprodt{w}{z}}{z \in K} \le \iprodt{w}{\bar{z}} + \sigma \norm[2]{w}.
  \end{aligned}
\end{equation}
Here we have used that \(K\) is lower-comprehensive to assume that \(w
\ge 0\); more precisely, that given any \(w \in \Rationals^E\)
produced by the oracle,  we can pick its non-negative part \(w_+ \in
\Rationals_+^{E}\).
\Cref{item:gauge-pairing,item:gauge-deciding} correspond,
respectively, to solving the
problems~\cref{eq:weak-optimization-problem-abl-K,%
  eq:weak-separation-problem-1} in polynomial time.
\cite[Theorem~3.1]{GrotschelLovaszEtAl1981} shows that, if we can
solve~\cref{eq:weak-separation-problem-1} in polynomial time, then we
can solve the weak optimization problem over \(K\) in polynomial time:
\begin{equation}
  \label{eq:weak-optimization-problem}
  \begin{aligned}
    \text{Given }
    w \in \Rationals_{+}^E
    \text{ and }
    \sigma > 0
    & \text{ as input},
    \text{ compute } \bar{z} \in \Rationals^E \text{ such that }\\
    & \min\setst{\norm[\infty]{\bar{z} - z}}{z \in K} \le \sigma, \text{ and} \\
    & \max\setst{\iprodt{w}{z}}{z \in K} \le \iprodt{w}{\bar{z}} + \sigma\norm[1]{w}.
  \end{aligned}
\end{equation}
For convenience, we have changed from the Euclidean norm to
\(\norm[\infty]{\cdot}\) and \(\norm[1]{\cdot}\), which we can do
since \(\norm[\infty]{\cdot} \le \norm[2]{\cdot} \le
\sqrt{\card{E}}\,\norm[1]{\cdot}\) and the oracle has a running time
bounded by a polynomial on \(\log(1/\sigma)\).
\cite[Corollary~3.5]{GrotschelLovaszEtAl1981} shows that if we can
solve the weak optimization problem over \(\abl(K)\) in polynomial
time, we can solve the weak optimization problem over \(K\).
Hence if we can solve~\cref{eq:weak-optimization-problem-abl-K}, we
can also solve~\cref{eq:weak-optimization-problem}.
Thus we assume we can solve~\cref{eq:weak-optimization-problem} in
polynomial time for every graph \(G\).

From \cref{item:ellipsoid-gauge-close} and
\Cref{prop:1,prop:bound-conversion} we have that
\begin{equation}
  \label{eq:42}
  \alpha \psi_G^{\polar}(w)
  \le \mc(G, w)
  \le \psi_G^{\polar}(w)
  \text{ for every }
  w \in \Lp{E}.
\end{equation}
Let \(\beta \in (0, \alpha)\), and set \(\tau \coloneqq 1 - \beta/\alpha\).
Set
\[
  \sigma \coloneqq \frac{1}{4\alpha}\frac{\tau}{1 - \tau}.
\]
Let \(w \in \Rationals_+^{E}\), and let \(\bar{z} \in \Rationals^{E}\) be the
output of the oracle in~\cref{eq:weak-optimization-problem} for input
\(w\) and \(\sigma\).
Note that we are not assuming that \(\bar{z} \ge 0\).
Let \(z_0 \in \Lp{E}\) be such that \(\norm[\infty]{\bar{z} - z_0} \le
\sigma\) and \(\psi_G(z_0) \le 1\).
Then
\begin{detailedproof}
  \mc(G, w)
  &\le \psi_G^{\polar}(w)
  &&\text{by~\cref{eq:42}}\\
  &\le \iprodt{w}{\bar{z}} + \sigma \norm[1]{w}
  &&\text{by~\cref{eq:weak-optimization-problem},
  as \(\psi_G^{\polar}(w) = \max\setst{\iprodt{w}{z}}{z \in K}\)}\\
  &\le \iprodt{w}{z_0} + \iprodt{w}{(\bar{z} - z_0)_{+}} + \sigma \norm[1]{w}
  &&\text{since \(w \ge 0\) and \(\bar{z} \le z_0 + (\bar{z} - z_0)_+\)}\\
  &\le \iprodt{w}{z_0} + \norm[1]{w} \norm[\infty]{(\bar{z} - z_0)_{+}} + \sigma \norm[1]{w}\\
  &\le \iprodt{w}{z_0} + \norm[1]{w} \norm[\infty]{\bar{z} - z_0} + \sigma \norm[1]{w}\\
  &\le \iprodt{w}{z_0} + 2\sigma\norm[1]{w}
  &&\text{since \(\norm[\infty]{\bar{z} - z_0} \le \sigma\)}\\
  &\le \psi_G^{\polar}(w)\psi_G(z_0) + 2\sigma\norm[1]{w}
  &&\text{since \(z_0 \ge 0\) as \(z_0 \in K\)}\\
  &\le \psi_G^{\polar}(w) + 2\sigma\norm[1]{w}
  &&\text{since \(\psi_G(z_0) \le 1\) as \(z_0 \in K\)}\\
  &\le \frac{1}{\alpha}\mc(G, w) + 2\sigma\norm[1]{w}
  &&\text{by~\cref{eq:42}}\\
  &\le \paren[\Big]{
    \frac{1}{\alpha} + 4\sigma
  } \mc(G, w)
  &&\text{since \(\tfrac{1}{2}\norm[1]{w} \le \mc(G, w)\).}
\end{detailedproof}
Since
\[
  \paren[\bigg]{
    \frac{1}{\alpha} + 4\sigma
  }\mc(G, w)
  = \frac{1}{\alpha}\paren[\bigg]{
    1 + \frac{\tau}{1 - \tau}
  }\mc(G, w)
  = \frac{1}{\alpha(1 - \tau)}\mc(G, w)
  = \frac{1}{\beta}\mc(G, w),
\]
we have
\[
  \mc(G, w)
  \le \iprodt{w}{\bar{z}} + \sigma\norm[1]{w}
  \le \frac{1}{\beta}\mc(G, w).
\]
Since \(q \coloneqq \iprodt{w}{\bar{z}} + \sigma\norm[1]{w}\) is
computable from the output of~\cref{eq:weak-optimization-problem},
the proof is done.
\end{proof}

Let \(\alpha \in (0, 1)\).
Assume that for every graph \(G = (V, E)\), there exists a positive
definite monotone gauge \(\psi_G \colon \Lp{E} \to \Lp{}\) such that
\(\psi_G(z) \le \fcc(G, z) \le \frac{1}{\alpha} \psi_G(z)\)
for every \(z \in \Lp{E}\).
\Cref{thm:ellipsoid-conversion} provides precise statements, taking
into account the finite arithmetic of Turing machines.
Problem (1) in \cref{thm:ellipsoid-conversion} can be interpreted as
the problem of computing a \(w \in \Lp{E}\) paired to the input \(z
\in \Lp{E}\) via the gauge \(\psi_G\).
Recalling~\cref{eq:optw-def}, this is analogous to computing an
element in \(w \in \optw_G(z)\) for a given input \(z \in \Lp{E}\).
In particular, solving problem (1) implies that given \(z \in
\Lp{E}\), one can compute \(w \in \Lp{E}\) such that \((w, z)\) is an
\(\alpha\)-pairing.
Problem (2) in \cref{thm:ellipsoid-conversion} can be interpreted as
deciding if \(\psi_G(z) \le 1\), and when that is not the case,
computing \(w \in \Lp{E}\) certifying \(\psi_G(z) > 1\).

\begin{corollary}
\label{cor:pairing-NP-hard}
Let \(\alpha \in (\frac{16}{17}, 1)\).
Let \(\psi_G \colon \Lp{E} \to \Lp{}\) be a family of positive
definite monotone gauges for every graph \(G\).
Further assume that \(\psi_G(z) \le \fcc(G, z) \le \frac{1}{\alpha}
\psi_G(z)\) for every graph \(G = (V, E)\) and \(z \in \Lp{E}\).
Both of the following problems are NP-hard:
\begin{enumerate}
\item
  Given \((G, z, \sigma)\) as input, output \(w \in \Rationals^E\)
  such that \(
    \min\setst{\norm[2]{w - w_0}}{w_0 \in \Lp{E},\,
      \psi_G^{\polar}(w_0) \le 1} \le \sigma
  \) and \(\psi_G(z) \le \iprodt{z}{w} + \sigma \norm[2]{z}\);
\item
  Given \((G, z, \sigma)\) as input, either
  \begin{enumerate}
  \item conclude that
    \(
    \min\setst{\norm[2]{z - z_0}}{
      z_0 \in \Lp{E},\, \psi_G(z_0) \le 1
    } \le \sigma
    \) or
  \item output \(w \in \Rationals_+^E\) such that
    \(\psi_G^{\polar}(w) \le \iprodt{w}{z} + \sigma \norm[2]{w}\).
  \end{enumerate}
\end{enumerate}
In both cases, the input is a graph \(G = (V, E)\), a vector \(z \in
\Rationals_+^E\), and \(\sigma \in \Rationals_+\).
\end{corollary}
\begin{proof}
By \cite{Hastad2001,TrevisanSorkinSudanWilliamson2000}, it is NP-hard
to compute \(q \in \Rationals\) such that \(\beta q \le \mc(G, w) \le
q\) for every \(\beta \in (\frac{16}{17}, 1]\).
The result follows from \cref{thm:ellipsoid-conversion}.
\end{proof}

\begin{corollary}
\label{cor:UGC}
Let \(\alpha \in (\GWalpha, 1)\).
Let \(\psi_G \colon \Lp{E} \to \Lp{}\) be a family of positive
definite monotone gauges for every graph \(G\).
Further assume that \(\psi_G(z) \le \fcc(G, z) \le \frac{1}{\alpha}
\psi_G(z)\) for every graph \(G = (V, E)\) and \(z \in \Lp{E}\).
Assuming the Unique Games Conjecture, both of the following problems
are NP-hard:
\begin{enumerate}
\item
  Given \((G, z, \sigma)\) as input, output \(w \in \Rationals^E\)
  such that \(
    \min\setst{\norm[2]{w - w_0}}{w_0 \in \Lp{E},\,
      \psi_G^{\polar}(w_0) \le 1} \le \sigma
  \) and \(\psi_G(z) \le \iprodt{z}{w} + \sigma \norm[2]{z}\);
\item
  Given \((G, z, \sigma)\) as input, either
  \begin{enumerate}
  \item conclude that
    \(
    \min\setst{\norm[2]{z - z_0}}{
      z_0 \in \Lp{E},\, \psi_G(z_0) \le 1
    } \le \sigma
    \) or
  \item output \(w \in \Rationals_+^E\) such that
    \(\psi_G^{\polar}(w) \le \iprodt{w}{z} + \sigma \norm[2]{w}\).
  \end{enumerate}
\end{enumerate}
In both cases, the input is a graph \(G = (V, E)\), a vector \(z \in
\Rationals_+^E\), and \(\sigma \in \Rationals_+\).
\end{corollary}
\begin{proof}
By \cite[Theorem~1]{KhotKindlerEtAl2007}, assuming the Unique Games
Conjecture, it is NP-hard to compute \(q \in \Rationals\) such that
\(\beta q \le \mc(G, w) \le q\) with \(\beta \in (\GWalpha, 1]\).
The result follows from \cref{thm:ellipsoid-conversion}.
\end{proof}

\subsection{Approximation Factor}
\label{ssec:approximation-factor}
The algorithms we presented work in two steps: first they solve an SDP
relaxation, and then they employ a rounding procedure to convert
nearly optimal SDP solutions into actual combinatorial solutions ---
namely, (the shore of) a cut or a fractional cut cover.
Each of these two stages impacts the approximation factor.
Let \(G = (V, E)\) be a graph.
The \emph{integrality ratio of a maximum cut instance \((G, w)\) (with
  respect to \(\GW\))} is \(\mc(G, w)/\GW(G, w)\).
The \emph{integrality ratio of a fractional cut-covering instance
\((G, z)\) (with respect to \(\GW^{\polar}\))} is \(\GW^{\polar}(G,
z)/\fcc(G, z)\).
In either case, the integrality ratio is a number between \(0\) and
\(1\) capturing how well the semidefinite programming relaxation
approximates the actual optimal value of the problem at hand.

Our theory ties the integrality ratios of both problems.
Let \(G = (V, E)\) be a graph, and let \(w \in \Lp{E}\).
Then
\begin{equation}
  \label{eq:integrality-ratio-conversion}
  \frac{\GW^{\polar}(G, z)}{\fcc(G, z)}
  \le \frac{\mc(G, w)}{\GW(G, w)}
  \text{ for every }
  z \in \optz_G(w).
\end{equation}
Indeed, this follows from
\(\GW(G, w) \GW^{\polar}(G, z) = \iprodt{w}{z} \le \mc(G, w) \fcc(G,
z)\).
In this way, the set \(\optz_G(w)\) describes instances \((G, z)\) of
the fractional cut covering problem with the same or worse integrality
ratio.
This construction can be made algorithmic via the tools developed
in~\cref{sec:certificates}.
If the graph \(G\) is edge-transitive, then one can prove that
\((\ones, \ones) \in \GWOpt(G)\) and
\(
  \mc(G)\fcc(G)
  = \card{E}
  = \GW(G)\GW^{\polar}(G)
\).
In particular, equality holds
in~\cref{eq:integrality-ratio-conversion}, and the cycle \(C_5\) on five
vertices is a concrete example with bad integrality ratio for
both problems, as
\[
  0.878
  \approx \GWalpha
  \le \frac{\GW^{\polar}(C_5)}{\fcc(C_5)}
  = \frac{\mc(C_5)}{\GW(C_5)}
  \approx 0.884,
\]
where unweighted graph parameters are evaluated from their weighted
versions with~\(\ones\) as edge weights.
The integrality ratio can be arbitrarily close to \(\GWalpha\).

\begin{proposition}
\label{prop:FeigeSchechtmanExample}
For every \(\eps > 0\), there exists a graph \(G\) such that for
every \(z \in \optz_G(\ones)\),
\[
  \GWalpha
  \le \frac{\GW^{\polar}(G, z)}{\fcc(G, z)}
  \le \frac{\mc(G)}{\GW(G)} \le \GWalpha + \eps.
\]
\end{proposition}
\begin{proof}

Let \(\eps > 0\).
\nameandcite{FeigeSchechtman2002} prove that there exists a graph
\(G\) such that
\[
  \frac{\mc(G)}{\GW(G)} \le \GWalpha + \eps.
\]
The result follows from~\cref{cor:1,eq:integrality-ratio-conversion}.
\end{proof}

\cref{prop:FeigeSchechtmanExample}, despite showing that the
approximation factor \(\GWalpha\) is tight in our analysis, leaves
open the possibility that strengthening our semidefinite programming
relaxations could lead to better results, even if we keep the same
rounding procedure.
This is not the case, as there exist graphs in which the relaxation is
tight, but the rounding procedure still produces solutions with
approximation factor as bad as \(\GWalpha\).
We now present such an example.
Similar to the proof of \cref{prop:FeigeSchechtmanExample}, our result
builds on what is known in the literature, and it exploits certain
simple eigenvalue bounds.

Let \(G = (V, E)\) be a graph, and let \(L_G \coloneqq
\Laplacian_G(\ones)\) denote the unweighted Laplacian of \(G\).
It holds that
\begin{equation}
  \label{eq:eigenvalues-bounds}
  \GW(G) \le \frac{\card{V}}{4}\lambdamax(L_G)
  \text{ and }
  \GW^{\polar}(G)
  \ge
  \frac{4 \card{E}}{\card{V}}
  \frac{1}{\lambdamax(L_G)}.
\end{equation}
One may easily produce feasible solutions for~\cref{eq:GW-gaugef-rho}
and~\cref{eq:GW-polar-supf} that prove~\cref{eq:eigenvalues-bounds}.
The first inequality is a well-known bound on the maximum cut value,
easily obtained from the characterization of eigenvalues in terms of
Rayleigh quotients; see, for example,
\cite[Lemma~13.7.4]{GodsilRoyle2001}.
More generally, it holds that
\[
  \GW(G, w)
  =
  \min \setst[\Big]{
    \tfrac{\card{V}}{4}
    \lambdamax(\Laplacian_G(w) + \Diag(u))
  }{
    u \in \Reals^V,\,
    \iprodt{u}{\ones} = 0
  };
\]
this formulation was introduced in \cite{DelormePoljak1993}.
The second inequality in~\cref{eq:eigenvalues-bounds} is used ---
implicitly --- in \cite{Samal2015}.
Let \(a, b \in \Naturals\) be such that \(b \le a\).
Let \(\Ham(a, b)\) denote the \emph{Hamming distance graph}, which has
as vertex set \(\set{0, 1}^a\), with two vertices adjacent if they differ
in at least \(b\) entries.
We denote by \(\Ham_{=}(a, b)\) the \emph{exact Hamming distance
graph}, which is the spanning subgraph of \(\Ham(a, b)\) where two
vertices are adjacent when they differ in exactly \(b\) entries.
It is known \cite{AlonSudakov2000,Samal2015} that
\begin{equation}
  \label{eq:eigenvalue-exact-Hamming}
  \text{if }
  b \le a < 2b
  \text{ and \(b\) is even, then }
  \lambdamin\paren[\big]{
    A(\Ham_{=}(a, b))
  }
  = \binom{a}{b}\paren[\bigg]{1 - \frac{2b}{a}},
\end{equation}
where \(A_G \in \Sym{V}\) denotes the adjacency matrix of a graph
\(G = (V, E)\).
The interest on the smallest eigenvalue of the adjacency matrix in
both works \cite{AlonSudakov2000,Samal2015} stems directly
from~\cref{eq:eigenvalues-bounds}: since \(G \coloneqq \Ham_{=}(a, b)\) is
\(\binom{a}{b}\)-regular, it follows that \(\lambdamax(L_G) =
\binom{a}{b} - \lambdamin(A_G)\).

\begin{proposition}
\label{prop:AlonSudakovExample}

For every real number \(\beta > \GWalpha\) there exists a graph \(G =
(V, E)\), as well as witnesses \((\rho, x)\) and \((\mu, Y)\) of the
membership \((\ones, \ones) \in \GWOpt(G)\), satisfying the following
conditions.
One has that \(\GW(G) = \mc(G)\) and \(\GW^{\polar}(G) = \fcc(G)\).
Furthermore,
\begin{equation}
  \label{eq:AlonSudakovGap}
  \Ebb\paren[\big]{\card{\delta(\GWrv(Y))}}
    < \beta \mc(G)
    \text{ and }
  \tfrac{1}{\beta} \fcc(G) < \iprodt{\ones}{y},
\end{equation}
where \(y \in \Lp{\Powerset{V}}\) is defined as in~\cref{eq:13}.
\end{proposition}
\begin{proof}

By~\cref{eq:GWalpha-def},
there exists a rational number \(\zeta \in (-1, 0)\) such that
\begin{equation}
\label{eq:26}
  \GWalpha
  \le \frac{2}{1 - \zeta} \frac{\arccos (\zeta)}{\pi}
  < \beta.
\end{equation}
Let \(a \in \Naturals\) be such that \(b \coloneqq (1 - \zeta)
\frac{a}{2}\) is an even natural number, so that
\begin{equation*}
  \zeta = 1 - 2b/a.
\end{equation*}
We now prove the statement holds for \(G \coloneqq \Ham_{=}(a, b)\).

Let \(U \in \Reals^{[a] \times V(G)}\) be defined by \(Ue_s \coloneqq 2s
- \ones \in \Reals^{[a]}\) for every \(s \in V(G) = \set{0, 1}^{[a]}\).
We claim that
\begin{equation}
  \label{eq:optimal-solutions-exact-Hamming}
  (\rho, x) \coloneqq \paren{
    \tfrac{b}{a} \card{E},
    \tfrac{1}{2} \tbinom{a-1}{b-1}\ones
  }
  \text{ and }
  (\mu, Y) \coloneqq \paren{
    \tfrac{a}{b}, \tfrac{1}{b} U^\transp U
  }
  \text{ witness the membership }
  (\ones, \ones) \in \GWOpt(G).
\end{equation}
Feasibility of \((\tfrac{a}{b}, \tfrac{1}{b}U^\transp U)\)
in~\cref{eq:GW-polar-def-intro} with \(z \coloneqq \ones\) follows
directly from the definition of \(U\) and~\(G\), since
\begin{equation}
  \label{eq:25}
  Y_{ij} = \tfrac{a}{b}\zeta
  \qquad
  \mathrlap{
    \text{for every }ij \in E,
  }
\end{equation}
whereas feasibility of \((\rho,
x)\) in~\cref{eq:GW-gaugef-rho} with \(w \coloneqq \ones\) follows
from~\cref{eq:eigenvalue-exact-Hamming}, since \(L_G = \binom{a}{b}I
- A_G\) as \(G\) is \(\tbinom{a}{b}\)-regular.
It is immediate that \(\rho\mu = \card{E}\).
Thus, \cref{eq:optimal-solutions-exact-Hamming} is proved.
We claim that
\begin{equation}
  \label{eq:AlonSudakov-exact-pairing}
  (\ones, \ones)
  \text{ is an exact pairing}.
\end{equation}
Note that, in particular, \cref{eq:AlonSudakov-exact-pairing} ensures
that \(\GW(G) = \mc(G)\) and \(\GW^{\polar}(G) = \fcc(G)\).
For every \(i \in [a]\), set \(S_i \coloneqq \setst{s \in \set{0,
    1}^a}{s_i = 1}\).
One may easily check that \(\card{\delta(S_i)} = \binom{a - 1}{b -
  1}2^{a - 1}\) for every \(i \in [a]\).
By~\cref{eq:optimal-solutions-exact-Hamming}
and~\cref{eq:GW-approximation}, all such cuts are maximum, since
\begin{equation*}
  \GW(G)
  = \frac{b}{a}\card{E}
  = \frac{b}{a}\binom{a}{b} 2^{a - 1}
  = \binom{a - 1}{b - 1}2^{a - 1}
  = \card{\delta(S_i)}
  \le \mc(G)
  \le \GW(G).
\end{equation*}
Now consider the fractional cut cover \(\bar{y} \coloneqq
\tfrac{1}{b}\sum_{i \in [a]} e_{S_i}\).
By definition of \(\Ham_{=}(a, b)\), each edge belongs to \(b\) of the
cuts in \(\setst{\delta(S_i)}{i \in [a]}\), so \(\sum_{S \subseteq V}
\bar{y}_S \incidvector{\delta(S)} = \ones\).
Hence~\cref{eq:optimal-solutions-exact-Hamming}
and~\cref{eq:GW-polar-approx-pledge} imply that \(\bar{y}\) is an
optimal fractional cut cover, as
\begin{equation*}
  \frac{a}{b}
  = \GW^{\polar}(G)
  \le \fcc(G)
  \le \iprodt{\ones}{\bar{y}}
  = \frac{a}{b}.
\end{equation*}
This concludes the proof of~\cref{eq:AlonSudakov-exact-pairing}.
By~\cref{eq:GW-Y-edge-marginal}, \cref{eq:25}, and linearity of
expectation,
\[
  \Ebb\paren[\big]{
    \card{\delta(\GWrv(Y))}
  }
  = \sum_{ij \in E} \prob\paren{ij \in \delta(\GWrv(Y))}
  = \frac{\arccos(\zeta)}{\pi} \card{E}.
\]
Let \(y,\, p \in \Lp{\Powerset{V}}\) and \(\bar{\nu} \in \Lp{}\) be
defined from \((\mu, Y)\) as in~\cref{eq:13}.
Then for every \(ij \in E\),
\[
  e_{ij}^\transp\paren[\bigg]{\,
    \sum_{S \subseteq V} p_S \incidvector{\delta(S)}
  }
  = \prob\paren{ij \in \delta(\GWrv(Y))}
  = \frac{\arccos(\zeta)}{\pi}
\]
by~\cref{eq:GW-Y-edge-marginal}.
Hence \(\bar{\nu} = \pi/\arccos(\zeta)\).
Since \(a/b = 2/(1 - \zeta)\), we have that
\begin{equation*}
  \frac{
    \Ebb\paren[\big]{\card{\delta(\GWrv(Y))}}
  }{\mc(G)}
  = \frac{2}{1 - \zeta}\frac{\arccos(\zeta)}{\pi}
  = \frac{\fcc(G)}{\iprodt{\ones}{y}},
\end{equation*}
and this number is \(< \beta\) by~\cref{eq:26}.
\end{proof}

\nameandcite{Karloff1999} studied a family of graphs closely related
to the one featured in the proof of \Cref{prop:AlonSudakovExample},
using it to bound the quality of the approximation factor obtained
by \nameandcite{GoemansWilliamson1995}.
The~specific construction in our proof is a simplification of this
work due to \nameandcite{AlonSudakov2000}.
The relevance of \cref{prop:AlonSudakovExample} to our algorithms
inherits several aspects of the relevance of these examples to the
maximum cut setting.
Let \(\beta\) be a real number such that \(\beta > \GWalpha\).
\Cref{prop:AlonSudakovExample} presents an obstruction for the use of
the randomized hyperplane technique to produce \(\beta\)-certificates.
Let \(G = (V, E)\) be a graph whose existence is ensured by
the proposition.
Then \(G\) defines instances where the SDP
relaxations~\cref{eq:GW-intro-def} and~\cref{eq:GW-polar-def-intro}
are tight --- i.e., the integrality ratio is one ---, but where the
rounding itself can be responsible for the approximation factor
\(\GWalpha\) of the final algorithm.
The first inequality in~\cref{eq:AlonSudakovGap} states that the
expected value of a cut produced by the random hyperplane technique
will be too small for the desired approximation factor~\(\beta\).
The second inequality in~\cref{eq:AlonSudakovGap} does not directly
translate to the setting of
our algorithm, which takes finitely many samples and outputs a
surrogate for the \(y \in \Lp{\Powerset{V}}\) defined in~\cref{eq:13}.
\Cref{prop:AlonSudakovExample} and~\cref{eq:y-as-asymptotic-behavior}
do imply that, almost surely, the objective value \(\mu_t/t\) obtained
from~\cref{eq:hat-y-def} deteriorates above \(\tfrac{1}{\beta}
\fcc(G)\) for sufficiently large values of \(t \in \Naturals\).
Note, however, that similar to how~\cref{eq:AlonSudakovGap} allows for
a finite number of samples to define a fractional cut cover with
objective value better than \(\tfrac{1}{\beta} \fcc(G)\), it also allows
for a single cut to have objective value better than \(\beta \mc(G)\).

Let \(G = (V, E)\) be a graph and let \(w \in \Lp{E}\).
\Cref{prop:FeigeSchechtmanExample,prop:AlonSudakovExample} describe
limitations of the approach we chose for producing approximation
algorithms for the maximum cut and fractional cut-covering problems.
The \emph{Unique Games Conjecture} \cite{KhotKindlerEtAl2007}
implies that it is NP-hard to compute, given an instance \((G, w)\) of
the maximum cut problem as input, an
upper bound to \(\mc(G, w)\) with a better approximation guarantee
than~\cref{eq:GW-approximation}.
Recall from~\cref{cor:UGC} that this conjectured optimality extends to
the SDP in~\cref{eq:GW-polar-def-intro} and the fractional
cut-covering problem: obtaining any better approximation factor on the
objective value for~\cref{alg:1} would disprove the conjecture.
It is the case, however, that even under UGC, our developments leave
open the possibility of an approximation algorithm for the fractional
cut-covering problem which is \emph{not} based on a positive definite
monotone gauge.
Although the reader could see this as an invitation to develop
non-convex approximation algorithms for the weighted fractional
cut-covering problem, this could also simply hint at the existence of
a (not yet developed) direct reduction from the Unique Label Cover
problem to the fractional cut-covering problem.

\section{Concluding Remarks}
\label{sec:conclusion}

We summarize the main contributions of this paper and discuss future
directions.
Key to our contributions is precisely describing the relationship
between the (weighted) maximum cut and the weighted fractional
cut-covering problems through the lens of gauge duality: the functions
\(\mc(G, \cdot)\) and \(\fcc(G, \cdot)\) form a gauge dual pair.
Crucially, the SDP relaxation for the maximum cut problem, utilized
by~\nameandcite{GoemansWilliamson1995}, and the SDP relaxation we
provided for the fractional cut-covering problem share the same
property: \(\GW(G, \cdot)\) and \(\GW^\polar(G, \cdot)\) form a gauge
dual pair (see \Cref{prop:1}).
This explicit connection establishes the background and foundation for
the development of our algorithms.
Gauge duality promptly yields a bound conversion procedure, enabling
us to extend the \(\GWalpha\) approximation ratio between
\(\mc(G, \cdot)\) and \(\GW(G, \cdot)\) to a \(1/\GWalpha\)
approximation ratio between \(\fcc(G, \cdot)\) and
\(\GW^\polar(G, \cdot)\) (see \Cref{prop:bound-conversion}).

The understanding of this connection organizes our efforts in
algorithmic design:
\begin{enumerate}[(i)]
\item Optimal solutions \(Y \in \Psd{V}\) of the SDP relaxation
  \cref{eq:GW-intro-def} for maximum cut are employed
  in~\cite{GoemansWilliamson1995} to sample (the shore of) a cut by
  the randomized hyperplane technique, with \(Y\) as the generating
  parameter.
  We extend this technique to the fractional cut-covering problem in
  two stages: initially by showing that the marginal probabilities for
  the shores provide a fractional cut cover with the same
  approximation quality, and
  subsequently by sparsifying our cover through a polynomial-time
  randomized procedure (see
  \cref{prop:sdp-rounding,prop:probabilistic-rounding-cover-pre,%
    thm:rounding-algorithm}).
\item We address the problem of simultaneously obtaining
  approximate solutions for the maximum cut and the weighted
  fractional cut-covering problems with certificates of approximate
  optimality.
  This inspires our definition of the set \(\GWOpt(G)\), which
  precisely links tightness of
  \(\iprodt{w}{z} \leq \GW(G,w)\GW^{\polar}(G,z)\) in gauge duality with the
  optimality of solutions for the SDP relaxations.
  This directly motivates the definition of \(\beta\)-pairings and
  \(\beta\)-certificates, which are essential to our simultaneous
  certification process (see
  \cref{prop:real-number-model-certificate}, and
  \cref{thm:certificates-from-w,thm:certificates-from-z}).
\item Given the strong algorithmic focus of our work, it is
  crucial to deal with the fact that exact optimal solutions for SDPs
  may not be computable in polynomial time.
  We combine the strategies from the previous items: in
  \cref{subsec:algorithimic_certicates}, we introduce the set
  \(\GWOpt_{\eps,\sigma}(G)\), which is a perturbed version of
  \(\GWOpt(G)\), and then develop a randomized approximation algorithm
  which requires only nearly optimal solutions for the SDPs.
\item When run on a connected graph with \(n\) vertices and \(m\)
  edges, our algorithms rely on obtaining a nearly optimal solution
  \(Y\) to an instance of the relevant SDP problems, which can be done
  in \(O(m^4)\) time, followed by the rounding procedure on~\(Y\),
  which involves a single \(n \times n\) Cholesky factorization,
  \(O(n\ln(n))\) samples from a standard Gaussian distributions, and
  \(O(n^2 \ln(n))\) extra work.
\item \Cref{sec:possible-improvements} addresses many aspects of our
  approach that are optimal.
  Many of them are mirror images under gauge duality of corresponding
  properties for the Goemans and Williamson's algorithm for maximum cut
  (see \cref{prop:AlonSudakovExample}).
\end{enumerate}

Throughout our work we have assumed the real-number machine model
\cite{BlumCuckerEtAl1998} with two extra oracles: one sampling from a
standard normal distribution, and one computing Cholesky
factorizations.
This second assumption is equivalent to assuming exact computation of
square roots of real positive numbers.
Despite~these assumptions, our results build towards an implementation
in the Turing machine model.
Explicitly, the Slater points for~\cref{eq:GW-eps-polar-def} used
in~\cref{sec:solvers} may allow one to utilize the work of
\nameandcite{deKlerkVallentin2016} to conclude that both
\(\GW_{\eps}\) and \(\GW_{\eps}^{\polar}\) may be computed in
polynomial time on a Turing machine.
\Cref{prop:sampling-GWOpt} identifies the main computational steps
required to round these optimal solutions, and it may be extended to
the probabilistic Turing machine model with an appropriate analysis of
approximate square root computations.

The interplay of duality with probabilistic aspects that permeates
this work suggests
interesting avenues for future research.
Note that, essential to our approach, Goemans and Williamson's
randomized hyperplane technique casts an optimal solution~\(Y\)
for~\cref{eq:GW-intro-def} as a \emph{distribution} of cuts.
This enables edges that do not occur in any heavy cut to be probably
covered after appropriate thickening.
The SDP perturbation \cref{eq:GW-eps-polar-gaugef} makes this even
more robust, since for any feasible solution~\((\mu,Y)\), sampling
from \(\GWrv(Y)\) covers every edge with probability at least
\(\sqrt{2\eps}/\pi\) (by~\cref{eq:gap:GW-eps-cover}) and \(Y\) has
full rank.
The analysis in~\cite{GoemansWilliamson1995} for maximum cut relies on
expected values, whereas for fractional cut covers we rely on
concentration inequalities (see
\cref{prop:probabilistic-rounding-cover-pre,%
  prop:GWOpt-approx-sampling}).
In this respect, it would be interesting to find a gauge dual analogue
of the analysis in \cite[Sec.~3.1]{GoemansWilliamson1995}, which
improves the approximation factor~\(\GWalpha\) when the maximum cut is
large.
A similar question can be made about a gauge dual analogue of low-rank
SDP solutions with improved approximation factors, as
in~\cite{AvidorZ05a}.
One~may also ask whether the techniques from~\cite{MahajanR99a} can be
applied to derandomize our algorithms.

Further natural research directions include the use of duality in
close relatives of the maximum cut problem, such as maximum bisection
problem (see, e.g., \cite{PoljakR95a,austrin2016better}) and Nesterov's
generalization to quadratic optimization problems~\cite{Nesterov98a}.
While this paper was under review, there has been a very significant
advance in this research direction: \cite{BPdCSST2024} presents a
generalized framework containing, among other problems, both
\cite{Nesterov98a} and all Boolean constraint satisfaction problems
whose constraints contain at most two variables (Boolean 2-CSPs).
Beyond the maximum cut and fractional cut-covering problems, it is
natural to search for and study other pairs of combinatorial
optimization problems that are linked by gauge duality.
For example, the literature has both time-honored
\cite{GrotschelLovaszSchrijver1986} and more recent
\cite{BenedettoProencadeCarliSilvaEtAl2021} publications leveraging
the fact that
the stability number and the fractional chromatic number of a
graph define gauges dual to each other.
Specially in cases where an approximation algorithm rounds a solution
to a convex relaxation of a combinatorial problem, the (gauge) dual
combinatorial problem may be approximated utilizing the ideas within
this work.
On a broader note, it seems interesting to look for further pair of
problems which can be ``simultaneously approximated''.
\Cref{def:beta-cert} provides a convenient formalization of
simultaneous approximation for this work, but even within the
maximum cut and fractional cut-covering context there are potentially
sensible alternative definitions, as mentioned in
\Cref{re:beta-certificate-alternatives}.

\printbibliography

\clearpage
\appendix
\section{Edges with Arbitrarily Small Probability of being Covered}
\label{sec:small-edges}

\begin{proposition}
\label{prop:small-edges-example}
Let \(\eps \in (0,2)\).
Set \(G \coloneqq K_3\), set \(\beta \coloneqq {4}/{(4 - \eps)^2}\),
and set
\begin{subequations}
  \begin{align}
    \label{eq:tiny-example-w-def}
    w
    &\coloneqq (2 - \eps)\beta(e_{12} + e_{13}) + \beta e_{23} \in \Lp{E(G)},\\
    z
    &\coloneqq e_{12} + e_{13} + \eps e_{23} \in \Lp{E(G)}.
  \end{align}
\end{subequations}
Then
\begin{equation}
  \label{eq:31}
  (1, \bar{x})
  \text{ and }
  (\bar{\mu}, \bar{Y})
  \text{ witness the membership }
  (w, z) \in \GWOpt(G),
\end{equation}
where
\begin{equation*}
  \bar{x}
  \coloneqq
  \frac{1}{4 - \eps}\paren{(2 - \eps)e_1 + e_2 + e_3},
  \quad
  \bar{\mu}
  \coloneqq \frac{4}{4 - \eps},
  \quad
  \text{and}
  \quad
  \bar{Y}
  \coloneqq
  \bar{\mu}
  \begin{bmatrix*}
    1 & \eps/2 -1 & \eps/2 - 1\\
    \eps/2 - 1 & 1 & 1 - 2\eps + \eps^2/2\\
    \eps/2 -1 & 1 - 2\eps + \eps^2/2 & 1
  \end{bmatrix*}.
\end{equation*}
In particular, \((\bar{\mu},\bar{Y})\) is an optimal solution for
the SDP~\cref{eq:GW-polar-def-intro}.
\end{proposition}
\begin{proof}
It is immediate that \(\diag(\bar{Y}) = \bar{\mu}\ones\).
The direct computation
\begin{equation*}
  \frac{(4 - \eps)}{4}
  \bar{Y}
  =
  \begin{bmatrix}
    1 \\ \eps/2 - 1 \\ \eps/2 -1
  \end{bmatrix}
  \begin{bmatrix}
    1 & \eps/2 - 1 & \eps/2 - 1
  \end{bmatrix}
  + \paren{\eps - \eps^2/4}
  \begin{bmatrix}
    0 \\ 1 \\ -1
  \end{bmatrix}
  \begin{bmatrix}
    0 & 1 & -1
  \end{bmatrix}
\end{equation*}
shows that \(\bar{Y} \in \Psd{V}\).
Moreover, from~\cref{eq:Laplacian-adjoint-def} we get that
\[
  \paren[\big]{\tfrac{1}{4}\Laplacian_G^*(\bar{Y})}_{13}
  = \paren[\big]{\tfrac{1}{4}\Laplacian_G^*(\bar{Y})}_{12}
  = \frac{2}{4 - \eps} - \frac{\eps - 2}{4 - \eps}
  = 1
  = z_{12}
  = z_{13}.
\]
Furthermore,
\begin{equation*}
  \paren[\big]{\tfrac{1}{4}\Laplacian_G^*(\bar{Y})}_{23}
  = \frac{2}{4 - \eps} - \frac{2(1 - 2\eps + \eps^2/2)}{4 - \eps}
  = \frac{4\eps - \eps^2}{4 - \eps}
  = \eps
  = z_{23}.
\end{equation*}
Thus, \((\bar{\mu},\bar{Y})\) is feasible
in~\cref{eq:GW-polar-def-intro} for \((G, z)\).
Set \(u \coloneqq (2 - \eps)e_1 + e_2 + e_3\).
Moreover,
\begin{detailedproof}
  \tfrac{1}{4}\Laplacian_G(w)
  &= \tfrac{1}{4}(2 - \eps)\beta \Laplacian_G(e_{12} + e_{13})
  + \tfrac{1}{4}\beta \Laplacian_G(e_{23})
  &&\text{by~\cref{eq:tiny-example-w-def}}\\
  &=
  \tfrac{1}{4}\beta\paren[\Bigg]{
   (2 - \eps)
  \begin{bmatrix}
    \phantom{-}2 & -1 & -1\\
    -1 & \phantom{-}1 & \phantom{-}0\\
    -1 & \phantom{-}0 & \phantom{-}1
  \end{bmatrix}
  +
  \begin{bmatrix}
    0 & \phantom{-}0 & \phantom{-}0 \\
    0 & \phantom{-}1 & -1           \\
    0 & -1           & \phantom{-}1
  \end{bmatrix}
  }\\
  &= \tfrac{1}{4}\beta
  \begin{bmatrix}
    4 - 2\eps & \eps -2 & \eps -2\\
    \eps - 2 & 3 - \eps & -1\\
    \eps - 2 & -1 & 3 - \eps
  \end{bmatrix}\\
  &= \tfrac{1}{4}\beta
  \paren[\Bigg]{
    (4 - \eps)
    \begin{bmatrix}
      2 - \eps & 0 & 0\\
      0 & 1 & 0\\
      0 & 0 & 1
    \end{bmatrix}
    -
    \begin{bmatrix}
      (2 - \eps)^2 & 2 - \eps & 2 - \eps\\
      2 - \eps & 1 & 1\\
      2 - \eps & 1 & 1
    \end{bmatrix}
  }\\
  &= \tfrac{1}{4}\beta\paren[\big]{(4 - \eps)\Diag(u) - \oprod{u}{u}}
  &&\text{by the definition of \(u\)}\\
  &\preceq
  \tfrac{1}{4}\beta(4 - \eps)\Diag(u)\\
  &= \Diag(\bar{x})
  &&\text{since \(4\bar{x} = \beta(4 - \eps) u\)}.\\
\end{detailedproof}
Thus \((1, \bar{x})\) is feasible in~\cref{eq:GW-gaugef-rho}
for \((G, w)\) as \(\iprodt{\ones}{\bar{x}} = 1\).
Since
\[
  \iprodt{w}{z}
  = 2(2 - \eps)\beta + \eps \beta
  = \beta(4 - \eps)
  = \frac{4}{4 - \eps}
  = \bar{\mu},
\]
we obtain~\cref{eq:31} from~\cref{eq:GWOpt-def}.
Note that since \((1, \bar{x})\) is feasible
in~\cref{eq:GW-gaugef-rho}, then \((w, \bar{x})\) is feasible
in~\cref{eq:GW-polar-supf}.
Since~\cref{eq:GW-polar-def-intro} and~\cref{eq:GW-polar-supf} form a
primal-dual pair of SDPs, \(\bar{\mu} = \iprodt{w}{z}\) implies
optimality of \((\bar{\mu}, \bar{Y})\) in~\cref{eq:GW-polar-def-intro}.
\end{proof}

\section{SDP Solutions Defining Fractional Cut Covers with Exponential Support}
\label{sec:exp-supp}

\Cref{thm:exp-supp} below presents a rigorous statement and proof
regarding~\cref{eq:exp-supp-pledge}, i.e., about the exponential
support of a fractional cut cover derived from
\cref{prop:sdp-rounding}.

\begin{theorem}
  \label{thm:exp-supp}
  Let \(n \geq 3\) be an integer.
  Set \((V,E) \coloneqq K_n\) and
  \(\bar{Y} \coloneqq 2I-\tfrac{2}{n}\oprodsym{\ones} \in \Sym{V}\).
  For every \(i \in [n]\), let \(g_i\) be independently sampled from
  the standard normal distribution, and set
  \(h \coloneqq \norm{g}^{-1}g\).
  Then, for each nonempty \(S \subsetneq V\),
  \[
    \prob\paren[\big]{\GWrv(\bar{Y}, h) = S} > 0.
  \]
\end{theorem}

\begin{proof}
  Let \(S \subsetneq V\) be nonempty.
  Set
  \[
    \beta \coloneqq \min\set[\bigg]{\frac{2n - \card{S} - 1}{\card{S}-1},
      \frac{n+\card{S}-1}{n-\card{S}-1}} \in \Reals,
  \]
  where we consider \(1/0\) to be \(+\infty\).
  It is straightforward to check that \(\beta>1\).
  For each \(v \in V\), let \(A_v\) be the event that
  \(1< g_v <\beta\) and let \(B_v\) be the event that
  \(-\beta< g_v <-1\).
  Since \(g_v\) is sampled from the standard normal
  distribution, we have that \(\prob(A_v) = \prob(B_v) \eqqcolon p\)
  is a positive constant depending only on \(\beta\).

  Set
  \[
    D \coloneqq
    \paren[\bigg]{
      \bigcap_{s\in S} A_s
    }
    \cap
    \paren[\bigg]{
      \bigcap_{t\in V\setminus S} B_t
    }.
  \]
  Note that, by the independence of \((g_v)_{v \in V}\), we have that
  \(\prob(D) = p^n > 0\).
  We complete the proof by showing that \(D\) implies that
  \begin{equation}
    \label{eq:explicit-outcome}
    \GWrv(\bar{Y}, h) = S.
  \end{equation}
  We will use the fact that the sampled shore is
  \begin{equation}
    \label{eq:15}
    \begin{split}
      \GWrv(\bar{Y}, h)
      & = \setst{i \in V}{\iprodt{e_i}{\bar{Y}^{\half}h} \ge 0}
        = \setst{i \in V}{\iprodt{e_i}{(I-\tfrac{1}{n}\oprodsym{\ones})h} \ge 0}
      \\
      & = \setst[\Big]{i \in V}{h_i \ge \frac{\iprodt{\ones}{h}}{n}}
        = \setst[\Big]{i \in V}{g_i \ge \frac{\iprodt{\ones}{g}}{n}}
      \\
      & = \setst[\Big]{i \in V}{(n-1) g_i \ge \sum_{v\in V\setminus\set{i}} g_v}.
    \end{split}
  \end{equation}
  Assume that \(D\) holds.
  We first prove `\(\supseteq\)' in~\cref{eq:explicit-outcome}.
  Let \(s \in S\).
  Then
  \begin{equation*}
    \frac{1}{n-1}\sum_{v\in V\setminus\set{s}} g_v
    =
    \frac{1}{n-1}
    \paren[\bigg]{
      \sum_{s'\in S \setminus \set{s}} g_{s'}
      +
      \sum_{t\in V\setminus S} g_t
    }
    <
    \frac{1}{n-1}\left(\beta (\card{S}-1) - (n-\card{S})\right)
    \leq 1 < g_s,
  \end{equation*}
  since our choice of \(\beta\) ensures that
  \(\beta (\card{S}-1) \leq 2n-\card{S}-1\).
  Hence, \(s \in \GWrv(\bar{Y}, h)\) by~\cref{eq:15}.

  We now prove `\(\subseteq\)' in~\cref{eq:explicit-outcome}.
  Let \(t \in V\setminus S\).
  Then
  \begin{equation*}
    \frac{1}{n-1}\sum_{v\in V\setminus\set{t}} g_v
    =
    \frac{1}{n-1}
    \paren[\bigg]{
      \sum_{t'\in V\setminus (S \cup \set{t})} g_{t'}
      +
      \sum_{s\in S} g_s
    }
    >
    \frac{1}{n-1}\left(-\beta(n-\card{S}-1) + \card{S}\right)
    \geq -1 > g_t,
  \end{equation*}
  since our choice of \(\beta\) ensures that
  \(\beta(n-\card{S}-1) \leq n+\card{S}-1\).
  Thus, \(t \notin \GWrv(\bar{Y}, h)\) by~\cref{eq:15}.
\end{proof}

\section{SDP Solvers}
\label{sec:solvers}

In this section, we analyze the running time of an interior-point
method (IPM) to solve the SDP~\cref{eq:GW-eps-polar-gaugef} to near
optimality:
\begin{equation*}
  \tag{\ref{eq:GW-eps-polar-gaugef}}
  \GW_{\eps}^{\polar}(G, z)
  = \min\setst{
    \mu
  }{
    \mu \in \Reals_+,\,
    Y \in \Sym{V},\,
    Y \succeq \mu \eps I,\,
    \tfrac{1}{4}\Laplacian_G^*(Y) \ge z,\,
    \diag(Y) = \mu \ones
  }.
\end{equation*}
\Cref{alg:1,alg:fcc_mc} in Theorems~\ref{thm:rounding-algorithm} and
\ref{thm:certificates-from-z-approx}, resp., rely on obtaining such
nearly optimal solutions for \cref{eq:GW-eps-polar-gaugef}, where the
parameter \(\eps\) lies in \(\halfopen{0,1}\).

For the purposes of stating the running time of IPMs, we will consider
an arbitrary primal SDP in the format
\begin{subequations}
  \label{eq:16}
  \begin{alignat}{3}
    \text{Minimize}   &    & & \iprod{c}{x}                 \\*
    \label{eq:17}
    \text{subject to} &\ \ & & \Acal(x) \succeq_{\Lbb^*} b, \\*
    \label{eq:18}
                      &    & & x \in \Kbb,
  \end{alignat}
\end{subequations}
and its (syntactically symmetric) dual,
\begin{subequations}
  \label{eq:19}
  \begin{alignat}{3}
    \text{Maximize}   &    & & \iprod{b}{y}                \\*
    \text{subject to} &\ \ & & y \in \Lbb,                 \\*
                      &    & & \Acal^*(y) \preceq_{\Kbb^*} c,
  \end{alignat}
\end{subequations}
where
\begin{subequations}
  \begin{alignat}{4}
    \label{eq:21}
    \Kbb
    & \coloneqq \Psd{n_1} \oplus \Lp{n_2} \oplus \Reals^{n_3}
    && \subseteq \Sym{n_1} \oplus \Reals^{n_2} \oplus \Reals^{n_3}
    & \eqqcolon \Xbb
    \\
    \shortintertext{and}
    \label{eq:11}
    \Lbb
    & \coloneqq \Psd{m_1} \oplus \Lp{m_2} \oplus \Reals^{m_3}
    && \subseteq \Sym{m_1} \oplus \Reals^{m_2} \oplus \Reals^{m_3}
    & \eqqcolon \Ybb
  \end{alignat}
\end{subequations}
for some integers \(n_1,n_2,n_3,m_1,m_2,m_3 \in \Naturals\), the map
\(\Acal \colon \Xbb \to \Ybb\) is linear, \(b \in \Ybb\), and
\(c \in \Xbb\).
Here, the \emph{dual cone} of a cone~\(C\) in Euclidean space~\(\Ebb\)
is
\(C^* \coloneqq \setst{y \in \Ebb}{\forall x \in C,\,\iprod{y}{x} \geq
  0}\), and each of the notations \(a \succeq_C b\) and
\(b \preceq_C a\), with \(a,b \in \Ebb\), means that \(a-b \in C\).

Note that~\cref{eq:18,eq:21} allow for variables composed of positive
semidefinite matrices, nonnegative vectors, and free scalar variables.
Similarly, \cref{eq:17,eq:11} enables one to require affine functions
of the variables to be positive semidefinite, nonnegative, or
\emph{equal} to zero.
As an example, the SDP~\cref{eq:GW-eps-polar-gaugef} can be cast in
the format~\cref{eq:16} by setting
\begin{gather*}
  \Kbb
  \coloneqq
  \Psd{0} \oplus \Lp{1} \oplus \Sym{V},
  \qquad
  \Lbb
  \coloneqq
  \Psd{V} \oplus \Lp{E} \oplus \Reals^V,
  \\
  \Acal
  \colon \mu \oplus Y \in \Reals^1 \oplus \Sym{V}
  \mapsto
  \paren{Y - \mu\eps I}
  \oplus
  \paren[\big]{\tfrac{1}{4}\Laplacian_G^*(Y)}
  \oplus
  \paren[\big]{\mu\ones-\diag(Y)},
  \\
  b \coloneqq 0 \oplus z \oplus 0,
  \qquad
  c \coloneqq 1 \oplus 0.
\end{gather*}

Some IPMs proceed by producing a sequence of iterates tracking the
so-called central path.
We will encode the procedure that updates an iterate to the next one
by a function~\(\Xi\), so that from an iterate \((x_t,y_t)\) of
primal-dual solutions, the next iterate will be
\(\Xi(x_t,y_t) \eqqcolon (x_{t+1},y_{t+1})\).

The number of iterations of IPMs so that the duality gap for a pair of
solutions \((x_t,y_t)\) is a \(\delta\)-fraction of the inital duality
gap can be bounded by a function \(\psi\) on the initial Slater points
\((x_0,y_0) \coloneqq (\slater{x},\slater{y})\).
Set
\begin{equation}
  \label{eq:sdp-N-def}
  N \coloneqq n_1+n_2+m_1+m_2.
\end{equation}
For each pair \((x,y)\) of primal-dual feasible solutions, where
\begin{alignat*}{4}
  x & \eqqcolon X_1 && \oplus x_2 && \oplus x_3 && \in \Xbb,
  \\*
  y & \eqqcolon Y_1 && \oplus y_2 && \oplus y_3 && \in \Ybb,
  \\
  \shortintertext{and with corresponding slacks}
  \Acal(x) - b & \eqqcolon U_1 && \oplus u_2 && \oplus u_3 &&\in \Ybb,
  \\*
  c - \Acal^*(y) & \eqqcolon V_1 && \oplus v_2 && \oplus v_3 && \in \Xbb,
\end{alignat*}
define
\begin{multline}
  \label{eq:14}
  \psi(x,y) \coloneqq
  N\ln\paren[\bigg]{
    \frac{1}{N}
    \iprod[\Big]{
      X_1 \oplus x_2 \oplus Y_1 \oplus y_2
    }{
      V_1 \oplus v_2 \oplus U_1 \oplus u_2
    }
  }
  \\
  -
  \ln \paren[\bigg]{
    \det(X_1)\det(V_1)\det(Y_1)\det(U_1)
    \paren[\Big]{\prod x_2}
    \paren[\Big]{\prod v_2}
    \paren[\Big]{\prod y_2}
    \paren[\Big]{\prod u_2}
  },
\end{multline}
where, for each vector \(a\), we denote
\(\prod a \coloneqq \det(\Diag(a))\).
The function \(\psi\) takes into account two important factors that
affect the number of required iterations: the first term in the RHS
of~\cref{eq:14} depends on the initial duality gap
\begin{inlinemath}
  \iprod[\big]{
    X_1 \oplus x_2 \oplus Y_1 \oplus y_2
  }{
    V_1 \oplus v_2 \oplus U_1 \oplus u_2
  },
\end{inlinemath}
whereas the second term is related to the centrality of the initial
Slater point.
Indeed, it is intuitive that in IPMs, a good initial point should have
a reasonably good duality gap and at the same time not being too close
to the boundary of the feasible region.

The next result is adapted from~\cite[Theorem~4.5]{Tuncel10a} for our
format of SDPs:
\begin{theorem}
  \label{thm:sdp-central}
  Let \(\delta \in (0,1)\) and let \((x_0,y_0)\) be a primal-dual pair
  of feasible solutions for~\cref{eq:16} and~\cref{eq:19}, respectively,
  such that
  \begin{equation*}
    \psi(x_0,y_0) \leq \sqrt{N} \ln \left(1/\delta\right),
  \end{equation*}
  where \(N\) is defined as in~\cref{eq:sdp-N-def} and~\(\psi\) as
  in~\cref{eq:14}.
  Define the sequence \((x_t,y_t)_{t=0}^{\infty}\) by
  \((x_{t+1},y_{t+1}) \coloneqq \Xi(x_t,y_t)\) for each
  \(t \in \Naturals\).
  Define the sequence \((u_t,v_t)_{t=0}^{\infty}\) by
  \(u_t \coloneqq \Acal(x_t)-b\) and \(v_t \coloneqq c-\Acal^*(y_t)\)
  for each \(t \in \Naturals\).
  Then
  \begin{equation*}
    \iprod{x_t \oplus u_t}{v_t \oplus y_t}
    \leq
    \delta
    \iprod{x_0 \oplus u_0}{v_0 \oplus y_0}
    \qquad
    \text{for each }
    t \geq \bar{t} \coloneqq 24\sqrt{N}\ln(1/\delta).
  \end{equation*}
\end{theorem}

\begin{proposition}
  \label{prop:SDP-polynomial-time}
  Let \(\sigma \in (0,2/3)\).
  There exists a polynomial-time algorithm that, given a graph
  \(G = (V,E)\) with \(n\)~vertices and \(m\)~edges, a nonzero vector
  \(z \in \Lp{E}\), and a number \(\eps \in [0,1/3]\), computes
  feasible solutions \((w^*,x^*)\) for~\cref{eq:GW-eps-polar-supf} and
  \((\mu^*,Y^*)\) for \cref{eq:GW-eps-polar-gaugef} such that
  \(\mu^* \leq \iprodt{z}{w^*} +\sigma \norm[\infty]{z}\).
  The algorithm consists of applying an interior-point method for
  \(T \coloneqq 24\paren{n+m+1}\ln(8/\sigma)\)
  iterations; each iteration encoding one application of the
  function~\(\Xi\) can be made to run in time \(O((n+m)^3)\).
\end{proposition}
\begin{remark}
  Note that, since \(\sigma\) is constant, the
  SDPs~\cref{eq:GW-eps-polar-def} can be nearly solved (in the sense
  of the \(\sigma\norm[\infty]{z}\) additive error) in strongly
  polynomial time.
\end{remark}
\begin{proof}[Proof of \cref{prop:SDP-polynomial-time}]
  First we write \cref{eq:GW-eps-polar-gaugef} and its dual.
  We scale the cost function of \cref{eq:GW-eps-polar-gaugef} by
  \(N \coloneqq n+m+1\), and we normalize the edge weights by setting
  \(\bar{z} \coloneqq z/\norm[\infty]{z}\):
  \begin{equation*}
    \begin{array}{rrl}
      \textrm{(P)} & \inf & N\mu\\[2pt]
      & \text{subject to}&  Y - \mu \eps I \succeq 0,\\[2pt]
      &&  \tfrac{1}{4}\Laplacian_G^*(Y) \geq \bar{z},\\[2pt]
      &&  -\diag(Y) + \mu \ones = 0,\\[2pt]
      && Y \in \Sym{V},\\[2pt]
      && \mu \in \Reals_{+},
    \end{array}
    \quad
    \begin{array}{rrl}
      \textrm{(D)} & \sup & \iprodt{\bar{z}}{w}\\[2pt]
      & \text{subject to}&
      S\in \Psd{V},\\[2pt]
      & & w\in \Reals_+^E,\\[2pt]
      & & x\in \Reals^V,\\[2pt]
      & & S + \tfrac{1}{4}\Laplacian_G(w) - \Diag(x) = 0,\\[2pt]
      & & \iprodt{\ones}{x} - \eps \trace(S) \leq N.
    \end{array}
  \end{equation*}
  Set
  \begin{alignat*}{3}
    & \slater{\mu} \coloneqq 4,
    & \qquad
    & \slater{Y} \coloneqq 4I,
    \\
    & \slater{w}\coloneqq\ones,
    & \qquad
    & \slater{x}\coloneqq \tfrac{1}{2} \deg_G + \ones,
    & \qquad
    & \slater{S} \coloneqq I + \tfrac{1}{4}D_G + \tfrac{1}{4}A_G,
  \end{alignat*}
  where \(\deg_G \colon V \to \Naturals\) is the degree function
  of~\(G\) and \(D_G \coloneqq \Diag(\deg_G)\).
  It is straightforward to check that \(\slater{Y}\oplus\slater{\mu}\)
  and \(\slater{S}\oplus \slater{w}\oplus \slater{x}\) are Slater
  points for (P) and (D), resp., with corresponding slacks
  \begin{align*}
    \slater{U} &\coloneqq \slater{Y} - \slater{\mu} \eps I = 4(1-\eps) I,\\
    \slater{u} &\coloneqq \tfrac{1}{4}\Laplacian_G^*(\slater{Y}) - \bar{z} = 2\ones-\bar{z},\\
    \slater{\nu} &\coloneqq N - \iprodt{\ones}{\slater{x}} + \eps \trace(\slater{S})
    = 1+\tfrac{\eps}{2}m + \eps n.
  \end{align*}
  The duality gap between \(\slater{Y} \oplus \slater{\mu}\) and
  \(\slater{S} \oplus \slater{w} \oplus \slater{x}\) is
  \begin{equation}
    \label{eq:initial-duality-gap}
    \begin{split}
    \slater{\mu} \slater{\nu}
    +
    \iprod{\slater{U}}{\slater{S}}
    +
    \iprodt{\slater{u}}{\slater{w}}
    &=
    4(1+\tfrac{\eps}{2}m + \eps n)
    +
    4(1-\eps)\trace(I + \tfrac{1}{4}D_G + \tfrac{1}{4}A_G)
    +
    2\iprodt{\ones}{\ones}-\iprodt{\ones}{\bar{z}}
    \\
    &\leq
    4(1+\tfrac{\eps}{2}m + \eps n) + 4n + 4m
    \leq
    4(1+m/2+n) + 4n+4m
    \\
    &\leq
    8(1+m+n)
    = 8N.
    \end{split}
  \end{equation}
  We will now compute an upper bound for the value of the function
  \(\psi\) at \(\slater{Y} \oplus \slater{\mu}\) and
  \(\slater{S} \oplus \slater{w} \oplus \slater{x}\).
  Since we already computed an upper bound for the duality gap in
  \cref{eq:initial-duality-gap}, we will now lower bound the
  determinants:
  \begin{multline*}
    \slater{\mu} \cdot \det(\slater{U}) \cdot \paren[\Big]{\prod \slater{u}}
    \cdot
    \slater{\nu} \cdot \det(\slater{S})\cdot \paren[\Big]{\prod \slater{w}}
    \\
    =
    4 \cdot \det\paren[\big]{4(1-\eps)I} \cdot \prod (2\ones-\bar{z})
    \cdot
    (1+\tfrac{\eps}{2}m + \eps n)
    \cdot \det\paren[\big]{I + \tfrac{1}{4}D_G + \tfrac{1}{4}A_G}
    \cdot\prod \ones
    \\
    \geq
    4\det\paren[\big]{4(1-\eps)I} \det(I)
    \geq
    4^n (1-\eps)^n,
  \end{multline*}
  where in the first inequality we used that \(\norm[\infty]{\bar{z}} = 1\)
  and \(I + \tfrac{1}{4}D_G + \tfrac{1}{4}A_G \succeq I\).
  Thus,
  \begin{equation*}
    \psi(\slater{\mu}\oplus \slater{Y}, \slater{S}\oplus
    \slater{w}\oplus  \slater{x})
    \leq
    N\ln \paren[\Big]{\frac{1}{N}(8N)} - \ln(4^n (1-\eps)^n)
    \leq
    N\ln(8).
  \end{equation*}
  Set
  \(\delta \coloneqq \paren{\sigma/8}^{\sqrt{N}} \leq
  \min\set[\big]{\sigma/8, 8^{-\sqrt{N}}}\).
  By \cref{thm:sdp-central}, after \(24\sqrt{N}\ln(1/\delta) = O(N)\)
  iterations of \(\Xi\), the~duality gap is at most
  \(8 \delta N \leq \sigma N\).
  That is, we obtain \(\tilde{Y} \oplus \tilde{\mu}\) and
  \(\tilde{S} \oplus \tilde{w} \oplus \tilde{x}\), feasible for (P)
  and (D), resp., such that
  \(N \tilde{\mu} \leq \iprodt{\bar{z}}{\tilde{w}} + N\sigma\).
  Hence,
  \begin{inlinemath}
    (\mu^*,Y^*)
    \coloneqq
    \paren[\big]{
      \norm[\infty]{z}\tilde{\mu}
      ,
      \norm[\infty]{z}\tilde{Y}
    }
  \end{inlinemath}
  is
  feasible for \cref{eq:GW-eps-polar-gaugef}
  and
  \begin{inlinemath}
    (w^*,x^*)
    \coloneqq
    \paren[\big]{
      \tfrac{1}{N} \tilde{w}
      ,
      \tfrac{1}{N} \tilde{x}
    }
  \end{inlinemath}
  is feasible for \cref{eq:GW-eps-polar-supf},
  and their objective values satisfy
  \begin{equation*}
    \mu^*
    =
    \norm[\infty]{z}\tilde{\mu}
    \leq
    \norm[\infty]{z}\frac{\iprodt{\bar{z}}{\tilde{w}}}{N}
    +
    \sigma\norm[\infty]{z}
    =
    \iprodt{z}{w^*} + \sigma\norm[\infty]{z}.\qedhere
  \end{equation*}
\end{proof}

\end{document}